\documentclass[a4paper,oneside,12pt]{article}
\usepackage[titletoc,toc,title]{appendix}
\usepackage{manfnt} 

\newcommand{\homotopy}[3] 
{
	\xymatrix{{#1} \ar@2{~>}@<0.25ex>[r]^{{#2}} & {#3}}
}

\newcommand{\Ho}{\cc{H}{{o}}(\cc{C},\Sigma)}

\newcommand{\HoC}{\cc{H}{{o}}(\cc{C})}

\newcommand{\Wofc}{\cc{H}{{o}}_{fc}(\cc{C},\cc{W})}

\newcommand{\Wo}{\cc{H}{{o}}(\cc{C},\cc{W})}

\newcommand{\HoA}{\cc{H}{{o}}(\cc{A},\Sigma)}
\usepackage[latin1]{inputenc}
\usepackage{latexsym}
\usepackage{hyperref}
\usepackage[all]{xy}
\xyoption{rotate}
\usepackage{amsmath,amssymb,amsthm,xspace,a4wide,color}
\usepackage{verbatim} 
\usepackage{xcolor}

\numberwithin{equation}{section}

\newlength\Colsep
\theoremstyle{definition}
  \begingroup
        \newtheorem{remark}[equation]{Remark}
        
        \newtheorem{sinnadastandard}[equation]{\textcolor{white}{-}\hspace{-.3cm}}

	\newtheorem{notation}[equation]{Notation}
\endgroup

\theoremstyle{plain}
  \begingroup
        \newtheorem{theorem}[equation]{Theorem}
        \newtheorem{lemma}[equation]{Lemma}
        \newtheorem{proposition}[equation]{Proposition}
        \newtheorem{corollary}[equation]{Corollary}

	    \newtheorem{definition}[equation]{Definition}

\endgroup

\newcommand{\cqd}{\hfill$\Box$}  

\makeatletter
\newcommand{\circlearrow}{}
\DeclareRobustCommand{\circlearrow}{%
  \mathrel{\vphantom{\rightarrow}\mathpalette\circle@arrow\relax}%
}
\newcommand{\circle@arrow}[2]{%
  \m@th
  \ooalign{%
    \hidewidth$#1\circ\mkern1mu$\hidewidth\cr
    $#1\longrightarrow$\cr}%
}
\makeatother

\makeatletter
\newcommand{\xRightarrow}[2][]{\ext@arrow 0359\Rightarrowfill@{#1}{#2}}
\newcommand{\xLeftarrow}[2][]{\ext@arrow 0359\Leftarrowfill@{#1}{#2}}
\makeatother

\newcommand{\mr}[1]{\stackrel{#1}{\longrightarrow}}

\newcommand{\ml}[1]{\stackrel{#1}{\longleftarrow}}

\newcommand{\meq}[1]{\stackrel{#1}{=}}

\newcommand{\mrr}[2] 
 {
  \xymatrix@C=4.5ex@R=2.4ex
         {
          {} \ar@<1.2ex>[r]^{#1}
             \ar@<-1.1ex>[r]^{#2}
          & {}
         }
 }

\newcommand{\mrl}[2] 
  {
  \xymatrix@C=4.5ex@R=2.4ex
         {
          {} \ar@<1.2ex>[r]^{#1}             
          & {} \ar@<1.1ex>[l]_{#2}
         }
 }

\newcommand{\Mr}[1]{\stackrel{#1}{\Rightarrow}}
\newcommand{\Mrl}[1]{\stackrel{#1}{\Longrightarrow}}
\newcommand{\Ml}[1]{\stackrel{#1}{\Leftarrow}}

\newcommand{\xr}[1]{\xrightarrow{#1}}
\newcommand{\Xr}[1]{\xRightarrow{#1}}

\newcommand{\mvr}[1]{\xymatrix{{} \ar@{~>}[r]^{#1} & {}}}

\newcommand{\cc}[1]{\mathcal{#1}}
\newcommand{\C}{\mathcal{C}}

\newcommand{\eps}{\varepsilon}

\newcommand{\ff}{\mathsf}
\newcommand{\nn}{\mathbf}

\newcommand{\comw}{\textcolor{white}}

\newcommand{\adj}[2]
{
\ar@/^1ex/[r]^{{#1}}
\ar@{}[r]|{\bot}
\ar@/_1ex/@{<-}[r]_{{#2}} 
}

\newcommand{\jda}[2]{
\ar@/^1ex/[r]^{#1}
\ar@{}[r]|{\top}
\ar@/_1ex/@{<-}[r]_{#2} }

\newcommand{\adjbis}[2]{
\ar@/^2ex/[r]^{{#1}}
\ar@{}[r]|{\bot}
\ar@/_2ex/@{<-}[r]_{{#2}} }

\newcommand{\jdabis}[2]{
\ar@/^2ex/[r]^{#1}
\ar@{}[r]|{\top}
\ar@/_2ex/@{<-}[r]_{#2} }

\newcommand{\idavuelta}[2]{
\ar@/^1ex/[r]^{#1}
\ar@/_1ex/@{<-}[r]_{#2} }

\newcommand{\hpy}
           {
            \hspace{-1ex}
            \xymatrix{ {} \ar@2{~>}@<0.25ex>[r] & {} }
            \hspace{-0.6ex}
            }
\newcommand{\mrhpy}[1]
            {\stackrel{#1}{\hpy}}

\newcommand \germ{\underset{g}{\thicksim}}

\newcommand{\dcell}[1]  
          {
		   \ar@<8pt>@{-}[d]+<-4pt,8pt>
           \ar@<-8pt>@{-}[d]+<4pt,8pt>
           \ar@{}[d]|{#1}
          }

\newcommand{\dcellb}[1]   
          {
           \ar@<10pt>@{-}[d]+<-5pt,8pt>
           \ar@<-10pt>@{-}[d]+<5pt,8pt>
           \ar@{}[d]|{#1}
          }
\newcommand{\dcellbbis}[1]   
{
	\ar@<15pt>@{-}[d]+<-5pt,8pt>
	\ar@<-15pt>@{-}[d]+<5pt,8pt>
	\ar@{}[d]|{#1}
}
\newcommand{\dcellbymedio}[1]   
          {
           \ar@<15pt>@{-}[d]+<-7.5pt,10pt>
           \ar@<-15pt>@{-}[d]+<7.5pt,10pt>
           \ar@{}[d]|{#1}          
          }
\newcommand{\dcellbis}[1]   
          {
           \ar@<15pt>@{-}[d]+<-7.5pt,0pt>
           \ar@<-15pt>@{-}[d]+<7.5pt,0pt>
           \ar@{}[d]|{#1}          
          }          
\newcommand{\dcellbymediobis}[1]   
          {
           \ar@<15pt>@{-}[d]+<-7.5pt,-5pt>
           \ar@<-15pt>@{-}[d]+<7.5pt,0pt>
           \ar@{}[d]|{#1}          
          }
\newcommand{\deq}        
         {
          \ar@{=}[d]
         }

\newcommand{\ddeq}{\ar@{=}[dd]}
\newcommand{\dddeq}{\ar@{=}[ddd]}         
\newcommand{\ddddeq}{\ar@{=}[dddd]}         
\newcommand{\dddddeq}{\ar@{=}[ddddd]}                  
\newcommand{\ddddddeq}{\ar@{=}[dddddd]}                           
\newcommand{\dddddddeq}{\ar@{=}[ddddddd]}         
\newcommand{\ddddddddeq}{\ar@{=}[dddddddd]}         
\newcommand{\dddddddddeq}{\ar@{=}[ddddddddd]}                  
\newcommand{\ddddddddddeq}{\ar@{=}[dddddddddd]}                  
\newcommand{\dddddddddddeq}{\ar@{=}[ddddddddddd]}                  

\newcommand{\dreq}       
         {
          \ar@{=}[dr]
         }

\newcommand{\dleq}       
         {
          \ar@{=}[dl]
         }

\newcommand{\dccell}[1]    
          {
           \ar@{-}[ld]
           \ar@{-}[rd]
           \ar@{}[d]|{#1}
          }

\newcommand{\dcellbb}[1]   
          {
           \ar@<20pt>@{-}[d]+<-10pt,12pt>
           \ar@<-20pt>@{-}[d]+<10pt,12pt>
           \ar@{}[d]|{#1}
          }

\newcommand{\dl}    
          {
           \ar@<-2pt>@{-}[d]+<4pt,8pt>
          }

\newcommand{\dr}    
          {
           \ar@<2pt>@{-}[d]+<-4pt,8pt>
          }
\newcommand{\drbis}    
          {
           \ar@<-2pt>@{-}[d]+<-4pt,8pt>
          }
\newcommand{\drmediobis}    
          {
           \ar@<-1pt>@{-}[d]+<-4pt,8pt>
          }
\newcommand{\dc}[1]    
          {
           \ar@{}[d]|{#1}
          }
\newcommand{\dcr}[1]    
          {
           \ar@{}[dr]|{#1}
          }
\newcommand{\dcl}[1]    
          {
           \ar@{}[dl]|{#1}
          }

\newcommand{\uccell}[1]      
          {
           \ar@{-}[ur]
           \ar@{}[u]|{#1}
           \ar@{-}[ul]
          }

\newcommand{\uccellb}[1]     
          {
           \ar@<-1ex>@{-}[ur]
           \ar@{}[u]|{#1}
           \ar@<1ex>@{-}[ul]
          }

\newcommand{\dcellopa}[1]  
          {
		   \ar@<5pt>@{-}[d]+<5pt,8pt>
           \ar@<-5pt>@{-}[d]+<-5pt,8pt>
           \ar@{}[d]|{#1}
          }

\newcommand{\dcellop}[1]  
          {
		   \ar@<6pt>@{-}[d]+<6pt,8pt>
           \ar@<-6pt>@{-}[d]+<-6pt,8pt>
           \ar@{}[d]|{#1}
          }

\newcommand{\dcellopb}[1]  
          {
					 \ar@<7pt>@{-}[d]+<7pt,8pt>
           \ar@<-7pt>@{-}[d]+<-7pt,8pt>
           \ar@{}[d]|{#1}
          }

\newcommand{\dcellopbb}[1]  
          {
					 \ar@<8pt>@{-}[d]+<8pt,8pt>
           \ar@<-8pt>@{-}[d]+<-8pt,8pt>
           \ar@{}[d]|{#1}
          }
          
\newcommand{\did}{\ar@2{-}[d]}
\newcommand{\ddid}{\ar@2{-}[dd]}

\newcommand{\dig}{ \ar@2{-}[d] & & }

\newcommand{\op}[1]
          {
           \ar@{-}[ld]
           \ar@{-}[rd]
           \ar@{}[d]|{#1}
          }

\newcommand{\opbis}[1]
          {
           \ar@{-}[ld]
           \ar@{-}[rd]
           \ar@{}[d]|>>>>{#1}
          }

\newcommand{\opb}[1]
          {
           \ar@<-2pt>@{-}[ld]
           \ar@<2pt>@{-}[rd]
           \ar@{}[d]|{#1}
          }        
\newcommand{\opmediob}[1]
          {
           \ar@<-1pt>@{-}[ld]
           \ar@<1pt>@{-}[rd]
           \ar@{}[d]|{#1}
          }  
          
\newcommand{\opbymedio}[1]
          {
           \ar@<-3pt>@{-}[ld]
           \ar@<3pt>@{-}[rd]
           \ar@{}[d]|{#1}
          }     
\newcommand{\opbb}[1]
          {
           \ar@<-4pt>@{-}[ld]
           \ar@<4pt>@{-}[rd]
           \ar@{}[d]|{#1}
          }               
\newcommand{\opbbb}[1]
          {
           \ar@<-6pt>@{-}[ld]
           \ar@<6pt>@{-}[rd]
           \ar@{}[d]|>>{#1}
          } 
\newcommand{\opbbbbis}[1]
          {
           \ar@<-6pt>@{-}[ld]
           \ar@<6pt>@{-}[rd]
           \ar@{}[d]|<<<<<{#1}
          } 
\newcommand{\opunodos}[1]
          {
           \ar@{-}[ld]
           \ar@{-}[rrd]
           \ar@{}[dr]|{#1}
          }
          
\newcommand{\opunodosb}[1]
          {
           \ar@<-2pt>@{-}[ld]
           \ar@<2pt>@{-}[rrd]
           \ar@{}[dr]|{#1}
          }

\newcommand{\opdosuno}[1]
          {
           \ar@{-}[lld]
           \ar@{-}[rd]
           \ar@{}[d]|{#1}
          }          
          
\newcommand{\opdosdos}[1]
          {
           \ar@{-}[lld]
           \ar@{-}[rrd]
           \ar@{}[d]|{#1}
          }    
\newcommand{\opdostres}[1]
          {
           \ar@{-}[lld]
           \ar@{-}[rrrd]
           \ar@{}[d]|{#1}
          }    
\newcommand{\optresuno}[1]
          {
           \ar@{-}[llld]
           \ar@{-}[rd]
           \ar@{}[d]|{#1}
          }   
\newcommand{\optresdos}[1]
          {
           \ar@{-}[llld]
           \ar@{-}[rrd]
           \ar@{}[d]|{#1}
          }    

\newcommand{\optrestres}[1]
          {
           \ar@{-}[llld]
           \ar@{-}[rrrd]
           \ar@{}[d]|{#1}
          }            
\newcommand{\opcincocinco}[1]
          {
           \ar@{-}[llllld]
           \ar@{-}[rrrrrd]
           \ar@{}[d]|{#1}
          }                      
          
\newcommand{\clputo}[1]
          {
           \ar@{-}[ur]
           \ar@{}[u]|{#1}
           \ar@{-}[ul]
          }

\newcommand{\clold}[1]
          {
           \ar@{-}[ur]
           \ar@{}[u]|{#1}
           \ar@{-}[ul]
          }

\newcommand{\clb}[1]
          {
           \ar@<-1ex>@{-}[ur]
           \ar@{}[u]|{#1}
           \ar@<1ex>@{-}[ul]
          }
\newcommand{\clbb}[1]
          {
           \ar@<-2ex>@{-}[ur]
           \ar@{}[u]|{#1}
           \ar@<2ex>@{-}[ul]
          }
\newcommand{\clmediob}[1]
          {
           \ar@<-.5ex>@{-}[ur]
           \ar@{}[u]|{#1}
           \ar@<.5ex>@{-}[ul]
          } 
\newcommand{\clrightb}[1]
          {
           \ar@<-1ex>@{-}[ur]
           \ar@{}[u]|{#1}
           \ar@{-}[ul]
          }
\newcommand{\clunodos}[1]
          {
           \ar@{-}[urr]
           \ar@{}[u]|{#1}
           \ar@{-}[ul]
          }
          
\newcommand{\cldosuno}[1]
          {
           \ar@{-}[ur]
           \ar@{}[u]|{#1}
           \ar@{-}[ull]
          }
\newcommand{\cldosdos}[1]
          {
           \ar@{-}[urr]
           \ar@{}[u]|{#1}
           \ar@{-}[ull]
          }         
\newcommand{\cltresdos}[1]
          {
           \ar@{-}[urr]
           \ar@{}[u]|{#1}
           \ar@{-}[ulll]
          }           
\newcommand{\cldostres}[1]
          {
           \ar@{-}[urrr]
           \ar@{}[u]|{#1}
           \ar@{-}[ull]
          }     
\newcommand{\cltrestres}[1]
          {
           \ar@{-}[urrr]
           \ar@{}[u]|{#1}
           \ar@{-}[ulll]
          }            
\newcommand{\clcincocinco}[1]
          {
           \ar@{-}[urrrrr]
           \ar@{}[u]|{#1}
           \ar@{-}[ulllll]
          }
           
\newcommand{\ardr}{\ar@{-}[dr]}
\newcommand{\ardrr}{\ar@{-}[drr]}
\newcommand{\ardrrr}{\ar@{-}[drrr]}
\newcommand{\ardrrrr}{\ar@{-}[drrrr]}
\newcommand{\ardl}{\ar@{-}[dl]}
\newcommand{\ardll}{\ar@{-}[dll]}
\newcommand{\ardlll}{\ar@{-}[dlll]}
\newcommand{\ardllll}{\ar@{-}[dllll]}
\newcommand{\ardlllll}{\ar@{-}[dlllll]}

\newcommand{\cellrdE}[3] 
{\xymatrix@C=7ex@R=2.4ex
        {
		\ar@<1.9ex>[r]^{#1} 
		\ar@{}@<-1.3ex>[r]^{\!\! {#2} \, \!\Downarrow}
		\ar@<-1.1ex>[r]_{#3} & 
		}
}
\newcommand{\cellrdEcorta}[3] 
{\xymatrix@C=5ex@R=2.4ex
       {
		\ar@<1.6ex>[r]^{#1} 
		\ar@{}@<-1.3ex>[r]^{\!\! {#2} \, \!\Downarrow}
		\ar@<-1.1ex>[r]_{#3} & 
	   }
}
\newcommand{\cellrd}[3] 
 {
  \xymatrix@C=5ex@R=2.4ex
         {
          {} \ar@<1.1ex>[r]^{#1}
             \ar@{}@<-1.3ex>[r]^{\Downarrow \; {#2}}
             \ar@<-1.2ex>[r]_{#3}
          & {}
         }
}
\newcommand{\cellrdb}[3] 
 {
  \xymatrix@C=7ex@R=2.4ex
         {
          {} \ar@<1.4ex>[r]^{#1}
             \ar@{}@<-1.3ex>[r]^{\Downarrow \; {#2}}
             \ar@<-1.1ex>[r]_{#3}
          & {}
         }         
 }
 \newcommand{\scellrd}[3] 
 {
  \xymatrix@C=4.5ex@R=2.4ex
         {
          {} \ar@<1.4ex>[r]^{#1}
             \ar@{}@<-1.3ex>[r]^{\!\! \Downarrow \, {#2}}
             \ar@<-1.1ex>[r]_{#3}
          & {}
         }
}
          
 \newcommand{\modif}[3] 
 {
  \xymatrix@C=7ex@R=2.4ex
         {
          {} \ar@<1.6ex>@{=>}[r]^{#1}
             \ar@{}@<-1.3ex>@{=>}[r]^{\!\! {#2} \, \!\downarrow}
             \ar@{}@<-1.1ex>[r]_{#3}
          & {}
         }
 }

\newcommand{\cellld}[3] 
 {
  \xymatrix@C=6ex@R=2.4ex
         {
            {}
          & {} \ar@<1.0ex>[l]^{#3}
          \ar@{}@<-1.7ex>[l]^{\!\! {#2} \, \!\Downarrow}
	                                 \ar@<-1.7ex>[l]_{#1}
         }
 }

\newcommand{\cellpairrd}[4] 
 {
  \xymatrix@C=8ex@R=2.2ex
         {
          {} \ar@<1.6ex>[r]^{#1}
             \ar@{}@<-1.3ex>[r]^{\!\! \Downarrow \, {#2} 
                                 \;\;\; \Downarrow \, {#3} }
             \ar@<-1.1ex>[r]_{#4}
          & {}
         }
 }

\usepackage[textwidth=0.89in,textsize=scriptsize]{todonotes}

\newcommand{\clcputo}{\dcell} 
\newcommand{\clcc}{\dcellb}

\newcommand{\cla}{\dcellopa}

\newcommand{\eqq}{\deq}    
\newcommand{\dla} 
          {
           \ar@<1pt>@{-}[d]+<4pt,8pt>
          }
\newcommand{\dlcopia} 
          {
           \ar@<-2pt>@{-}[d]+<4pt,8pt>
          }
\newcommand{\drcopia}  
          {
           \ar@<2pt>@{-}[d]+<-4pt,8pt>
          }
\newcommand{\dll} 
          {
           \ar@<-3pt>@{-}[d]+<4pt,8pt>
          }
\newcommand{\drr} 
          {
           \ar@<3pt>@{-}[d]+<-4pt,8pt>
          }
\newcommand{\dccopia}[1]    
          {
           \ar@{}[d]|{#1}
          }
\newcommand{\dcrcopia}[1]    
          {
           \ar@{}[dr]|{#1}
          }
\newcommand{\dclcopia}[1]    
          {
           \ar@{}[dl]|{#1}
          }
\newcommand{\dcla}[1]    
          {
           \ar@<-2pt>@{}[d]|{#1}
          }

\newcommand{\vc}{\vcenter}

\newcommand{\eq}[1]   
         {
          \ar@<#1pt>@{=}[d]
         } 
 
\newcommand{\sd}[2]  
          {
           \ar@<#1pt>@{-}[d]+<#2pt,8pt>
          }

\newcommand{\na}[2]  
          {
           \ar@<#1pt>@{}[d]|{#2}
          }

\newcommand{\cl}[2]            {
		   \ar@<#1pt>@{-}[d]+<-4pt,8pt>
           \ar@<-#1pt>@{-}[d]+<4pt,8pt>
           \ar@{}[d]|{#2}
          }
\newcommand{\hs}[1]{\hspace{#1pt}}

\newcommand{\gi}{id}

\begin{document}

\title{Model bicategories and their homotopy bicategories}

\author{Descotte M.E., Dubuc E.J., Szyld M.}

\date{\vspace{-5ex}}

\maketitle

\begin{abstract}
	We give the definitions of model bicategory and $q$-homotopy, which are natural generalizations of the notions of model category and homotopy to the context of bicategories. 
For any model bicategory $\mathcal{C}$, denote by $\mathcal{C}_{fc}$ the full sub-bicategory of the fibrant-cofibrant objects. We prove that the 2-dimensional localization of $\mathcal{C}$ at the weak equivalences can be computed as a bicategory $\mathcal{H}{{o}}(\mathcal{C})$ whose objects and arrows are those of $\mathcal{C}_{fc}$ and whose 2-cells are classes of $q$-homotopies up to an equivalence relation. {When considered for a model category, $q$-homotopies coincide with the homotopies as considered by Quillen.} 
The pseudofunctor $\mathcal{C} \stackrel{q}{\longrightarrow} \mathcal{H}{{o}}(\mathcal{C})$ which yields the localization is constructed by using a notion of fibrant-cofibrant replacement in this context.
{We include an appendix with a general result of independent interest on a transfer of structure for lax functors, that we apply to obtain a pseudofunctor structure for the fibrant-cofibrant replacement.}      
\end{abstract}

\section{Introduction}\label{intro} 

{
The notion of model category, originally introduced in \cite{Quillen}, is the basis for a great deal of modern homotopy theory. 
A basic feature of a model category is that it allows the construction of its  localization at a class of arrows, the \emph{weak equivalences}, as a quotient by the congruence determined by \emph{homotopies} between arrows. 
The lifting properties relating the three distinguished families of arrows of a model category are at the very core of the theory. Also, on its own, lifting properties are important in category and higher category theory, perhaps most notably in the definition of quasi-categories \cite{JoyalSS}, or $\infty$-categories \cite{Lurie}.}

On the other hand, basic 2-dimensional categorical structures such as 2-categories and bicategories \cite{Ben} have extensively been used since the '60s.
{In this context, given a category, it is natural to explore the existence of a higher structure which collapses into the category when we identify the arrows having a 2-cell between them.
One may think of these \mbox{2-cells} as being {\em hidden} behind the equality of arrows of the category, and making them explicit allows to study their structure.}
{It is then natural to think of Quillen's homotopies as arrows between morphisms rather than as generators of an equivalence relation. This leads to consider a notion of model 2-category or model bicategory, where the diagrams in the axioms defining a model category are required to commute up to invertible 2-cells.}
As far as we know, this idea first appeared for 2-categories in print as an open question in \mbox{\cite[Problem 8.1]{Hovey}.} 
A set of axioms for a \mbox{2-category} $\cc{C}$ to be a model 2-category is given in \cite{TesisEmi}, and more recently equivalent axioms are considered in \mbox{\cite[Chapter 7]{TesisBarton}} 
(for further information regarding their motivation for considering such a notion, we refer the reader to \cite[\S 8]{Hovey} and to \cite[Introduction]{TesisBarton}).
We note explicitly that these axioms deal only with the underlying $(2,1)$-category structure of $\cc{C}$; that is, they involve only its invertible 2-cells. 

\smallskip

{\em 
The main contributions of this article are a generalization to the context of bicategories of the concepts of lifting property and model category, 
which involve the non-invertible \mbox{2-cells} of the bicategory, 
and a construction of the homotopy bicategory of a model bicategory, that is its localization at the weak equivalences.}

\smallskip

As in dimension 1, a model bicategory consists of a bicategory together with three families of arrows, namely fibrations, cofibrations and weak equivalences, satisfying a set of axioms.  
The axioms we give are a natural generalization to bicategories of those given by Quillen in the sense that they are obtained by requiring the diagrams to commute up to invertible 2-cells, and by considering a 2-dimensional aspect of the lifting properties which relate these families of arrows. In particular, when we consider a category and three families of arrows as a trivial bicategory with the same families of arrows, the two notions coincide: it will be a model bicategory if and only if it is a model category.

Note that in the case of a model category $\cc{C}$, the resulting  homotopy bicategory is a 2-category that has the same arrows than $\cc{C}$, but now it also has 2-cells determined by the usual homotopies, which are not collapsed into equalities \cite{JaG}.
The homotopy 2-category carries in this way a richer structure than the homotopy category, which could be used in particular to compute finer homotopical invariants, and we describe such a potential application in the last paragraph of this introduction. Further applications of model bicategories that we are currently exploring involve studying the higher structure which could be {\em hidden} in the many examples of model categories, such as \cite{Lack2Mod}, \cite{JoyalSS}.

 {
It is {\em tempting} to conjecture that model bicategories can be related to $(\infty,2)$-categories in the sense that model categories are related to $(\infty,1)$-categories, but this is a too vast question that exceeds the scope of our work here, and for which we do not have any concrete proposal at the moment. 
There is however a situation related to our 2-localization $\cc{C} \mr{} \cc{H}o(\cc{C})$ of a model category $\cc{C}$ considered as a discrete model bicategory which would be interesting to investigate in more detail. Let $\cc{C} \mr{} \cc{L}^H\cc{C}$ be the $(\infty,1)$-category modeled by the hammock (or simplicial) localization defined in \cite{DK1}, \cite{DK2}. Since  
$\pi_0 \cc{L}^H\cc{C}$ and  $\pi_0 \cc{H}o(\cc{C})$ are both equivalent to the homotopy category $\nn{H}o(\cc{C})$, there should be a simplicial functor $\cc{L}^H\cc{C} \mr{} \cc{H}o(\cc{C})$ 
\footnote{2-categories can be seen as simplicial categories by identifying their hom categories with their nerves.} 
which would determine a localization of the $(\infty,1)$-category richer than the homotopy category, and which does not collapse homotopies between arrows, as shown in the diagram:
$$\xymatrix
      {
       \cc{C}          \ar[r] \ar[rd] 
     & \cc{L}^H\cc{C}  \ar[d] \ar@/^4ex/[dd]^{\pi_0}
    \\
     & \cc{H}o(\cc{C}) \ar[d]^{\pi_0}
    \\
       {} & \nn{H}o(\cc{C})
      } 
$$ 

We note that a notion of model $(\infty,1)$-category has been defined \cite[\S 1]{MGa}, and that a localization (in the sense of 
$(\infty,1)$-category theory) of such an $(\infty,1)$-category at the weak equivalences, that keeps track of the higher homotopies, has recently been constructed \mbox{\cite[\S 1]{MGb},} but as the author states there these homotopies are no longer given by cylinders and path objects like the 2-cells of $\mathcal{H}{{o}}(\mathcal{C})$ are. It could also be interesting to compare further both constructions.
}

\vspace{1ex}

We describe now precisely how the axioms for a model bicategory that we introduce here differ from the ones considered in \cite{TesisEmi} and \cite{TesisBarton}.
No significant work is required in order to consider these same axioms for a bicategory instead of a 2-category. In particular, regarding the {\em lifting properties} relating the families of arrows, it is natural to consider the following situation (we refer to Section \ref{sec:modelbicat} for details):

Let $\cc{C}$ be a bicategory and $A \mr{i} X$, $Y \mr{p} B$ be two morphisms in $\cc{C}$. They determine the following diagram
$$
\xymatrix@R=2ex
    {
     \cc{C}(X,\,Y)  \ar@/^5ex/[drr]^{i^\ast}
                   \ar[dr]^(.65){h} 
                   \ar@/_4ex/[rddd]^{p_\ast}  \\
   & \nn{P} \ar[dd]_{\pi_2} 
            \ar[r]^(.4){\pi_1}
            \ar@{}[rdd]|{\Downarrow \, \cong}
   & \cc{C}(A,\,Y) \ar[dd]^{p_\ast} 
   \\ \\
   {} 
   & \cc{C}(X,\,B) \ar[r]^{i^\ast} 
   & \cc{C}(A,\,B)
    }
$$
where $\nn{P}$ is the pseudo pullback of categories. 
In \cite{TesisEmi} and \cite{TesisBarton}, the lifting property is taken to mean that $h$ is essentially surjective (see for example \cite[Prop. 7.1.12]{TesisBarton}).

\smallskip

{\em 
The only substantial modification we make to the aforementioned set of axioms, and that deals with non-invertible 2-cells of $\cc{C}$, is that we require the functor $h$ to be also essentially full, see Definition \ref{essfull}.}
 Note that when all the 2-cells of $\cc{C}$ are invertible, if $h$ is essentially surjective then it is automatically essentially full. 

Notably, this single change suffices for our purpose of constructing the {\em homotopy bicategory of $\cc{C}$}, that is its 2-dimensional localization at the weak equivalences.
We also show that this homotopy bicategory is locally small (resp. a 2-category) when $\cc{C}$ is so.

\smallskip

We give now a description of the main results in this paper. 
Given a model bicategory $\C$, denote by $\C_{fc}$ the full sub-bicategory given by the fibrant-cofibrant objects of $\C$. 
We construct a bicategory $\HoC$ whose objects and arrows are those of $\C_{fc}$ and whose 2-cells are the classes of $q$-homotopies (a 2-dimensional version of Quillen's notion of homotopy) up to an equivalence relation. 
We prove that the 2-dimensional localization of $\C$ with respect to the class $\cc{W}$ of weak equivalences can be computed as $\HoC$. 
More specifically, our main theorem (Theorem~\ref{teo:localizationformodel}) asserts that there is a pseudofunctor $\C \mr{q} \HoC$ which sends weak equivalences to equivalences and has the following universal property:  

\begin{equation} \label{eq:uploc}
Hom(\HoC,\cc{D}) \mr{q^*} Hom_{\cc{W},\Theta}(\C,\cc{D})
\end{equation}

\noindent is a biequivalence of bicategories for every bicategory $\cc{D}$; where $Hom_{\cc{W},\Theta}(\C,\cc{D})$ stands for the full sub-bicategory of $Hom(\C,\cc{D})$ given by those pseudofunctors that send weak equivalences into equivalences. 
{We also show that the homotopy bicategory $\HoC$ is locally small (resp. a 2-category) when $\cc{C}$ is so.}

{We define $q$-homotopies as $w$-homotopies which satisfy some extra conditions related to fibrations and cofibrations. 
The notion of $w$-homotopy was introduced in \cite{DDS.loc_via_homot} and is taken with respect  only to the class $\cc{W}$}. {Note that in op. cit. $w$-homotopies are called homotopies, but  here we chose to use the prefix $w$- to explicitly distinguish them from $q$-homotopies.}
Also note that these are \emph{left} homotopies, and that a dual development with \emph{right} homotopies is possible since the class $\cc{W}$ is assumed with no other properties than the self-dual \emph{3 for 2} axiom.
This definition allows to use previous results of \cite{DDS.loc_via_homot}, where for an arbitrary pair $(\cc{A},\cc{W})$ given by a family of arrows of a bicategory we construct a bicategory ${\cc{H}{{o}}(\cc{A},\cc{W})}$ whose objects and arrows are those of $\cc{A}$ and whose 2-cells are constructed with $w$-homotopies. 
In this general case there is no vertical composition of $w$-homotopies, and thus the 2-cells of ${\cc{H}{{o}}(\cc{A},\cc{W})}$ are given by classes of finite sequences of $w$-homotopies by an appropriate equivalence relation. 
Associated with ${\cc{H}{{o}}(\cc{A},\cc{W})}$ there is a {\em canonical} \mbox{2-functor} \mbox{$\cc{A} \mr{i} {\cc{H}{{o}}(\cc{A},\cc{W})}$.}

A natural way in which a proof of Theorem \ref{teo:localizationformodel} can be attempted is to consider Quillen's proof of the localization theorem for model categories in \cite{Quillen}, and to try to generalize it to this context using $q$-homotopies. In doing so, the first obstacle one faces is that there is no canonical way to compose (whisker) 
arbitrary $q$-homotopies with arrows. This is already so for the 1-dimensional case in \cite{Quillen}, as is well known: for arbitrary arrows $X \mrr{f}{g} Y$ in a model category, 
$f \stackrel{\ell}{\sim} g$ doesn't imply $uf \stackrel{\ell}{\sim} ug$ (where $\stackrel{\ell}{\sim}$ is the left homotopy relation). The other composition poses no problem, $fu \stackrel{\ell}{\sim} gu$ is easily seen to hold, as well as $uf \stackrel{r}{\sim} ug$ for a right homotopy  $f \stackrel{r}{\sim} g$. The way in which this obstacle is avoided in \cite{Quillen} is by considering $X$ to be a cofibrant object, and $Y$ fibrant, so that there is a left homotopy between two arrows if and only if there is a right homotopy.  An important feature of (left) $w$-homotopies is that both compositions (whiskerings) have a trivial canonical definition, and so this obstacle disappears.

\smallskip

{\em There is a rich interplay between $w$- and $q$-homotopies, and that was the key we have found to overcoming the obstacle described above {for the $q$-homotopies} and {furthermore} proving Theorem~\ref{teo:localizationformodel}.}
On the one hand, $w$-homotopies allow us to define a homotopy bicategory.
On the other hand, $q$-homotopies satisfy several results analogous to the ones that can be found in \cite{Quillen}, that we establish and use in this article. 

\smallskip

More explicitly, to prove Theorem~\ref{teo:localizationformodel}  we will use the following results:

\smallskip

\noindent {\bf A.} $\HoC$ inherits the bicategory structure of 
$\Wo$ constructed in \cite{DDS.loc_via_homot}, in such a way that \mbox{$\HoC = \Wofc {\,\subset \, \Wo}$}, that is the full sub-bicategory of $\Wo$ given by the fibrant-cofibrant objects. 
We will deduce this from the following three facts:

\smallskip

\noindent A1. $q$-homotopies can be composed vertically: for any pair of composable $q$-homotopies there is a single $q$-homotopy representing the composition.

\noindent A2.  
For any $w$-homotopy between arrows with fibrant codomain there is a fibrant $w$-homotopy in the same class. For any fibrant $w$-homotopy (between arbitrary arrows) there is a fibrant $q$-homotopy in the same class (we say that a homotopy is fibrant if its cylinder is so; this is made precise in Definition \ref{def:whomot}). Dual results hold for right homotopies (when the domain is a cofibrant object).

\noindent A3.  
For any left $q$-homotopy with $X$ a cofibrant object, there exists a right $q$-homotopy in the same class, which can be taken with any chosen path object. Dually, when $Y$ is a fibrant object, any right $q$-homotopy can be replaced by a left one.

\vspace{1ex}

In what follows both left and right homotopies are involved, so for clarity we will indicate with a superscript ``$\ell$" or ``$r$" which are the \mbox{$2$-cells} of the homotopy bicategory being considered.
Working in the dual model bicategory yields a bicategory 
$\cc{H}o(\cc{C}, \cc{W})^r$ where the 2-cells are given by the right homotopies, and which does not coincide with $\cc{H}o(\cc{C}, \cc{W})^\ell$. But by {A2.} and {A3.} we have that, when X is cofibrant and Y is
fibrant, all possible notions of homotopy (right or left, $w$ or
$q$, fibrant or cofibrant) coincide in the following sense: for any two
choices of these concepts, for any homotopy of that kind there is one of
the other kind in the same class. In particular, 
$\cc{H}o_{fc}(\cc{C}, \cc{W})^\ell = \cc{H}o_{fc}(\cc{C}, \cc{W})^r$,
thus $\cc{H}o(\cc{C})$
could be defined using any of these concepts and this would lead to the same homotopy bicategory.

\smallskip

\noindent {\bf B.} There is a fibrant-cofibrant replacement for model bicategories, that is an assignment from $\C$ to $\C$, defined on objects and arrows (but not on 2-cells)
\mbox
{
$X \mrr{f}{g} Y
\leadsto
RQX \mrr{RQf}{RQg} RQY$,
}
 satisfying the expected conditions for a replacement: 
{all the objects $QX$ are cofibrant, all the objects $RQX$ are fibrant-cofibrant, and} 
there are trivial fibrations $QX \mr{\rho_X} X$, trivial cofibrations $QX \xr{\lambda_{QX}} RQX$, and invertible 2-cells

$$
\vcenter{\xymatrix@C=3pc{ RQX \ar@{<-}[r]^{\lambda_{QX}} \ar[d]_{RQf} \ar@{}[dr]|{\Downarrow \lambda_{Qf}} & 
QX \ar@{}[dr]|{\Downarrow \rho_f} \ar[d]_{Qf} \ar[r]^{\rho_X} & X \ar[d]^f \\
RQY \ar@{<-}[r]_{\lambda_{QY}} &  QY \ar[r]_{\rho_Y} & Y }} 
$$

Note that {$Q$ and $R$ are just assignments  {on objects and arrows} and not necessarily pseudofunctors}, because we are not assuming that there is a (pseudo)functorial factorization, just as in Quillen's original axioms in \cite{Quillen}. Nevertheless, we will show that there are unique structural 2-cells (that is, left homotopies) which make
$Q$ into a pseudofunctor \mbox{$\cc{H}o(\cc{C}, \cc{W})^\ell \mr{Q} \cc{H}o(\cc{C}, \cc{W})^\ell$} and, dually, right homotopies which make $R$ into a pseudofunctor $\cc{H}o(\cc{C}, \cc{W})^r \mr{R} \cc{H}o(\cc{C}, \cc{W})^r$ in such a way that $\rho$, $\lambda$ become  pseudonatural transformations $Q \Mr{\rho} id$, $id \Mr{\lambda} R$. We do this by means of a {\em transfer of pseudofunctor structure}, a novel technique introduced in this paper. A natural context in which to state and prove the transfer results is the more general one of lax functors between bicategories, and this is done in Appendix \ref{sec:transfer}. 

\vspace{1ex}

\noindent{\bf C.} 
The switch between left and right homotopies ({A3.} above) allows to compose the replacement pseudofunctors $Q$ and $R$ as follows. 
The results in {A3.} define 
a {\em switch} \mbox{pseudofunctor} 
$\xymatrix{
\cc{H}o^{f}_c(\cc{C},\cc{W})^\ell     
                 \ar[r]^s 
  & \cc{H}o^{c}_c(\cc{C},\cc{W})^r
}$ and its dual 
$\xymatrix{
\cc{H}o^{c}_f(\cc{C},\cc{W})^r     
                 \ar[r]^{s^{op}} 
  & \cc{H}o^{f}_c(\cc{C},\cc{W})^\ell
}$ 
where $\ell$ stands for left, $r$ for right, $f$ for fibrant, and $c$ for cofibrant; 
the superscripts indicate homotopies, and the subscripts objects. 
The pseudofunctors $Q$ and $R$ defined in {\bf B.} can then be restricted to the bicategories shown in the diagram below:
$$
\xymatrix@C=5ex
   {
    \cc{H}o^{f}(\cc{C},\cc{W})^\ell       
                 \ar[r]^Q
  & \cc{H}o^{f}_c(\cc{C},\cc{W})^\ell     
                 \ar[r]^s
                 \ar@{->}@<-4pt>`u[rrr]`[rrr]^{R_\ell}[rrr]
  & \cc{H}o^{c}_c(\cc{C},\cc{W})^r        
                 \ar[r]^R
  & \cc{H}o^{c}_{fc}(\cc{C},\cc{W})^r     
                 \ar[r]^{s^{op}} 
  & \cc{H}o^{f}_{fc}(\cc{C},\cc{W})^\ell  
   }
$$
We define a fibrant-cofibrant replacement as this composition. Note that its domain and its codomain are bicategories of left homotopies.

\smallskip

{\em Although ultimately for the fibrant-cofibrant replacement both left and right homotopies are necessary, the latter become implicit in our construction. 
This allows for a concise, clear proof of the localization theorem, where many of the technical difficulties
of working with bicategories are avoided.}

\noindent {\bf D.}  
The $2$-functor {$\C \mr{i} \Wo$} factors through 
$\cc{H}o^f(\cc{C}, \cc{W})$\footnote{{Since $RQ$ involves only left homotopies, we drop the superscript $\ell$}.}, and we define 
\mbox{$q = RQ \,i: \C \mr{i} \cc{H}o^f(\cc{C}, \cc{W}) 
\mr{RQ} \HoC = \cc{H}o_{fc}(\cc{C}, \cc{W}) \subset 
\cc{H}o(\cc{C}, \cc{W})$.} Its universal property is deduced formally from the properties that $RQ$ and $i$ have separately.

\vspace{2ex}

Despite the independent interest that the results in this article may have, our motivation to develop a theory of model bicategories comes from potential applications in the homotopy theory of topoi. Given a site $\cc{C}$, the category of coverings with refinements as arrows fails to be filtered, but it underlies a 2-category which is 2-filtered in the sense of \cite{DS}. The \v{C}ech nerve $\ff{COV}(\cc{C}) \mr{\check{C}} \ff{SS}$ into the category of simplicial sets followed by the functor $\ff{SS} \mr{} Ho(\ff{SS})$ into the homotopy category factors through the poset $cov(\cc{C})$ of coverings under refinements, which is filtered, and so it determines a pro-object \mbox{$cov(\cc{C}) \mr{} \ff{SS}$,} that is, it determines an object of the category $Pro(Ho(\ff{SS}))$, where the information coded into the explicit homotopies is lost. In \cite{DD}, the first two authors developed a 2-dimensional theory of pro-objects which generalizes Grothendieck's pro-objects theory, and in \cite{TesisEmi} it is proven that a 2-model structure in a 2-category $\cc{C}$ can be lifted to the 2-category $2$-$Pro(\cc{C})$, a 2-dimensional generalization of the construction in \cite{EH}. This paper, in particular, complements \cite{DD} and \cite{TesisEmi}. The \v{C}ech nerve can be seen as a simplicial 2-pro-object, and thus it determines an object of the 2-category $Ho(2$-$Pro(\ff{SS}))$. In particular, shape theory of topological spaces discards the explicit homotopies and works with the \v{C}ech nerve in the category $Pro(Ho(\ff{SS}))$. Strong shape theory works in the category $Ho(Pro(\ff{SS}))$, that keeps the information coded in the explicit homotopies, but has the difficulty that the \v{C}ech nerve is not an object of $Pro(Ho(\ff{SS}))$.  Our results provide a conceptual framework to use the \v{C}ech nerve in strong shape theory as an object in \mbox{$\cc{H}o(2$-$Pro(\ff{SS}))$,} just as it is used in shape theory. {Note that the results in \cite{TesisEmi} are applied to lift the model structure in the category $\ff{SS}$ into a $2$-model structure in the 2-category $2$-$Pro(\ff{SS})$.}

\vspace{1ex}

{\bf About functorial factorizations.} We note that in {several} modern approaches to model categories functorial factorizations are considered as part of the model category axioms
{, but we have two reasons not to assume this stronger condition here.} 
The existence of such factorizations, under some extra assumptions on a model category (that permit Quillen's \emph{small object argument} \cite{Quillen}) has been the subject of several articles, for example \cite{Garner1}, \cite{Garner2}, \cite{Barthel-Riehl}. However for the construction of the model structure in $Pro(\cc{C})$ in \cite{EH}, even if the factorization in $\cc{C}$ is functorial, this does not follow for $Pro(\cc{C})$  because the extra assumptions on $\cc{C}$ are not inherited by $Pro(\cc{C})$ (see also \cite{Isaksen}, \cite{Barnea}).
This is also the case in the 2-dimensional version in \cite{DD} and \cite{TesisEmi}, which was our original motivation for introducing model bicategories and its basic theory to be applied in $2$-$Pro(\ff{SS})$. 

Independently of the example above, we think that to use a functorial factorization in the basic elementary treatment of the theory, especially if it is considered as part of the structure, is actually a step backwards from Quillen's revolutionary insight that to define homotopies any suitable defined \emph{cylinder object} can be used in place of the classical cylinder $X \times I$ ($I$ the topological or simplicial unit interval). 
He shows and stresses in \cite{Quillen} that functoriality of cylinders is not necessary and can be dropped. 
Finally, functorial factorizations yield functorial fibrant-cofibrant replacements which allow to consider only left (or only right) homotopies for proving that the homotopy category is the localization (see for example \cite[Thm. 1.2.10]{Hovey} or 
\mbox{\cite[Section 4]{JaG}),} unlike in the original proof of \cite[I, \S 1, Theorem 1]{Quillen} where both homotopies play an essential role.  
This breaks the beautiful interplay between both sides of the duality that is inherent to a model category.  

\subsection*{Organization}

The paper is structured as follows.
Section~\ref{sec:prelim} contains some preliminaries which are convenient to have explicitly at hand. In Section~\ref{sec:modelbicat} we give the basic definitions of the theory of model bicategories. 
In Section~\ref{sub:descripcionHoC} we 
define the homotopy bicategory $\HoC$ and prove the results explained in {\bf A} above, as well as other relevant results for $w$- and $q$-homotopies {such as the relation between right and left homotopies mentioned in {\bf E}}.
{In Section~\ref{sec:replacement-model} we construct the fibrant-cofibrant replacement, that is the pseudofunctors $Q$ and $RQ$, and the pseudonatural transformations $\rho$, $\lambda$ as explained in {\bf B} {and {\bf E}}}. Section \ref{sec:locthm} consists of Theorem~\ref{teo:localizationformodel} and its proof {as outlined in {\bf C}}. Appendix \ref{sec:transfer} establishes the transfer of structure for bicategories and pseudofunctors {mentioned in {\bf D}} that we apply in the paper.

\section{Preliminaries} \label{sec:prelim}

\smallskip

The concepts of bicategory, pseudofunctor, pseudonatural transformation and modification are ubiquitous in the literature, so we have decided to omit their definitions here. {The reader can check Appendix \ref{sec:transfer} where we found it necessary to recall these definitions because the axioms are explicitly used, contrary to the case of the body of the paper.}

\vspace{1ex} 

\begin{sinnadastandard} \label{basic notation} 
{\bf Notation.} 
For a bicategory $\cc{C}$ and objects $X,\, Y \in \cc{C}$, 
$\cc{C}(X,\,Y)$ denotes the hom-category of arrows and $2$-cells. For two bicategories $\cc{C},\, \cc{D}$, $Hom(\cc{C},\,\cc{D})$ denotes the bicategory of pseudofunctors, pseudonatural transformations and modifications. Vertical composition is denoted by ``$\circ$", and horizontal
composition by ``$\ast\,$". We consider ``$\,\ast\,$" more binding than
``$\,\circ\,$" .
\end{sinnadastandard} 

\noindent 
{\bf Coherence.} 
There is a well-known coherence theorem (see for example \cite{basicbicat}) which generalizes the coherence theorem for tensor categories. Given any sequence of composable arrows in a bicategory, the parentheses determine the order in which the compositions are performed. The coherence theorem states that the arrows resulting of any choice of parentheses (and adding or subtracting identities) are canonically isomorphic by an unique 2-cell built with the associators and the unitors. This justifies the following abuse of notation which greatly simplifies the computations: 

\begin{sinnadastandard} \label{sin:abuse} 
\emph{We write any horizontal composition of arrows omitting the parentheses and the identities. In this way, the associator and the unitors disappear in the diagrams of 2-cells.}
\end{sinnadastandard}

{\begin{sinnadastandard} \label{dualbicategory}
{\bf Dual bicategory.}
Given a bicategory $\cc{C}$, the \emph{dual} bicategory, denoted $\cc{C}^{op}$, consists of the bicategory with the same objects, arrows and $2$-cells, but formally reversing the direction of the arrows (while retaining the direction of the $2$-cells).
\end{sinnadastandard}}

\begin{sinnadastandard} \label{equivalencias}
{{\bf Equivalences}}.
 
1. An arrow $X \mr{f} Y$ of a bicategory is an {\em equivalence} if there exists an arrow $Y \mr{g} X$ (which we call a {\em quasi-inverse} of $f$) and isomorphisms $g * f \cong id_X$, $f * g \cong id_Y$. It is well-known that these isomorphisms can be taken satisfying the usual triangular identities, and we will assume that this is the case when needed.

{2. It is well known that $X \mr{f} Y$ is an equivalence if and only if for every object $Z$ the functor $\C(Z,X) \mr{f_*} \C(Z,Y)$ is an equivalence of categories, and if and only if the functor \mbox{$\C(Y,Z) \mr{f^*} \C(X,Z)$} is so.}

3. {Recall that} a pseudonatural transformation $\theta: F \Mr{} G: \C \mr{} \cc{D}$ 
is an equivalence in the {2-category} $Hom(\C,\cc{D})$, i.e. there exists $G \Mr{\mu} F$ and invertible modifications \mbox{$\theta \circ \mu \cong id_{G}$, $\mu \circ \theta\cong id_{F}$,} if and only if each $\theta_X$ is an equivalence in $\cc{D}$. 
{For a proof, see for example \cite[1.10]{PronkWarren}, or, for a detailed proof \cite[1.17]{JaG}.} {It can be checked that this same proof also holds for the 2-categories of lax functors.}

4. A pseudofunctor $\C \mr{F} \cc{D}$ is a {\em biequivalence of bicategories}
if there exist a pseudofunctor $\cc{D} \mr{G} \C$ (which we call {a {\em bi-inverse} or} a {\em pseudoinverse} of $F$) and pseudonatural transformations \mbox{$GF \Mr{\alpha} id_{\C}$,} $FG \Mr{\beta} id_{\cc{D}}$ which are equivalences.
\end{sinnadastandard}

\begin{sinnadastandard} \label{wfull}
{\bf On a weak notion of full functor.}
{For a category $\nn{X}$, we denote by $\nn{X}^\nn{2}$ the category that has the arrows of $\nn{X}$ as objects and the commutative squares as arrows.} 
Note that a functor between categories $\nn{X} \mr{F} \nn{A}$ is full if and only if the induced functor 
{$\nn{X}^\nn{2} \mr{F^2} (\nn{Im_F})^\nn{2} \subset \nn{A}^\nn{2}$} is surjective on objects, where $\nn{Im_F}$ denotes the full image {of $F$}.

\begin{definition} \label{essfull}
We say that a functor between categories  $\nn{X} \mr{F} \nn{A}$ is {\em essentially full} if and only if the induced functor 
{$\nn{X}^\nn{2} \mr{F^2} (\nn{Im_F})^\nn{2} \subset \nn{A}^\nn{2}$} is essentially surjective. 

That is, given any $FX \mr{g}  FY$ there exists $X',\, Y'$, 
$X' \mr{f} Y'$, and 
isomorphisms $FX' \mr{a} FX$, $FY' \mr{b} FY$ such that 
$g \circ a = b \circ Ff$.
\end{definition}

\begin{remark}
A notion of essentially full functor between categories which is similar to the one in the present article was considered in \cite[p.469]{Rosicky}. 
We recall that there is also a notion of {\em essentially full pseudofunctor} $F$ between bicategories: it means that the induced functors between the Hom-categories are essentially surjective. This means that $F$ is surjective on 1-cells, up to invertible 2-cells, that is, given any 
$FX \mr{g}  FY$ there exists  $X \mr{f} Y$, and an invertible  
$2$-cell $Ff \Mr{\alpha} g$. 
We warn that for a functor it is not equivalent to be  essentially full as in Definition \ref{essfull} or as a pseudofunctor between the discrete bicategories.
\end{remark}

The following can be easily checked:

\begin{remark} \label{MM1}
A functor between categories  $\nn{X} \mr{F} \nn{A}$ is essentially surjective and essentially full if and only if $\nn{X}^\nn{2} \mr{F^2} \nn{A}^\nn{2}$ is essentially surjective, that is, for every $A$, $B$, and $A \mr{g} B$ in $\nn{A}$, there exist $X$, $Y$, isomorphisms $FX \mr{a} A$, $FY \mr{b} B$, and $X \mr{f} Y$ in 
$\nn{X}$ such that $g \circ a = b \circ Ff$.
\end{remark}
 \end{sinnadastandard}
 
\begin{sinnadastandard} \label{sin:limits}
We will consider \emph{biLimits} in an arbitrary bicategory $\cc{C}$, whose universal property is established by postulating an equivalence between the appropriate hom-categories (as in for example \cite[\S 6]{Kellyelem}). 
By the term \emph{psLimit} (``ps" should be read as ``pseudo"), we require the equivalence to be an isomorphism. The use of an initial capital letter was adopted following \cite{Cole}. 
We will consider in $\cc{C}$ the concepts of initial object, product, pullback, comma-object, tensor and their dual versions terminal object, coproduct, pushout, cocomma-object, cotensor. 
We specify a case which differs from the terminology in \cite{Kellyelem}. Given a diagram $B \mr{f} D \ml{g} C$, by its \emph{biPullback} we refer to the {bi}-universal diagram of the form
$$
\xymatrix{ P \ar@{}[rd]|{\cong \, \Downarrow \; \lambda} \ar[d]_{\pi_2} \ar[r]^{\pi_1} & C \ar[d]^{g} \\ B \ar[r]_f & D}
$$

We also refer to $\pi_1$  as the \emph{biPullback of} $f$ along $g$ (resp. $\pi_2$ as the \emph{biPullback of} $g$ along $f$). When $\lambda$ is not required to be invertible, we refer to this biLimit as a \emph{lax-biPullback} (which according to the terminology in \cite[(4.2)]{Kellyelem}, would be called \emph{bi-iso-comma object} and \emph{bi-comma object} respectively).

We denote the bicoProduct of $X$ and $Y$ by 
$X \amalg Y$ and its inclusions by $i_0,i_1$. Given arrows 
$X \mr{f} Z$, $Y \mr{g} Z$, we denote the induced arrow (note that this is an abuse of notation since the arrow is not unique) by 
$\binom{f}{g}:X \amalg Y \mr{} Z$. We leave unnamed the invertible 2-cells $f \cong \binom{f}{g} \ast i_0$, 
$g \cong \binom{f}{g} \ast i_1$. Reciprocally, given an arrow 
$X \amalg Y \mr{} Z$ we denote it by $\binom{f}{g}$ to indicate that we define $f = \binom{f}{g} \ast i_0$, 
$g =\binom{f}{g} \ast i_1$.
Note that if we have in addition $Z \mr{h} W$, we have 
$h * \binom{f}{g} = \binom{h * f}{h * g}$.  
Given 2-cells $X \cellrdEcorta{f}{\alpha}{f'} Z$, 
$Y \cellrdEcorta{g}{\beta}{g'} Z$, we denote the induced 2-cell by $\binom{\alpha}{\beta}: \binom{f}{g} \Mr{} \binom{f'}{g'}$. 
When $Y=X$ we denote the codiagonal $\binom{id_X}{id_X}$ by 
$\nabla_X$. Similarly we denote the biProduct of $X$ and $Y$ by
$X \times Y$, and given arrows $Z \mr{f} X$, $Z \mr{g} Y$, we denote the induced arrow by 
$(f,\,g):Z \mr{} X \times Y $. The diagonal is denoted {$\Delta_X$.}
\end{sinnadastandard}

\section{Model bicategories} \label{sec:modelbicat}

In this section we will give the definition of a model bicategory. We consider Quillen's original axioms as were introduced in \cite[\S 1]{Quillen}. Since in the known examples Quillen's axioms of \emph{closed} model category hold, it is usual nowadays to consider these stronger axioms as the definition of model category. Especially since this notion (that of a closed model category) admits a neat and compact presentation in terms of \emph{weak factorization systems}.
{In this paper, however, 
in an attempt to be closer to the original construction of the homotopy category in \cite[\S 1]{Quillen},
we won't assume these stronger axioms.}
We think it is still convenient (and important) to consider Quillen's original axioms as a guide for the model bicategory axioms.

We now give the notion of lifting property for a pair of arrows in a bicategory. Note that a {stronger} notion in which the lifting is required to be {\em universal} is considered {for \emph{factorization systems} (not weak)} in 2-categories in \cite[1.3,1.4]{Vitale}. {This notion would correspond to asking the functor $h$ in the diagram below to be an equivalence of categories, instead of just essentially surjective and essentially full as in our definition.}

Let $\cc{C}$ be a bicategory and $A \mr{i} X$, $Y \mr{p} B$ be two morphisms in $\cc{C}$, they determine the following diagram (note that the associator defines a natural transformation 
\mbox{$p_\ast \, i^\ast \Mr{\theta} i^\ast \, p_\ast$).}
$$
\xymatrix@R=2ex
    {
     \cc{C}(X,\,Y)  \ar@/^5ex/[drr]^{i^\ast}
                   \ar[dr]^(.65){h} 
                   \ar@/_4ex/[rddd]^{p_\ast}  \\
   & \nn{P} \ar[dd]_{\pi_2} 
            \ar[r]^(.4){\pi_1}
            \ar@{}[rdd]|{\Downarrow \, \cong}
   & \cc{C}(A,\,Y) \ar[dd]^{p_\ast} 
   \\ \\
   {} 
   & \cc{C}(X,\,B) \ar[r]^{i^\ast} 
   & \cc{C}(A,\,B)
    }
$$
where $\nn{P}$ is the biPullback of categories (in fact a psPullback),  
$\pi_1 \, h = i^\ast$,  $\pi_2 \, h = \, p_\ast$.

\begin{definition}\label{def:lifting}
We say that a pair $(i,p)$ as above has the \emph{lifting property} if the functor $h = (i^\ast,\, p_\ast)$ in the diagram above is essentially surjective and essentially full\footnote{{In a preliminary version of this article a stronger axiom was required for the lifting property which implied the uniqueness of the filler. We are grateful to Valery Isaev for pointing this out to us.}}.
\end{definition}

{Note} that an object of $\nn{P}$ is a triplet 
 $(a, \,\gamma, \, b)$ as in the left below, and an arrow between two such objects is a pair of $2$-cells 
$(a, \, \gamma,\, b) 
\xymatrix@C=5ex{{} \ar[r]^{(\alpha,\,\beta)} & {}}
(a', \, \gamma',\, b')
$ 
as in the middle diagram, satisfying the equation on the right:
\begin{equation} \label{P_arrows}
\vcenter{\xymatrix@R=7ex@C=10ex
   {
    A \ar@{}[dr] | {\cong \, \Downarrow \;\; \gamma}
      \ar[r]^a \ar[d]^i & Y \ar[d]^p
    \\
    X \ar[r]^b & B
   }}
\vcenter{\xymatrix{,}}
\hspace{5ex}
\vcenter{\xymatrix@R=7ex@C=10ex 
    {
     A \ar[d]_{i} 
       \ar@<1.2ex>[r]^(.4){a}
       \ar@{}[r] | (0.6){\alpha \, \Downarrow}
       \ar@<-1.2ex>[r]^(.4){a'}
       \ar@{}[dr] | {\cong \, \Downarrow \;\; \gamma \; \gamma'}
   & Y \ar[d]^{p} 
   \\
     X \ar@<1.2ex>[r]^(.4){b}
       \ar@{}[r] | (0.6){\beta \, \Downarrow}
       \ar@<-1.2ex>[r]^(.4){b'} 
   & B
    }}
\vcenter{\xymatrix{,}} 
\hspace{5ex}
{ \beta \ast i \circ \gamma = 
\gamma' \circ p \ast \alpha }
\end{equation}
{
We now unfold Definition \ref{def:lifting} so as to have at hand the statements to be used in practice. We found it convenient to set aside what it means for $h$ to be essentially surjective, because it has independent interest and it is sufficient in many applications of the lifting property.
The proof of the following Lemma is immediate by the description of the category $\nn{P}$ above, and by Remark \ref{MM1}.

\begin{lemma}\label{lem:lifting_prop} 
    With the notation as in Definition \ref{def:lifting},
\begin{enumerate}
    \item $h$ is essentially surjective if and only if the following condition holds:

\smallskip
\noindent {\bf {Ls}.} For each triplet 
$(a, \,\gamma, \, b)$ as in the left diagram below, there exist a morphism $X \mr{f} Y$ and invertible 2-cells $\lambda$, 
$\rho$ as in the middle diagram, satisfying the equation on the right:

\vspace{-2ex}

\begin{equation}  \label{L1}
\vcenter{\xymatrix@R=7ex@C=10ex
   {
    A \ar@{}[dr] | {\cong \, \Downarrow \;\; \gamma}
      \ar[r]^a \ar[d]^i & Y \ar[d]^p
    \\
    X \ar[r]^b & B
   }}
\vcenter{\xymatrix{,}}
\hspace{5ex}
\vcenter{\xymatrix@R=7ex@C=9ex
      {
          A  \ar[r]^{a} \ar[d]^{i} 
             \ar@{}[rd]|(.22){\; \cong \, \Uparrow \, \lambda}
             \ar@{}[rd]|(.7){\cong \, \Downarrow \, \rho}  
        & Y  \ar[d]^{p}                 
        \\
          X  \ar[r]^{b} \ar[ru]^{f}
        & B
       }}
\hspace{1ex}
,
\hspace{5ex}
{\rho \ast i =  \gamma \circ p \ast \lambda}
\end{equation}
We say that $(f,\,\lambda,\, \rho)$ is a \emph{filler} for the square $(a, \,\gamma, \, b)$.

\item $h$ is essentially surjective and essentially full (that is, $(i,p)$ has the lifting property) if and only if the following condition holds:

\noindent {\bf {L}.}
For each $(a, \,\gamma, \, b)$, $(a', \,\gamma', \, b')$  
and 
$(a, \, \gamma,\, b) 
\xymatrix@C=5ex{{} \ar[r]^{(\alpha,\,\beta)} & {}} 
(a', \, \gamma',\, b')$ 
as in (\ref{P_arrows}), there exist fillers 
$(f,\,\lambda,\,\rho)$,  $(f',\,\lambda',\,\rho')$ as in (\ref{L1}), and 
$f \Mr{\delta} f'$ such that 
\begin{equation}
{ \lambda' \circ \delta \ast i
=
\alpha \circ \lambda
\hspace{1ex} , \hspace{5ex}
\rho' \circ p \ast \delta
=
\beta \circ \rho}
\end{equation}

\end{enumerate}
\end{lemma}
}

\begin{remark}
For $\alpha$ and $\beta$ invertible, it is easy to check that the statement {\bf L} above follows from {\bf Ls}. Indeed, taking a filler $(f,\lambda,\rho)$ for the square $(a, \,\gamma, \, b)$ and constructing the filler $(f,\alpha^{-1} \circ \lambda,\beta \circ \rho)$ for the square $(a', \,\gamma', \, b')$ does the trick.
\end{remark}

\vspace{1ex}

\noindent \dbend \hspace*{2ex}
Note that the lifting property is self-dual, equivalent to but different from its dual. 
 \hspace*{5.5ex} The equivalence holds because the $2$-cell 
 $\gamma$ is invertible. In the dual formulation \hspace*{6ex} the biPullback of categories is replaced by the op-biPullback.

\vspace{1ex}

It is convenient to spell out the statement dual to {\bf Ls} above:

\vspace{1ex} 

\noindent {\bf Ls$^{op}$.}
For each invertible $2$-cell 
$(a, \,\gamma, \, b)$ as in the left diagram below, there exist a morphism $f$ and invertible 2-cells $\lambda$, 
$\rho$ as in the middle diagram, satisfying the equation on the right:

\vspace{-2ex}

\begin{equation}  \label{L1op}
\vcenter{\xymatrix@R=7ex@C=10ex
   {
    A \ar@{}[dr] | {\cong \, \Uparrow \;\; \gamma}
      \ar[r]^a \ar[d]^i & Y \ar[d]^p
    \\
    X \ar[r]^b & B
   }}
\vcenter{\xymatrix{,}}
\hspace{5ex}
\vcenter{\xymatrix@R=7ex@C=9ex
      {
          A  \ar[r]^{a} \ar[d]^{i} 
             \ar@{}[rd]|(.22){\; \cong \, \Uparrow \, \lambda}
             \ar@{}[rd]|(.7){\cong \, \Downarrow \, \rho}  
        & Y  \ar[d]^{p}                 
        \\
          X  \ar[r]^{b} \ar[ru]^{f}
        & B
       }}
\hspace{1ex}
,
\hspace{5ex}
{ \gamma \circ \rho \ast i = p \ast \lambda}
\end{equation}

\begin{definition} \label{def:modelbicat}
 We say that a bicategory $\cc{C}$ is a \emph{model bicategory} if it is equipped with three classes of morphisms
$\cc{F}, co\cc{F}$ and $\cc{W}$ 
  called respectively fibrations, cofibrations and weak equivalences satisfying the following set of axioms:

\smallskip

\noindent {\bf M0.} $\cc{C}$ has bi{terminal} objects, biPullbacks, and their dual coLimit versions.

\smallskip

\noindent {\bf M1.} Given a cofibration $i$ and a fibration $p$, if one of them is a weak equivalence, then $(i,p)$ has the  lifting property {\bf L}. 

\smallskip

\noindent {\bf M2.} Every morphism $f$ can be factored up to an invertible 2-cell as $f \cong p \ast i$ with $i$ a cofibration and $p$ a fibration. 
In addition, we may require either $i$ or $p$ to be a weak equivalence.

\smallskip

\noindent {\bf M3.} Fibrations (respectively cofibrations) are closed under composition and biPullbacks (respectively biPushouts). Every equivalence is a fibration and a cofibration. If there is an invertible 2-cell $f \cong g$ and $f$ is a fibration (resp. a cofibration), then so is $g$.

\smallskip

\noindent {\bf M4.} If a morphism $f$ is the biPullback (resp. biPushout) of a fibration (resp. cofibration) which is also a weak equivalence, then $f$ is a weak equivalence.

\smallskip

\noindent {\bf M5.} The class of weak equivalences satisfies the ``3 for 2" axiom:
for every three arrows $f,g,h$ such that there is an invertible 2-cell $gf \cong h$, whenever two of the three arrows are weak equivalences, so is the third one. Also, every equivalence is a weak equivalence.
\end{definition}

Note that by interchanging fibrations with cofibrations and keeping the same weak equivalences it is clear that the bicategory dual to a model bicategory is also a model bicategory, so that each valid statement has a valid dual statement. 

\begin{remark}   \label{rem:axiomM1s}
{\bf M1s axiom.}
A notion of 2-model 2-category was introduced in \cite{TesisEmi} where the lifting axiom {\bf M1} was considered with the essentially surjective item {\bf Ls} in Lemma \ref{lem:lifting_prop}. We denote by {\bf M1s} this weaker axiom, which as stated before is sufficient in many applications of axiom {\bf M1}. The full strength of axiom {\bf M1} is nevertheless necessary to deal with arbitrary, non-invertible 2-cells in the fibrant-cofibrant replacement.
\end{remark}
\begin{remark}   \label{rem:axiomMM}
\noindent {\bf MM axioms.} Though the axioms in Definition \ref{def:modelbicat} suffice to construct the homotopy bicategory, in order to vertically compose the homotopies as in \cite{Quillen} further axioms are required, obtained by considering lax biPullbacks (and lax biPushouts) in axioms {\bf M0}, {\bf M3}, and {\bf M4}. Since in this paper these axioms are only needed for that construction, that we do in Section \ref{sub:compvert}, we chose to make them explicit there.
\end{remark}

\begin{remark}
{
Model categories are usually required to be complete and cocomplete, so axiom {\bf M0} above may seem to be too weak a condition regarding limit existence. 
One could also consider stronger completeness axioms depending on the context:

\smallskip

\noindent {\bf MC0.} $\cc{C}$ has all conical finite biLimits and biColimits

\noindent {\bf MW0.} $\cc{C}$ has all weighted finite biLimits and biColimits.}

\smallskip

\vspace{1ex}

{In \cite{TesisEmi} axiom {\bf MW0} was assumed but it is never used in this paper, and we chose to omit it. Related to this, 
we mention that the particular case of finite tensors and cotensors as well as a stronger lifting property (Quillen's axiom SM7) are assumed in Quillen's simplicial model categories, where 
$X \otimes \nn{I}$ and $\{\nn{I},X\}$ furnish functorial cylinders and path objects respectively (for $\nn{I}$ the unit interval). This is generalized to enriched model categories in \cite{GM}. In this paper we do not consider the $\bf Cat$-enriched version corresponding to simplicial model categories, and we have not investigated this line of research. As usual for \mbox{2-dimensional} category theory, the axioms coming from enrichment in $\bf Cat$ are too strict when compared to the axioms introduced here and in \cite{TesisEmi}, \cite{TesisBarton}.
}  
\end{remark}

\begin{remark}
A model category $\nn{C}$ (\cite{Quillen}) can be regarded as a model bicategory \mbox{(2-category)} in which every 2-cell is the identity.
\end{remark}

{Taking the set of connected components in the hom-categories  of the homotopy bicategories constructed in this paper yields the constructions and results of \cite{Quillen} for $\nn{C}$, see \cite{JaG} for a full account of this.}

\begin{remark}
{It can be checked that} any bicategory satisfying {\bf M0} is a model bicategory with $\cc{W}$ the equivalences, and every arrow a fibration and a cofibration. 
\end{remark}

For the rest of the article, 
except when we consider the more general case 
{$(\cc{C},\,\cc{W})$ of a category and a single class 
$\cc{W}$, $\C$ will denote an arbitrary model bicategory 
$(\cc{C},\, \cc{F},\, co\cc{F},\, \cc{W})$.}

\begin{definition}
We say that an arrow in $\C$ is a \emph{trivial (co)fibration} if it is simultaneously a (co)fibration and a weak equivalence. 
We will use the following notation: 

\begin{enumerate}
\item {$\xymatrix{\cdot \ar[r]^{\simeq} & \cdot }$ is an equivalence.}
\item $\xymatrix{\cdot \ar[r]^{\sim} & \cdot }$ is a weak equivalence.
\item $\xymatrix{\cdot \ar@{->>}[r] & \cdot }$ is a fibration,
$\xymatrix{\cdot \ar@{->>}[r]^{\sim} & \cdot }$ is a trivial fibration,
\item $\xymatrix{\cdot \ar@{^{(}->}[r] & \cdot }$ is a cofibration,
$\xymatrix{\cdot \ar@{^{(}->}[r]^{\sim} & \cdot }$ is a trivial cofibration
\end{enumerate} 
\end{definition}
\begin{definition} Let $X$ be an object of $\cc{C}$.
 \begin{enumerate}
  
  \item We say that $X$ is \emph{fibrant} if the morphism $X \mr{}*$ is a fibration.
  
  \item We say that $X$ is \emph{cofibrant} if the morphism $0\mr{}X$ is a cofibration.
                   
 \end{enumerate}

We denote by $\cc{C}_f$, $\cc{C}_c$, $\cc{C}_{fc}$ the full sub-bicategories of fibrant, cofibrant and fibrant-cofibrant objects (i.e. objects that are both fibrant and cofibrant) respectively. We denote with the same letters $\cc{F}$, $co\cc{F}$, $\cc{W}$ the restrictions of these families of arrows to the three bicategories.
\end{definition}

\begin{remark}
Note that $0$ and $*$ denote the Initial and the Terminal object respectively given by axiom {\bf M0}. More explicitly, for each $X\in \cc{C}$, there exists a morphism $0\mr{}X\in \cc{C}$ up to unique invertible 2-cell, and dually for $*$. 
In the previous definition the abuse of saying ``the morphism'' is justified by axiom {\bf M3}. \qed
\end{remark}

\begin{definition} \label{def:retrsect}
	Let $X \mr{s} Y$, $Y \mr{r} X$ be arrows of a bicategory. If there is an invertible 2-cell $rs \cong id_X$, $s$ is called a \emph{section for} $r$, and $r$ is called a \emph{retraction for} $s$.
	An arrow $X \mr{s} Y$ is called a \emph{section} if there exists $r$ such that $s$ is a section for $r$ and dually an arrow is called a \emph{retraction} if it admits a section. An arrow that is either a section or a retraction is called a {\em split} arrow.
\end{definition}

{A key fact in the theory of model categories} is that any weak equivalence between fibrant-cofibrant objects can be factored as a section followed by a retraction, both of them weak equivalences, and that this fact can be used to prove Whitehead's theorem (see \cite[Th. 1.10]{GJ}, and also \cite[Prop. 3.1.21]{S.paper11}). We now show the bicategorical equivalent of this statement, {whose short proof is just the bicategorical version of the standard model-category argument, and is included for the sake of completeness.}

\begin{proposition} \label{prop:compderetractos}
Let $X \mr{f} Y\in\C$ be a weak equivalence, with $X$ fibrant and $Y$ cofibrant. Then axiom {\bf M2} yields a factorization of $f$ as a composition
$X \mr{i} Z \mr{p} Y$ of a section and a retraction, both of them weak equivalences. If furthermore both $X$ and $Y$ are fibrant-cofibrant, then so is $Z$.
\end{proposition}

\begin{proof}
Let $X \mr{f} Y$ be such a weak equivalence. By axioms {\bf M2} and {\bf M5} we have 
$\vcenter
    {\xymatrix
       {
        X \ar[rr]_\sim ^f 
          \ar@{^{(}->}[rd]_i
          \ar@<-1ex>@{}[rd]^[@]{\sim} 
       & 
           \ar@{}[d]|{\cong} 
       & Y 
      \\ 
       & Z \ar@{->>}[ru]_p
           \ar@<-1ex>@{}[ru]^[@]\sim
       }
    }
$. The last statement of the proposition clearly holds by axiom {\bf M3}. We will show that $i$ is a section, dually it follows that $p$ is a retraction. Using axiom {\bf M1}, we have $\vcenter{\xymatrix{X \ar[r]^{id_X} \ar[d]_i & X \ar[d] \\ Z \ar[r] \ar@{.>}[ru]_r^\cong & 1}}$ as desired.
\end{proof}

\section{The homotopy bicategory of a model bicategory} \label{sub:descripcionHoC}

{This section depends heavily on the results of \cite{DDS.loc_via_homot}, and we recommend the reader to have at hand this reference. Nevertheless we will recall here some of the content of \cite{DDS.loc_via_homot} to improve readability.} 
We fix a model bicategory $\cc{C}$. 
We will define a notion of $q$-homotopy between arrows of $\C$, which is analogous to Quillen's notion of homotopy for model categories. 
Our objective is to construct a bicategory $\HoC$ whose objects and arrows are those of $\C_{fc}$, and whose 2-cells are classes of $q$-homotopies under an equivalence relation, 
{together with a pseudofunctor $\C \mr{q} \HoC$ which is the $2$-dimensional localization of $\C$  with respect to the class $\cc{W}$ of weak equivalences in the sense of the following universal property: 
$\C \mr{q} \HoC$ sends weak equivalences to equivalences and 
 $q$-precomposition $Hom(\HoC,\cc{D}) \mr{q^*} Hom_{\cc{W},\Theta}(\C,\cc{D})$ 
determines a biequivalence of bicategories for every bicategory $\cc{D}$,  where $Hom_{\cc{W},\Theta}(\C,\cc{D})$ stands for the full sub-bicategory of $Hom(\C,\cc{D})$ given by those pseudofunctors that send weak equivalences into equivalences 
(Theorem~\ref{teo:localizationformodel})}. 

The localization of a bicategory at a family of arrows was first considered in \cite{PRONK2} by means of a 2-dimensional version of Gabriel-Zisman's calculus of fractions.
There it is considered a bicategory $Hom_{\cc{W}}(\C,\cc{D})$ with the same objects as $Hom_{\cc{W},\Theta}(\C,\cc{D})$, but whose arrows are the pseudonatural transformations that map the arrows of $\cc{W}$ to equivalences (when interpreted as pseudofunctors from $\C$ into a cylinder bicategory).
However, we have recently noticed that a simple verification shows that 
$Hom_{\cc{W}}(\C,\cc{D}) = Hom_{\cc{W},\Theta}(\C,\cc{D})$ (more details can be found in \cite[Remark 3.7]{PrSc}), so 
the universal property that can be found in \cite[Th. 21]{PRONK2}, involving $Hom_{\cc{W}}(\C,\cc{D})$, is indeed the same one of Theorem \ref{teo:localizationformodel} but stated differently.

We will introduce $q$-homotopies as $w$-homotopies (a notion which depends only of the class $\cc{W}$) that satisfy some extra conditions with respect to the classes $\cc{F}$ and $co\cc{F}$.
This allows to use previous results of \cite{DDS.loc_via_homot}, where for an arbitrary pair $(\cc{A},\Sigma)$ given by a family $\Sigma$ of arrows of a bicategory $\cc{A}$ there is a construction of a bicategory $\HoA$ whose 2-cells are formed with $w$-homotopies. However, unless a vertical composition can be defined, the 2-cells of $\HoA$ consist of the classes of finite composable sequences of $w$-homotopies.
 
In \S \ref{sub:compvert} we will show that $q$-homotopies can be vertically composed, and 
in \S \ref{sub:relating} we will show that for any $w$-homotopy between arrows of $\C_{fc}$ there is a $q$-homotopy in the same class. This will allow to give $\HoC$ the bicategory structure of $\Wo$.

\begin{definition} \label{def:whomot} \label{lhpy}
We consider any pair $(\C,\cc{W})$ given by a family that we call {\em weak equivalences} of a bicategory $\C$. Let $X \in \C$. 
A \emph{$w$-cylinder} $C$ (\emph{for} $X$) is given by the data $C = (W,Z,d_0,d_1,x,s,\alpha_0,\alpha_1)$, 
	fitting in $$\vcenter{\xymatrix@C=1.5pc@R=1.5pc{X \ar[rr]^{d_0} \ar[dd]_{d_1} \ar[rrdd]|{\comw{M^M} x \comw{M^M}} & \ar@{}[dr]|{\cong \; \Downarrow \; \alpha_0 \;} & W \ar[dd]^[@]{\sim}_{s}\\
			\ar@{}[dr]|{ \; \cong \; \Uparrow \; \alpha_1} && \\ 
			W \ar[rr]_{\sim}^{s}  && Z}}$$ {($s$ is a weak equivalence, $\alpha_0$ and $\alpha_1$ are invertible 2-cells)}. 
We denote $\alpha_1^{-1} \circ \alpha_0$ by $\alpha$.

\vspace{1ex}

If $(\C,\cc{W})$ is part of a model bicategory structure, then a \emph{$q$-cylinder} is a $w$-cylinder such that $Z = X$, $x=id_X$, and the arrow $\binom{d_0}{d_1}: X \amalg X \mr{} W$ is a cofibration. In this case we write $C = (W,d_0,d_1,s,\alpha_0,\alpha_1)$.

\vspace{1ex}

Let $f,g: X \rightarrow Y\in \C$. A \emph{left $w$-homotopy $H$ from $f$ to $g$}, which we will denote by $f \mrhpy{H} g$, is given by the data  
	$H = (C,h,\eta,\eps)$, where $C$ is a $w$-cylinder for $X$ as above, $h$ is an arrow $W \mr{h} Y$ and $\eta,\eps$ are 2-cells $f \Mr{\eta} h * d_0$, $h * d_1 \Mr{\eps} g$. 
{We organize these data as follows:}

\vspace{-4ex}

{
\begin{equation} \label{eq:H}
\begin{tabular}{cc}
$f \mrhpy{H} g$: \hspace{3ex}  
$\vcenter{\xymatrix{X \ar@<-.5ex>[dr]_{x} \ar@<1ex>[r]^{d_0} 
             \ar@<-1ex>[r]_{d_1} & W \ar[r]^h \ar[d]^[@]\sim_s & Y \\ & Z}}$ &
\begin{tabular}{c}
$f \Mr{\eta} h \ast d_0$ \\
$s \ast d_0 \Mr{\alpha_0} x \Ml{\alpha_1} s \ast d_1$ \\
$h \ast d_1 \Mr{\eps} g$ 
\end{tabular}
\end{tabular}
\end{equation}
}
{We say that $H$ \emph{has invertible cells} if $\eta$ and $\varepsilon$ are invertible.}	

\vspace{1ex}

A \emph{left $q$-homotopy} is a left $w$-homotopy in which $C$ is a $q$-cylinder. 

\vspace{1ex}

We say that a $w$-cylinder (in particular a $q$-cylinder) is \emph{fibrant} if the arrow $s$ is a fibration. We use the same terminology for left $w$-homotopies.
\end{definition}

\noindent 
Note that the abuse of saying ``the arrow  $\binom{d_0}{d_1}$ is a cofibration" is justified by axiom {\bf M3}.

\vspace{1ex}

{
We record for convenience the dual structures of right 
$w$- and $q$-homotopies. These notions arise from considering the \emph{opposite model structure} on the bicategory $\C^{op}$. 

\begin{definition} \label{sin:rightsigmahomot} \label{rhpy}
A \emph{$w$-path-object} $P$ (\emph{for} an object $Y$) is given by the data \mbox{$P = (V,Z,c_0,c_1,y,t,\beta_0,\beta_1)$,} 
	fitting in 
$$\vcenter
{
\xymatrix@C=1.5pc@R=1.5pc
      {
       Z \ar[rr]^{t \;\; \sim} 
         \ar[dd]^{t}_[@]{\sim} 
         \ar[rrdd]|{\comw{M^M} y \comw{M^M}}
      &  \ar@{}[dr]|{\cong \; \Downarrow \; \beta_0 \;} 
      & V \ar[dd]^{c_0}
      \\
		\ar@{}[dr]|{ \; \cong \; \Uparrow \; \beta_1} 
	  && 
	  \\ 
		V \ar[rr]^{c_1}  
	  && Y
	  }
}
$$ 
We denote $\beta_1^{-1} \circ \beta_0$ by $\beta$.
A \emph{$q$-path-object} is a $w$-path-object such that $Z = Y$, $y=id_Y$, and the arrow $(c_0,\,c_1): V \mr{} Y$ is a fibration. In this case we write $P = (V,c_0,c_1,t,\beta_0,\beta_1)$.

\vspace{1ex}

A \emph{right $w$-homotopy $K$ from $f$ to $g$}, which we will denote by $f \mrhpy{K} g$, is given by the data  
	$K = (P,k,\delta,\epsilon)$, where $P$ is a $w$-path-object for $Y$, $k$ is an arrow $X \mr{k} V$ and $\delta,\epsilon$ are 2-cells $f \Mr{\delta} c_0 \ast k$, 
$c_1\ast k \Mr{\epsilon} g$.

A \emph{right $q$-homotopy} is a right $w$-homotopy in which $P$ is a $q$-path-object. 

{We say that a $w$-path-object (in particular a $q$-path-object) is \emph{cofibrant} if the arrow $t$ is a cofibration. We use the same terminology for right $w$-homotopies.}
\end{definition}
}
In this section we will work only with left $w$-homotopies, and thus omit to specify it.

\begin{remark} \label{rem:existencilindros}
For any $X\in\C$, we can construct a fibrant $q$-cylinder for $X$ as follows. 
We use axiom {\bf M2} and factorize 
$$\vcenter{\xymatrix{X \amalg X \ar[rr]^{\nabla_X} \ar@{^{(}->}[rd]_{\binom{d_0}{d_1}} && X \\
& W \ar@{->>}[ru]_s
\ar@<-1ex>@{}[ru]^[@]\sim \ar@{}[u]|>>>>{ \Uparrow \, \cong}}}$$
By the universal property of the coProduct 
$\binom{s * d_0}{s * d_1} \Mr{} \nabla_X$ corresponds to two invertible 2-cells $s * d_0 \Xr{\alpha_0} id_X$, $s * d_1 \Xr{\alpha_1} id_X $ and thus we have a fibrant $q$-cylinder $C = (W,d_0,d_1,s,\alpha_0,\alpha_1)$.
\end{remark}

The concepts of $q$-cylinder and $q$-homotopy are the bicategorical analogues of Quillen's notions in \cite{Quillen}. The following is the bicategorical equivalent of Quillen's \cite[1, Lemma 2]{Quillen}.

\begin{lemma}\label{A cofibrant da d y c cofibraciones triviales}
Let $X\in\C$ be a cofibrant object,  
and let $C$ be a $q$-cylinder for $X$ as in Definition \ref{def:whomot}. Then 
 both $d_0$ and $d_1$ are trivial cofibrations, and $W$ is a cofibrant object.
\end{lemma}

\begin{proof}
	By definition of $X\amalg X$, the fact that $X$ is cofibrant and axiom {\bf M3}, we have that 
$X \mr{i_0} X \amalg X$ is a cofibration. Then, since 
$d_0 \cong \binom{d_0}{d_1} \ast i_0$, by axiom {\bf M3} $d_0$ is a cofibration. Also, since $s\ast d_0 \cong id_X$, by axiom {\bf M5} 
$d_0$ is a weak equivalence. A similar reasoning shows that $d_1$ is also a trivial cofibration. 
Since $X$ is cofibrant, so is $W$. 
\end{proof} 

\begin{corollary} \label{coro:fibrantwhomotdafibcofibW}
Let $f,g: X \rightarrow Y\in \C$, and $f \mrhpy{H} g$ be a fibrant $q$-homotopy as in Definition \ref{def:whomot}. If $X$ is cofibrant and $Y$ is fibrant, then $W$ is a fibrant-cofibrant object. \qed
\end{corollary}

We {consider} now an equivalence relation between $w$-homotopies (in particular $q$-homotopies) in a way such that its classes will form the 2-cells of a bicategory. This relation {is defined} in \cite{DDS.loc_via_homot}, we recall it now:

\begin{sinnadastandard} \label{sin:varios1}
	We say that an arrow $X \mr{f} Y$ of a bicategory is a {\em quasiequivalence}, a concept weaker than that of equivalence, if for every object $Z$ the functors $\C(Z,X) \mr{f_*} \C(Z,Y)$ and $\C(Y,Z) \mr{f^*} \C(X,Z)$ are full and faithful.

	The $w$-homotopies can be thought of something that would be an actual $2$-cell if the weak equivalences were quasiequivalences. 
	More precisely, given a $w$-homotopy $f \mrhpy{H} g$ as in Definition \ref{def:whomot} for which $s$ is a quasiequivalence, we can associate to it the $2$-cell $f \Xr{\;\;\widehat{H}\;\;} g$ given by the composition 
	\mbox{$f \Mr{\eta} h \ast d_0 \Xr{h \ast \widehat{C}} h \ast d_1 \Mr{\eps} g$,} where $d_0 \Mr{\widehat{C}} d_1$ is the unique $2$-cell such that $s \ast \widehat{C} = {\alpha}$. For any 2-functor $\C \mr{F} \cc{D}$, we denote by $FH$ the structure obtained applying $F$ to each of the components of $H$.
In particular, if $F$ maps all the weak equivalences to quasiequivalences, it follows that there is a 2-cell $\widehat{FH}$ of $\cc{D}$ for each $w$-homotopy $H$.
\end{sinnadastandard}

For any bicategory $\cc{D}$, we denote by $(\C,\cc{W}) \mr{F} (\cc{D},q\Theta)$ a 2-functor which maps the weak equivalences to quasiequivalences\footnote{Working with the notion of quasiequivalence instead of equivalence is what allows us to consider here a 2-functor $F$ instead of a general pseudofunctor, see \cite[2.10]{DDS.loc_via_homot} for details.}.

\begin{definition} \label{def:notacionFyrelationHo}
 We say that two $w$-homotopies (in particular two $q$-homotopies) $\xymatrix{f \ar@2{~>}@<0.25ex>[r]^{H,K} & g}$ are in the same class if $\widehat{FH} = \widehat{FK}$ for all $(\C,\cc{W}) \mr{F} (\cc{D},q\Theta)$ (for all bicategories $\cc{D}$). Note that the construction in \ref{sin:varios1} can be easily dualized for a right $w$-homotopy, and thus this definition  makes sense when one of $H, K$ or both are right $w$-homotopies.
\end{definition}

\begin{remark} \label{rem:soloimportaalphatilde}
	It is the composition ${\alpha} = \alpha_1^{-1} \circ \alpha_0$ which is used in order to determine the class of a $w$-homotopy. Different pairs of $2$-cells $\alpha_0, \alpha_1$ can yield the same 2-cell ${\alpha}$. As an example (which will be relevant later) consider any $w$-homotopy  $\xymatrix{f \ar@2{~>}@<0.25ex>[r]^{H} & g}$ as in Definition~\ref{def:whomot}, and define $w$-cylinders for $X$, 
	$$C_0 = (W,Z,d_0,d_1,s*d_0,s,s*d_0,{\alpha}^{-1}), \quad
	C_1 = (W,Z,d_0,d_1,s*d_1,s,{\alpha},s*d_1),$$
     and  $w$-homotopies 
     $\xymatrix{f \ar@2{~>}@<0.25ex>[r]^{H_i} & g}$,
	$H_i = (C_i,h,\eta,\eps)$, $i=0,1$. In particular $[H] = [H_0] = [H_1]$.
\end{remark}

\begin{sinnadastandard} \label{sin:horiz_comp}
{\bf On the composition of homotopies with arrows.}
Consider arrows $X \mrr{f}{g} Y$ and a $w$-homotopy $H = (C,h,\eta,\eps): f \mrhpy{} g$, as in \eqref{eq:H}. We want to define the composition of $H$ with an arrow $Y \mr{r} Y'$ (on the right) and with an arrow $X' \mr{\ell} X$ (on the left):

$$
\begin{tabular}{cc}
$\vcenter{\xymatrix{X' \ar[r]^{\ell} & X \ar@<-.5ex>[dr]_{x} \ar@<1ex>[r]^{d_0} 
             \ar@<-1ex>[r]_{d_1} & W \ar[r]^h \ar[d]^[@]{\sim}_s & Y \ar[r]^{r} & Y' \\ & & Z}}$ &
\begin{tabular}{c}
$f \Mr{\eta} h \ast d_0$ \\
$s \ast d_0 \Mr{\alpha_0} x \Ml{\alpha_1} s \ast d_1$ \\
$h \ast d_1 \Mr{\eps} g$ 
\end{tabular}
\end{tabular}
$$

Given $Y \mr{r} Y'$ we can define $r * H = (C,r * h,r * \eta,r * \eps): r * f \mrhpy{} r * g$, and it is immediate that $\widehat{F(r \ast H)} = Fr \ast \widehat{FH}$ 
for any $(\C,\cc{W}) \mr{F} (\cc{D},q\Theta)$ as in Definition~\ref{def:notacionFyrelationHo}
(see \cite[Prop. 3.20, 3]{DDS.loc_via_homot} for a proof). Note that $r * H$ has the same $w$-cylinder $C$ as $H$. In particular, 
if $H$ is a $q$-homotopy, then so is $r * H$, and thus this can be used to define a composition $r * [H] = [r * H]$ in $\HoC$, which will be $\cc{H}o_{fc}(\C,\cc{W})
\subset \cc{H}o(\C,\cc{W})$, the full sub-bicategory of the fibrant-cofibrant objects, Definition \ref{sin:defdeHoCcomoHofc} below.

Given $X' \mr{\ell} X$, we can define a $w$-cylinder $C * \ell = (W,d_0 * \ell ,d_1 * \ell ,s,\alpha_0 * \ell ,\alpha_1 * \ell )$ for $X'$, note that $C * \ell$ is not in general a $q$-cylinder, even if $C$ was. 
We can also define a 
$w$-homotopy $H * \ell = (C * \ell, h, \eps * \ell, \eta * \ell)
$ and show that $\widehat{F(H * \ell)} = \widehat{FH} \ast F\ell$ for any $(\C,\cc{W}) \mr{F} (\cc{D},q\Theta)$ (see \cite[Prop. 3.20, 4]{DDS.loc_via_homot}). 
But, if $H$ is a $q$-homotopy it is not clear how to define a $q$-homotopy $H \ast \ell$ satisfying this equation. In the 1-dimensional case \cite{Quillen}, this is solved using the dual notion of right homotopy, but here we take a different approach. In \S \ref{sub:relating} below we will show that, for fibrant-cofibrant objects, any $w$-homotopy admits a $q$-homotopy in the same class. Thus this is the way in which the composition $[H] * \ell$ in $\HoC$ is defined, as a $q$-homotopy in the same class of $H * \ell$.
\end{sinnadastandard}

\begin{sinnadastandard} \label{sin:varios2} {\bf On the bicategory of $w$-homotopies.}
For an arbitrary pair $(\C,\cc{W})$ as in Definition \ref{def:whomot}, the definition in \ref{sin:varios1} can be extended to finite sequences of composable $w$-homotopies 
$\xymatrix{f \ar@2{~>}@<0.25ex>[r]^{H^1} & f_1 \ar@2{~>}@<0.25ex>[r]^>>>>>>{H^2} & f_2 \dotsb f_{n-1} \ar@2{~>}@<0.25ex>[r]^>>>>>>{H^n} & g}$
and in this way a bicategory $\Wo$ is defined. Its objects and arrows are those of $\C$, the 2-cells are the classes of finite sequences of composable $w$-homotopies, 
{$[H^n,\,\ldots \,,H^2,\, H^1]$}.
Horizontal composition (whiskering) is defined as in \ref{sin:horiz_comp} 
and vertical composition is given by juxtaposition.
Together with $\Wo$ there is a {\em projection} 2-functor $\C \mr{i} \Wo$, which is the identity on objects and arrows and maps a 2-cell $\mu$ of $\C$ to the class of a $w$-homotopy $I^{\mu}$ which satisfies that $\widehat{FI^\mu} = F\mu$ for any $(\C,\cc{W}) \mr{F} (\cc{D},q\Theta)$. 
We refer {the interested reader} to \cite{DDS.loc_via_homot} for more details on this construction. 
\end{sinnadastandard} 

{ 
We mention that $I^\mu$ can be chosen to be either one of the following two $w$-homotopies:
\begin{center}               
\begin{tabular}{lc}
$H^{\mu}_0: \vcenter{\xymatrix{X \ar@<-.5ex>[dr]_{id_X} \ar@<1ex>[r]^{id_X} 
             \ar@<-1ex>[r]_{id_X} & X \ar[r]^g \ar[d]^{id_X} & Y \\ & X}} \quad\quad$ & 
             $H^{\mu}_1: \vcenter{\xymatrix{X \ar@<-.5ex>[dr]_{id_X} \ar@<1ex>[r]^{id_X} 
             \ar@<-1ex>[r]_{id_X} & X \ar[r]^f \ar[d]^{id_X} & Y \\ & X}} \quad\quad$ \\
$\eta = \mu$, $\alpha_0 = \alpha_1 = id_X$, $\eps = g \quad\quad$   
&
$\quad\quad \eta = f$, $\alpha_0 = \alpha_1 = id_X$, $\eps = \mu$
\end{tabular}
\end{center}

\begin{remark} \label{choose_I^mu}
It clearly follows that $I^{\mu}$ can be choosen to be a fibrant 
$w$-homotopy, such that if $\mu$ is invertible, $I^\mu$ has invertible cells.
\end{remark} 
}

A dual construction allows to define a right $w$-homotopy $I^\mu$ such that $\widehat{FI^\mu} = F\mu$ for any $(\C,\cc{W}) \mr{F} (\cc{D},q\Theta)$. It follows:

\begin{remark}  \label{Imuleft=Ymurigth}
{The $w$-homotopy $I^\mu$ can be considered to be a left or a right $w$-homotopy as needed.} 
\end{remark}

\begin{remark} \label{def:construccionK}
Any $w$-cylinder $C$ as in Definition \ref{def:whomot} determines a $w$-homotopy with invertible cells:
\begin{center}
\begin{tabular}{cc}
$d_0 \mrhpy{H^{C}} d_1$: \hspace{3ex}  
$\vcenter{\xymatrix{X \ar[rd]_x \ar@<1ex>[r]^{d_0} 
            \ar@<-1ex>[r]_{d_1} & W \ar[r]^{id_W} \ar[d]^[@]{\sim}_s & W \\ & Z}}$ &
\begin{tabular}{c}
$d_0 \Xr{d_0} d_0$ \\
$s \ast d_0 \Mr{\alpha_0} x \Ml{\alpha_1} s \ast d_1$ \\
$d_1 \Xr{d_1} d_1$ 
\end{tabular}
\end{tabular}
\end{center} 
\end{remark} 

{
\begin{remark} \label{2compoH}
[Composing homotopies with 2-cells]
Consider a $w$-homotopy 
$f \mrhpy{H} g$ as in \eqref{eq:H}, and 2-cells
$f' \Mr{\nu} f$, $g \Mr{\mu} g'$. Simply composing $\nu$ with $\eta$, and $\mu$ with $\eps$, yields a new $w$-homotopy 
$f' \mrhpy{H'} g'$, $H' = \mu \circ H \circ \nu$, with the same cylinder as $H$, satisfying 
$[H'] = [I^\mu] \circ [H] \circ [I^\nu] 
                               = [I^\mu,\, H,\, I^\nu]$.  
Details can be found in \cite[3.6-3.7]{DDS.loc_via_homot} if needed.
\end{remark}
}

We recall now {a basic decomposition result for $w$-homotopies} proved in  \cite{DDS.loc_via_homot} that we will use later.

 \begin{proposition} \label{prop:Hcompde3} 
 (\cite[Proposition 3.31]{DDS.loc_via_homot})
Let $H$ be any $w$-homotopy as in Definition \ref{def:whomot}. Then $[H]$ can be decomposed as:
$$
[H] = [\eps \circ (h * H^C) \circ \eta] = 
[I^\eps] \circ [h * H^C] \circ [I^\eta] = 
[I^\eps,\, h * H^C,\, I^\eta].
$$

\vspace{-4ex}

\cqd
\end{proposition}

\begin{proposition} \label{invertibleH}
(\cite[Corollary 3.33]{DDS.loc_via_homot})
The class $[H]$ of any $w$-homotopy with invertible cells is an invertible $2$-cell in $\Ho$. 
\cqd
\end{proposition}

\begin{lemma}\label{2-cell da homotopia posta}
	Let $X \cellrdE{f}{\mu}{g} Y\in\C$ be a 2-cell and $C = (W,d_0,d_1,s,\alpha_0,\alpha_1)$ be a \mbox{$q$-cylinder} for $X$. Then we can construct $\homotopy{f}{I^\mu}{g}$ with cylinder $C$; moreover, if $\mu$ is invertible then $I^\mu$ has invertible cells. 
\end{lemma}
\begin{proof}
Let \mbox{$H=(C,f * s, f*\alpha_0^{-1}, \mu \circ f*\alpha_1)$.} To verify $H = I^\mu$, let $\C \mr{F} \cc{D}$ be a 2-functor that maps the weak equivalences to equivalences, then $\widehat{FH}$ is the composition
$$Ff \Xr{Ff * F\alpha_0^{-1}} Ff * Fs * Fd_0  \Xr{Ff * Fs * \widehat{FC}} Ff * Fs * Fd_1 \Xr{Ff * F\alpha_1} Ff \Xr{F\mu} Fg.$$

Since by definition we have $Fs * \widehat{FC} = F\alpha_1^{-1} \circ F\alpha_0$, the composition clearly reduces to $F\mu$ as desired.
\end{proof}

For an arbitrary pair $(\C,\cc{W})$, $i$ is not the localization of $\C$ at $\cc{W}$, since 
$i(f)$ will not be in general an equivalence for each weak equivalence. However $i(f)$ will always satisfy the ``faithful" part in the definition of quasiequivalence:

\begin{proposition} \label{prop:siemprefaithfulenHoC}
	For any weak equivalence $X \mr{f} Y$ and any object $Z$, the functors $\Wo(Z,X) \mr{f_*} \Wo(Z,Y)$ and $\Wo(Y,Z) \mr{f^*} \Wo(X,Z)$ \mbox{are faithful.}
\end{proposition}

\begin{proof} 
We deal with $f_*$ first, we consider arrows $g,\,h$ and 2-cells 
$\alpha,\,\beta$ as follows $Z \cellpairrd{g}{\alpha}{\beta}{h} X$, which are determined by sequences of composable $w$-homotopies, 
\mbox{$\alpha = [H^n,\,\ldots \,,H^2,\, H^1]$,} $\beta = [K^n,\,\ldots \,,K^2,\, K^1]$.
We have to show that $f * \alpha = f * \beta$ implies $\alpha = \beta$.  
We note that by definition $f * \alpha = f * \beta$ means that for any $(\C,\cc{W}) \mr{F} (\cc{D},q\Theta)$ as in Definition~\ref{def:notacionFyrelationHo} we have
$$Ff * (
\widehat{FH_n} \circ ... \circ \widehat{FH_2} \circ \widehat{FH_1})
=
Ff * (
\widehat{FK_m} \circ ... \circ \widehat{FK_2} \circ \widehat{FK_1}).
$$
Since $Ff$ is a quasiequivalence, the equality 
$\widehat{FH_n} \circ ... \circ \widehat{FH_2} \circ \widehat{FH_1} = 
\widehat{FK_m} \circ ... \circ \widehat{FK_2} \circ \widehat{FK_1}$ holds for arbitrary $(\C,\cc{W}) \mr{F} (\cc{D},q\Theta)$, that is $\alpha = \beta$. The case of  $f^*$ is dual.
\end{proof}

\subsection{Vertical composition of $q$-homotopies} 
\label{sub:compvert}

A situation in which $w$-homotopies can be vertically composed was considered in \mbox{\cite[Appendix B]{DDS.loc_via_homot}:}

\begin{lemma} \label{lema:fittingparacompvert}
Assume that we have 
$
\xymatrix@C=7ex{X
\ar[r]^{f_1,f_2,f_3} & Y}$, and $w$-homotopies 
\mbox{$f_1 \mrhpy{H^1} f_2 \mrhpy{H^2} f_3$} as in Definition~\ref{def:whomot}, with 
$Z^1 = Z^2 = Z$, $x^1 = x^2 = x$
fitting in the following diagram, where 
{$\nu^1,\, \nu^2$, $\gamma^1$, $\gamma^2$} are invertible 2-cells, and {$s,s^1,s^2$ are weak equivalences}: 
             
\begin{equation} \label{diagram:fittingparacompvert}
\xymatrix@C=8ex{& W^1 \ar[rd]_{b^1} \ar@/^2ex/[rrd]^{s^1}_>>>>>>>>>>{\cong \, \nu^1} \ar@/^4ex/[rrrd]^{h^1}_>>>>>>>>>>>>>{\cong \,  \gamma^1} \\
X \ar[ru]^{d_1^1} \ar[rd]_{d^2_0} \ar@{}[rr]|{\Downarrow \delta} && W \ar[r]^{s} \ar@/_3ex/[rr]^>>>>>>>{h} & Z & Y \\
& W^2 \ar[ru]^{b^2} \ar@/_2ex/[rru]_{s^2}^>>>>>>>>>>{\cong \,  \nu^2} \ar@/_4ex/[rrru]_{h^2}^>>>>>>>>>>>>>{\cong \,  \gamma^2}}
\end{equation}

Assume also that:

\smallskip

The 2-cell $h^1 \ast d_1^1 \Mr{\eps^1} f_2 \Mr{\eta^2} h^2 \ast d^2_0$ equals 
$h^1 \ast d_1^1 \Xr{\gamma^1 \ast d_1^1} h \ast b^1 \ast d_1^1 \Xr{h \ast \delta} h \ast b^2 \ast d^2_0 \Xr{\gamma^2 \ast d^2_0} h^2 \ast d^2_0,$

\smallskip

The 2-cell 
$s^1 \ast d_1^1 \Xr{(\alpha^1_1)} x \Xr{(\alpha^2_0)^{-1} \!\!\!\!\!\!\!\!} s^2 \ast d^2_0$ equals 
$s^1 \ast d_1^1 \Xr{\nu^1 \ast d_1^1} s \ast b^1 \ast d_1^1 \Xr{s \ast \delta} s \ast b^2 \ast d^2_0 \Xr{\nu^2 \ast d^2_0} s^2 \ast d^2_0.$

\vspace{2ex}

Then we can construct a $w$-homotopy $H$ from $f_1$ to $f_3$ such that $[H] = [H^2,H^1]$.

\medskip

Furthermore, $H$ can be constructed as follows.
Take the $w$-cylinder \mbox{$C = (W,Z,b^1 * d^1_0, b^2 * d_1^2, x, s, \alpha_0, \alpha_1)$,} with $\alpha_0$ and $\alpha_1$ defined as the compositions
$$\alpha_0: s * b^1 * d^1_0 \Xr{\nu^1 * d^1_0} s^1 * d^1_0 \Xr{\alpha^1_0} x, \quad \quad 
\alpha_1: s * b^2 * d_1^2 \Xr{\nu^2 * d_1^2} 
s^2 * d_1^2 \Xr{\alpha^2_1} x 
$$
Then $H$ is given by $H = (C,h,\eta,\eps)$, with
$\eta$ and $\eps$ defined as the compositions
$$\eta: f_1 \Mr{\eta^1} h^1 \ast d^1_0 \Xr{\gamma^1 \ast d^1_0} h \ast b^1 \ast d^1_0, \quad \quad \eps: h \ast b^2 \ast d_1^2 \Xr{\gamma^2 \ast d_1^2} h^2 \ast d_1^2 \Mr{\eps^2} f_3.$$ 

\vspace{-4ex}

\qed
\end{lemma}

We will compose $w$-homotopies with cylinders with the same 
$Z$ and $x$ (in particular $q$-homotopies) by showing that they can be made to fit this situation, and in the case of $q$-homotopies, that the resulting $w$-homotopy is in fact a $q$-homotopy. 
This is a generalization of Quillen's construction of  \cite[Lemmas 3,4]{Quillen} to bicategories, which will be possible if we assume a further condition in the definition of model bicategory, namely the additional axioms (recall \ref{sin:limits}):

\vspace{1ex}

\noindent {\bf MM0.} $\cc{C}$ has lax-biPullbacks and lax-biPushouts.
\vspace{1ex}

\noindent {\bf MM3.} Fibrations (respectively cofibrations) are closed under lax-biPullbacks (respectively lax-biPushouts).

\vspace{1ex}

\noindent {\bf MM4.} If a morphism $f$ is the lax-biPullback (resp. lax-biPushout) of a fibration (resp. cofibration) which is also a weak equivalence, then $f$ is a weak equivalence.
\vspace{1ex}

Note that these axioms hold automatically if every 2-cell of $\cc{C}$ is invertible. Also note that, since we are working with left homotopies, we will use only the lax-biPushout case of these axioms.
\begin{lemma} \label{lemacompo}
	Given $X\in\C$ cofibrant, arrows $
	\xymatrix
	{
		X
		\ar@<2.2ex>[r]^{f_1} 
		\ar[r]^{f_2}
		\ar@<-2.2ex>[r]^{f_3} 
		& 
		Y
	}\in\C$    
	and \mbox{$w$-homotopies} 
	$
	\xymatrix
	{
		f_1 \ar@2{~>}@<0.25ex>[r]^{H^1} 
		& 
		f_2 \ar@2{~>}@<0.25ex>[r]^{H^2} 
		& f_3
	}
	$,
with $Z^1 = Z^2 = Z$, $x^1 = x^2 = x$, 
there exists a $w$-homotopy $H$ from $f_1$ to $f_3$ such that $[H] = [H^2,H^1]$. If $H^1$ and $H^2$ are $q$-homotopies, so is 
$H$. 

Note that this means that for any 2-functor $\C \mr{F} \cc{D}$ that maps the weak equivalences to quasiequivalences, $\widehat{FH} = \widehat{FH^2} \circ \widehat{FH^1}$.
\end{lemma}
\begin{proof}              
	We construct the diagram \eqref{diagram:fittingparacompvert} in Lemma \ref{lema:fittingparacompvert} as follows: Let $(W,b^1,b^2,\delta)$ be the lax-biPushout of $d^1_1$ and $d^2_0$. Then $s$, $\nu^1$ and $\nu^2$ are induced by the 2-cell 
	\mbox{$s^1 \ast d^1_1 \Xr{\alpha^1_1} id_X \Xr{(\alpha^2_0)^{-1}} s^2 \ast d^2_0$,} and $h$, $\gamma^1$ and $\gamma^2$ are induced by the 2-cell \mbox{$h^1 \ast d_1^1 \Mr{\eps^1} f_2 \Mr{\eta^2} h^2 \ast d_0^2$.} By construction all the hypothesis of the Lemma are satisfied (the arrow $s$ is a weak equivalence since $b^1$ is so by axiom {\bf MM4} and Lemma \ref{A cofibrant da d y c cofibraciones triviales}). 
Thus, we have a $w$-homotopy $H$ such that $[H] = [H^1,H^2]$. Note that, by the construction of $H$, in order to conclude that $H$ is a $q$-homotopy when $H^1$ and $H^2$ are so, 
 it only remains to prove that $\binom{d_0}{d_1}$ is a cofibration. We consider the diagram

$$\xymatrix{
X \amalg X \ar@<2ex>@/^4ex/[rrrr]^{\binom{d_0}{d_1}} \ar[rr]^{d_0^1 \amalg id_X} \ar@{}[rd]|{\Downarrow} && W^1 \amalg X \ar[rr]^{\binom{b^1}{d_1}} & \ar@{}[rd]|{\Downarrow} & W \\
X \ar[u]_{i_0} \ar@{^{(}->}[r]_{d_0^1}^\sim & W^1 \ar[ru]^{i_0} && X \amalg X \ar[lu]_{d_1^1 \amalg id_X} \ar@{^{(}->}[r]_{\binom{d_0^2}{d_1^2}}& W^2 \ar[u]_{b^2}
}$$
	
\noindent in which the upper triangle commutes (up to isomorphism) by the definition $d_0 = b^1 * d_0^1$, the left square is a biPushout and the right square is a lax-biPushout. Using axioms {\bf M3} and {\bf MM3}, it follows that the top horizontal arrows in the squares are cofibrations and thus so is $\binom{d_0}{d_1}$.
\end{proof}

\begin{remark}
A different {proof of the previous proposition, under the extra assumption that $Y$ is fibrant but without using axiom {\bf MM3}, is also possible as follows. 
Proceed as above to construct the $w$-homotopy $H$, and then use Proposition~\ref{2-cellsWofc} below.} However, we preferred to give this proof because {we consider the} vertical composition of \mbox{$q$-homotopies} {to be} independent of the results which lead to Proposition \ref{2-cellsWofc} (which assume Y fibrant).
\end{remark}

\subsection{Relating w-homotopies and q-homotopies} \label{sub:relating}

\begin{sinnadastandard} \label{sin:strategy}
In this subsection we will show that, for arrows with fibrant codomain, there is a fibrant $q$-homotopy in the same class of any arbitrary $w$-homotopy.
Our strategy for doing this will consist on finding for any $w$-cylinder $C$ a fibrant $q$-cylinder $C'$ which is {\em linked} to $C$ by a zig-zag of cylinder morphisms, and using that in this case a $w$-homotopy $H$ (with respect to $C$) determines a $q$-homotopy $H'$ (with respect to $C'$) in the same class. 
\end{sinnadastandard}

\begin{definition} \label{defmorphcyl}
A morphism $C^1 \mr{M} C^2$ between $w$-cylinders (with the notation of Definition~\ref{def:whomot}) is given by the data $M = (k,\,\ell,\,\gamma_0, \gamma_1, \mu, \nu)$, where $k, \ell$ are arrows fitting in the following {\em prism} diagram
{
$$
\xymatrix
     {
      & 
      & W^1 \ar[dd]^(.4){s^1}|(.4)\circ
          \ar[drr]^{k} 
     \\      
        X \ar@<10pt>[urr]^{d_0^1}
          \ar@<-3pt>[urr]^{d_1^1}
          \ar@/_1pc/@<3pt>[rrrr]^(.7){d_0^2}
          \ar@/_1pc/@<-10pt>[rrrr]^(.7){d_1^2}
          \ar@{=}[dd] &&&& W^2\ar[dd]^(.6){s^2}|(.6)\circ 
     \\
      &
      & Z^1 \ar[drr]^{\ell}
     \\
        X  \ar[urr]^{x^1} 
           \ar@/_1pc/[rrrr]^{x^2}
      &
      &
      &
      & Z^2
    }
$$
}
\vspace{1ex}

\noindent and $\gamma_0, \gamma_1, \mu$ and $\nu$ are invertible 2-cells filling the faces as indicated in the diagrams below, and satisfying the two underlying equalities:
\begin{equation}  \label{eq:hiplemapara3pasos}
\vcenter
 {
  \xymatrix@C=4pc@R=1pc
     {
      X \ar@<1.5ex>[r]^{d^1_0}
        \ar@<.5ex>[r]_{d^1_1}
        \ar@<-0ex>@/_2ex/[ddr]^>>>>{d^2_0}
        \ar@<-1ex>@/_2ex/[ddr]_>>>>{d^2_1}
        \ar@<2ex>@/^3.5ex/[rr]^{x^1}
        \ar@<3.5ex>@{}[rr]|{\alpha_0^1 \Uparrow 
                                      \, \alpha_1^1 \, \Uparrow}
     & 
      W^1 \ar[dd]^{k} 
          \ar[r]|{\circ}^{s^1} 
          \ar@{}[ddr]|{\mu \, \Uparrow} 
     & 
      Z^1 \ar[dd]^{\ell} 
    \\
      \ar@{}[r]^{\gamma_0 \Uparrow \, \gamma_1 \, \Uparrow}
     & 
    \\
    & 
      W^2 \ar[r]|{\circ}^{s^2} & Z^2
     }
 } 
\quad = \quad 
\vcenter
 {
  \xymatrix@C=4pc@R=1pc
     {
      X \ar@{}@<3ex>[rr]|{\nu \, \Uparrow}
        \ar@/^3ex/[rrdd]^{x^2} 
        \ar@<-0ex>@/_2ex/[ddr]^{d^2_0}
        \ar@<-1ex>@/_2ex/[ddr]_{d^2_1} 
        \ar@<2ex>@/^3.5ex/[rr]^{x^1}
        \ar@{}[rrdd]|{\alpha_0^2 \Uparrow \, \alpha_1^2 \, \Uparrow}
    & 
    & 
     Z^1 \ar[dd]^{\ell} 
   \\
    & 
   \\
    & 
     W^2 \ar[r]|{\circ}^{s^2} & Z^2
    }
 } 
\end{equation}
{
That is:

 $s^2*d_0^2 \Xr{\gamma_0} s^2*k*d_0^1 \Xr{\mu * d_0^1} 
        \ell* s^1 *d_0^1 \Xr{\ell*\alpha_0^1} \ell*x^1$
$\;\;=\;\;$
 $s^2*d_0^2 \Xr{\alpha_0^2} x^2 \Xr{\nu} \ell*x^1$. 
 
\vspace{1ex}

 $s^2*d_1^2 \Xr{\gamma_1} s^2*k*d_1^1 \Xr{\mu * d_1^1} 
        \ell* s^1 *d_1^1 \Xr{\ell*\alpha_1^1} \ell*x^1$
$\;\;=\;\;$ $s^2*d_1^2 \Xr{\alpha_1^2} x^2 \Xr{\nu} \ell*x^1$.
}
\end{definition}

{
\begin{remark} \label{rem:alphasencorresp} 
Clearly in Definition \ref{defmorphcyl} the 2-cells $\alpha_0^2$ and $\alpha_1^2$ can be obtained from the 2-cells  
$\alpha_0^1$ and $\alpha_1^1$ respectively, in such a way that the conditions in the definition hold:

\begin{enumerate}

\item (transfer cylinder forward) Given all the data $C^1 \mr{M} C^2$ in Definition \ref{defmorphcyl}, but with $C^2$ lacking the invertible 2-cells $\alpha^2_0$ and $\alpha^2_1$, these 
can be uniquely defined in order to complete a cylinder $C^2$ in such a way that $M$ becomes a morphism of cylinders.

\item If $\ell$ is a quasiequivalence, the same happens the other way around:

(transfer cylinder backwards) Given all the data $C^1 \mr{M} C^2$ in Definition \ref{defmorphcyl}, but with $C^1$ lacking the invertible 2-cells $\alpha^1_0$ and $\alpha^1_1$, these 
can be uniquely defined in order to complete a cylinder $C^1$ in such a way that $M$ becomes a morphism of cylinders.
\end{enumerate}
\end{remark}
}
{
The following is a bicategorical version of the \emph{germ} relation between homotopies considered in 
\cite[Appendix A]{DDS.loc_via_homot}  and \cite{JaG}.
}

{
\begin{definition} \label{defgermrelation} 
Let $f \mrhpy{H^1} g$, $f \mrhpy{H^2} g$ be $w$-homotopies (with the notation of Definition~\ref {def:whomot}) with $w$-cylinders $C^1$, $C^2$ respectively. We say that $H^1$ is \emph{germ related} to $H^2$,  denoted $H^1 \germ > H^2$, if there exist a morphism 
$C^1 \mr{M} C^2$ as in Definition~\ref{defmorphcyl} and 
an invertible 2-cell $h^2 \ast k \Mr{\rho} h^1$ such that 
\begin{enumerate}
\item \label{(2)} $\eta^1$ equals the composition 
$f \Xr{\eta^2} h^2 \ast d^2_0 \Xr{h^2 \ast \gamma_0} h^2 \ast k \ast d^1_0 \Xr{\rho \ast d^1_0} h^1 \ast d^1_0$,
\item $\eps^1$ \label{(3)} equals the composition 
$h^1 \ast d^1_1 \Xr{\rho^{-1} \ast d^1_1} h^2 \ast k \ast d^1_1 \Xr{h^2 \ast \gamma_1^{-1}} h^2 \ast d^2_1 \Xr{\eps^2} g$. 
\end{enumerate}
\end{definition}
}
{
\begin{remark} \label{rem:etaepsencorresp}
Clearly in Definition \ref{defgermrelation} the 2-cells $\eta^1$ and $\eps^1$ can be obtained from the 2-cells  $\eta^2$ and $\eps^2$ respectively, in such a way that the conditions in the definition hold.  Also, the same happens the other way around. Thus:
\begin{enumerate}
\item (transfer homotopy forward) Given all the data in Definition \ref{defgermrelation} but with $H^2$ lacking the 2-cells $\eta^2$ and $\eps^2$, these 
can be uniquely defined in order to complete a $w$-homotopy $H^2$ in such a way that $H^1 \germ > H^2$.

\item (transfer homotopy backwards) Given all the data in Definition \ref{defgermrelation} but with $H^1$ lacking the 2-cells $\eta^1$ and $\eps^1$, these 
can be uniquely defined in order to complete a $w$-homotopy $H^1$ in such a way that $H^1 \germ > H^2$.
\end{enumerate}
\end{remark}
}
{
\begin{lemma} \label{lema:para3pasos}
With the notation of Definition \ref{defgermrelation}, if $H^1 \germ > H^2$, then 
$[H^1] = [H^2]$. Thus, the same holds for the generated equivalence relation $\germ$.  
\end{lemma}
} 
\begin{proof}
Recalling Definition \ref{def:notacionFyrelationHo}, let $(\C,\cc{W}) \mr{F} (\cc{D},q\Theta)$, and let $Fd^i_0 \Xr{\widehat{FC^i}} Fd^i_1$ be the 2-cell such that $Fs^i \ast \widehat{FC^i} = F{\alpha^i}$, $i=1,2$. From Definition \ref{defmorphcyl}, it follows that ${\alpha^2}$ equals the composition
\begin{equation} \label{eq:alpha2equals}
s^2 \ast d^2_0 \Xr{s^2 \ast \gamma_0} s^2 \ast k \ast d^1_0 \Xr{\mu \ast d^1_0} \ell \ast s^1 \ast d^1_0 \Xr{\ell \ast  {\alpha^1}} \ell \ast s^1 \ast d^1_1 \Xr{\mu^{-1} \ast d^1_1} s^2 \ast k \ast d^1_1 \Xr{s^2 \ast \gamma_1^{-1}} s^2 \ast d^2_1.
\end{equation}
We start by showing that the 2-cell $\widehat{FC^2}$ equals the composition 
$$Fd^2_0 \Xr{F\gamma_0} Fk \ast Fd^1_0 \Xr{Fk \ast \widehat{FC^1}} Fk \ast Fd^1_1 \Xr{F\gamma_1^{-1}} Fd^2_1.$$
By applying $Fs^2 \ast (-)$ to this composition, and comparing it with the value of the 2-functor F at the composition in \eqref{eq:alpha2equals} above, it follows that it suffices to show that $F(s^2 \ast k) \ast \widehat{FC^1}$ equals the composition 
$$F(s^2 \ast k \ast d^1_0) \Xr{F(\mu \ast d^1_0)} F(\ell \ast s^1 \ast d^1_0) \Xr{F(\ell \ast {\alpha^1})} F(\ell \ast s^1 \ast d^1_1) \Xr{F(\mu^{-1} \ast d^1_1)} F(s^2 \ast k \ast d^1_1).$$ This followws from the interchange law applied to the configuration
$\xymatrix@C=8ex{FX \ar@<1.6ex>[r]^{Fd^1_0} 
             \ar@{}@<-1.3ex>[r]^{\!\! {\widehat{FC^1}} \, \!\Downarrow}
             \ar@<-1.1ex>[r]_{Fd^1_1} & Fd^1_1      
\ar@<1.6ex>[r]^{F(\ell \ast s^1)} 
             \ar@{}@<-1.3ex>[r]^{\!\! {F\mu^{-1}} \, \!\Downarrow}
             \ar@<-1.1ex>[r]_{F(s^2 \ast k)} & FZ}$.              
Therefore, the 2-cell $\widehat{F H^2}$ is the composition
$$Ff \Mr{F\eta^2} Fh^2 \ast Fd^2_0 \Xr{\!\! Fh^2 \ast F\gamma_0 \!\!} Fh^2 \ast Fk \ast Fd^1_0 \Xr{\!\! F(h^2 \ast k) \ast \widehat{FC^1} \!\!} Fh^2 \ast Fk \ast Fd^1_1 \Xr{\!\! Fh^2 \ast F\gamma_1^{-1} \!\!} Fh^2 \ast Fd^2_1 \Mr{F \eps^2} Fg.$$
By looking at the {expression} of the 2-cells $\eta^1$ and $\eps^1$ in the hypothesis of the lemma, it follows that in order to show $\widehat{F H^2} = \widehat{F H^1}$ it suffices to show that $F(h^2 \ast k) \ast \widehat{FC^1}$ equals the composition 

\mbox{$Fh^2 \ast Fk \ast Fd^1_0 \Xr{F\rho \ast Fd^1_0} Fh^1 \ast Fd^1_0 \Xr{Fh^1 \ast \widehat{FC^1}} Fh^1 \ast Fd^1_1 \Xr{F\rho^{-1} \ast Fd^1_1} Fh^2 \ast Fk \ast Fd^1_1$,} 

\vspace{1ex}

\noindent which follows from the interchange law applied to the configuration
$\xymatrix@C=8ex{FX \ar@<1.6ex>[r]^{Fd^1_0} 
             \ar@{}@<-1.3ex>[r]^{\!\! {\widehat{FC^1}} \, \!\Downarrow}
             \ar@<-1.1ex>[r]_{Fd^1_1} & Fd^1_1      
\ar@<1.6ex>[r]^{Fh^1} 
             \ar@{}@<-1.3ex>[r]^{\!\! {F\rho^{-1}} \, \!\Downarrow}
             \ar@<-1.1ex>[r]_{F(h^2 \ast k)} & FY}$. 
\end{proof}

If $H^1$, $H^2$ in the previous lemma have the same $w$-cylinder $C = C^1 = C^2$, we may take $M$ to be the evident identity morphism and we have the following

\begin{corollary}
\label{coro:lemapara3pasosconC1igualC2}
Let $f \mrhpy{H^1} g$, $f \mrhpy{H^2} g$ be $w$-homotopies (with the notation as in Definition 
\ref {def:whomot}) with the same $w$-cylinder $C$.
Assume there exists an invertible 2-cell \mbox{$h^2 \Mr{\rho} h^1$} such that 
\begin{enumerate}
\item \label{(2bis)} $\eta^1$ equals the composition 
$f \Xr{\eta^2} h^2 \ast d^2 \Xr{\rho \ast d^1} h^1 \ast d^1$.
\item $\eps^1$ \label{(3bis)} equals the composition 
$h^1 \ast c^1 \Xr{\rho^{-1} \ast c^1} h^2 \ast c^2 \Xr{\eps^2} g$. 
\end{enumerate}
Then
 $[H^1] = [H^2]$. \qed
\end{corollary}

We develop now the strategy envisaged in \ref{sin:strategy}. 
The following lemmas lead to Proposition \ref{2-cellsWofc}, which states that for arrows with fibrant codomain, each $w$-homotopy has an associated fibrant $q$-homotopy in the same class.

\begin{lemma} \label{lema0}
	Let $\xymatrix{X \ar@<1ex>[r]^f \ar@<-1ex>[r]^g & Y}\in\C$, with $Y$ a fibrant object. 
	Let $\xymatrix{f \ar@2{~>}@<0.25ex>[r]^{H} & g}$ be a $w$-homotopy as in Definition \ref{def:whomot}. Then, there is a 
	$w$-homotopy $\xymatrix{f \ar@2{~>}@<0.25ex>[r]^{H'} & g}$ with $W'$ and $Z'$ fibrant objects, such that 
	{$H \germ > H'$.}  
\end{lemma}
\begin{proof}
 First we factorize using axiom {\bf M2}:
	$$\vcenter{\xymatrix{Z \ar[rr] \ar@{^{(}->}[rd]_\ell
			\ar@<-1ex>@{}[rd]^[@]{\sim}  && 1 \\
			& Z' \ar@{->>}[ru]
			\ar@{}[u]|>>>>{\cong}}}, \qquad \qquad
\vcenter{\xymatrix{W \ar[rr]^{h} \ar@{^{(}->}[rd]_k
			\ar@<-1ex>@{}[rd]^[@]{\sim}  && Y \\
			& W' \ar@{->>}[ru]_{h'}
			\ar@{}[u]|>>>>{\rho \ {\Uparrow} \ \cong}}}$$
	Then we use axioms {\bf M1} and {\bf M5} in order to have 
	$$\vcenter{\xymatrix@C=1.2pc@R=1.2pc{W \ar[rr]^{\ell * s}_{\sim} \ar@{^{(}->}[dd]_k
			\ar@<.3ex>@{}[dd]^[@]{\sim} 
			&& Z' \ar@{->>}[dd] \\ 
			\ar@{}[ru]|{\mu \, \Uparrow \, \cong} & \\
			W' \ar@{.>}[rruu]_{s'}
			\ar@<-1ex>@{}[rruu]^[@]{\sim}
			\ar[rr] && 1}}$$  
{ and we set $x' = \ell * x$. 
In this way we have a cylinder $C'$ except for the 2-cells 
$\alpha'_0$, $\alpha'_1$, and all the data of a morphism 
\mbox{$M = (k,\,\ell,\, id,\,id,\, \mu,\, id): C \mr{} C'$} as follows (compare with diagram (\ref{eq:hiplemapara3pasos})):
} 
$$\vcenter{\xymatrix@C=4pc@R=1pc{
			X \ar@<1.5ex>[r]^{d _0}
			\ar@<.5ex>[r]_{d _1}
			\ar@<-0ex>@/_2ex/[ddr]^>>>>{d'_0}
			\ar@<-1ex>@/_2ex/[ddr]_>>>>{d'_1}
			\ar@<2ex>@/^3.5ex/[rr]^{x }
			\ar@<3.5ex>@{}[rr]|{\alpha_0  \Uparrow \, \alpha_1  \, \Uparrow}
			& W  \ar[dd]^{k} \ar[r]_\sim^{s } \ar@{}[ddr]|{\mu \, \Uparrow} & Z  \ar[dd]^{\ell} \\
			\ar@{}[r]^{id \Uparrow \, id \, \Uparrow}
			& \\
			& W' \ar[r]_\sim^{s'} & Z'}} 
	\quad \; = \;\quad 
	\vcenter{\xymatrix@C=4pc@R=1pc{
			X \ar@{}@<3ex>[rr]|{id \, \Uparrow}
			\ar@/^3ex/[rrdd]^{x'} 
			\ar@<-0ex>@/_2ex/[ddr]^{d'_0}
			\ar@<-1ex>@/_2ex/[ddr]_{d'_1} 
			\ar@<2ex>@/^3.5ex/[rr]^{x }
			& & Z  \ar[dd]^{\ell} \\
			& \\
			& W' \ar[r]_\sim^{s'} & Z'}}$$

\vspace{1ex}

{We can then {\em transfer forward} $C$ and $H$ to get the desired $C'$ and $H'$:
first, by item 1 in \mbox{Remark \ref{rem:alphasencorresp}}, we complete $C'$ with 2-cells $\alpha'_0$, $\alpha'_1$.			
In turn, since we have an invertible 2-cell $h'*k \Mr{\rho} h$, the proof finishes by item 1 in \mbox{Remark \ref{rem:etaepsencorresp}.}
}			
\end{proof}

\begin{lemma} \label{lema:3pasos}
	Let $\xymatrix{f \ar@2{~>}@<0.25ex>[r]^{H} & g}$ be a $w$-homotopy as in Definition \ref{def:whomot}, with $W$ a fibrant object. Then, there is a 
	fibrant $w$-homotopy $\xymatrix{f \ar@2{~>}@<0.25ex>[r]^{H'} & g}$ such that {$H \germ > H'$.}
\end{lemma}
\begin{proof}
The strategy of this proof is similar to the one of the previous lemma.
	First we factorize 
	$\vcenter{\xymatrix{W \ar[rr]_\sim^s \ar@{^{(}->}[rd]_k
			\ar@<-1ex>@{}[rd]^[@]{\sim}  && Z \\
			& W' \ar@{->>}[ru]_{s'}
			\ar@<-1ex>@{}[ru]^[@]\sim
			\ar@{}[u]|>>>>{\mu \ \Uparrow \ \cong}}}$ using axioms {\bf M2} and {\bf M5}, 
{and we set $x' = x$.

In this way we have a cylinder $C'$ except for the 2-cells 
$\alpha'_0$, $\alpha'_1$, and all the data of a morphism 
\mbox{$M = (k,\,id_Z,\, id,\,id,\, \mu,\, id): C \mr{} C'$} as follows:
}
	$$\vcenter{\xymatrix@C=4pc@R=1pc{
			X \ar@<1.5ex>[r]^{d _0}
			\ar@<.5ex>[r]_{d _1}
			\ar@<-0ex>@/_2ex/[ddr]^>>>>{d'_0}
			\ar@<-1ex>@/_2ex/[ddr]_>>>>{d'_1}
			\ar@<2ex>@/^3.5ex/[rr]^{x }
			\ar@<3.5ex>@{}[rr]|{\alpha_0  \Uparrow \, \alpha_1  \, \Uparrow}
			& W  \ar[dd]^{k} \ar[r]_\sim^{s} \ar@{}[ddr]|{\mu \, \Uparrow} & Z  \ar[dd]^{id_Z} \\
			\ar@{}[r]^{id \Uparrow \, id \, \Uparrow}
			& \\
			& W' \ar[r]_\sim^{s'} & Z}} 
	\quad = \quad 
	\vcenter{\xymatrix@C=4pc@R=1pc{
			X \ar@{}@<3ex>[rr]|{id \, \Uparrow}
			\ar@/^3ex/[rrdd]^{x'} 
			\ar@<-0ex>@/_2ex/[ddr]^{d'_0}
			\ar@<-1ex>@/_2ex/[ddr]_{d'_1} 
			\ar@<2ex>@/^3.5ex/[rr]^{x }
			& & Z  \ar[dd]^{id_Z} \\
			& \\
			& W' \ar[r]_\sim^{s'} & Z}} $$ 
{ Then, by item 1 in \mbox{Remark \ref{rem:alphasencorresp}} we complete $C'$ with 2-cells $\alpha'_0$, $\alpha'_1$.
}

Now we use axiom {\bf M1} in order to have 
	$\vcenter{\xymatrix@C=1.2pc@R=1.2pc{W \ar[rr]^{id_W} 
			\ar@{^{(}->}[dd]_k
			\ar@<.3ex>@{}[dd]^[@]{\sim}		
			&& W \ar@{->>}[dd] \\ 
			\ar@{}[ru]|{\eps \, \Downarrow \, \cong} & \\
			W' \ar@{.>}[rruu]_{m} \ar[rr] && 1}}$,
and set $h' = h * m$, \mbox{$\rho = h * \eps: h' * k \Xr{} h$.} 
{The proof finishes then by item 1 in \mbox{Remark \ref{rem:etaepsencorresp}.}}
\end{proof}

\begin{remark}
	The previous two lemmas admit a unified proof using {Remarks \ref{rem:alphasencorresp} and \ref{rem:etaepsencorresp}} only once. However we consider that the proof in two separate steps is easier to follow.
\end{remark}

\begin{lemma} \label{lem:3pasosbis}
	Let $\xymatrix{f \ar@2{~>}@<0.25ex>[r]^{H} & g}$ be a fibrant $w$-homotopy, then
	there is a fibrant \mbox{$q$-homotopy} $\xymatrix{f \ar@2{~>}@<0.25ex>[r]^{H'} & g}$ such that {$H' \germ > H$.}
\end{lemma}
\begin{proof}
	As for the previous two lemmas, for clarity it is convenient to proceed in two steps that have a similar strategy (we will now transfer $C$ and $H$ backwards instead of forward).
	In step 1, we will show that given $H$ as in the hypothesis, there exists $H'$ with {$H' \germ > H$,} 
which is still fibrant and such that $Z' = X$ and $x' = id_X$. In step 2, we will show that given $H$ satisfying these extra conditions, there exists a $q$-homotopy $H'$ (i.e. satisfying {also} the condition that $\binom{d_0}{d_1}$ is a cofibration) that is still fibrant and such that {$H' \germ > H$}.
	
	\medskip

\noindent 1. 
{
By Remark \ref{rem:soloimportaalphatilde} we can consider 
$C = (W,Z,d_0,d_1,s*d_1,s,{\alpha},s*d_1)$ to be the cylinder of 
$H$, $H = (C,h,\eta,\eps)$}. 
Let $W'$ in the diagram below be the  Pullback of $s$ and $s \ast d_1$ (recall axioms {\bf M3} and {\bf M4} so that $s'$ results a trivial fibration). 
	From the universal property of the coProduct (recall our notation from \ref{sin:limits}), we have a 2-cell 
	$ \binom{{\alpha}}{\gi}:$
	$s \ast \binom{d_0}{d_1} \Mr{} 
	s \ast d_1 \ast \nabla_X $, which induces 
	by the universal property of the Pullback the diagram below
	
	$$\vcenter{\xymatrix@C=4pc@R=2.5pc{
			X\amalg X \ar@<2ex>@/^3.5ex/[rr]^{\nabla_X}
			\ar@<3.5ex>@{}[rr]|{\alpha' \ \Uparrow \ \cong}
			\ar@/_3ex/@{{}{ }{}}[r]|>>>>>>>{\gamma \ \Uparrow \ \cong}
			\ar[r]^{\binom{d'_0}{d'_1}}
			\ar@/_2ex/[rd]_{\binom{d_0}{d_1}} 
			& W' \ar[d]^{k} \ar@{->>}[r]_\sim^{s'} \ar@{}[dr]|{\mu \, \Uparrow} & X \ar[d]^{s \ast d_1} \\
			& W \ar@{->>}[r]_\sim^{s} & Z}}
	\;\;=\;\; 
	\vcenter{\xymatrix@C=4pc@R=2.5pc{
			X\amalg X \ar@<2ex>@/^3.5ex/[rr]^{\nabla_X}
			\ar@/_2ex/[rd]_{\binom{d_0}{d_1}} 
			\ar@<-2ex>@{}[rr]|{\binom{{\alpha}}{\gi} \; \Uparrow}
			&& X \ar[d]^{s \ast d_1} \\
			& W \ar@{->>}[r]_\sim^{s} & Z}}
	$$
	
	The equality above is clearly equivalent to the following two equalities as in \eqref{eq:hiplemapara3pasos}:
	
	$$\vcenter{\xymatrix@C=4pc@R=1pc{
			X \ar@<1.5ex>[r]^{d'_0}
			\ar@<.5ex>[r]_{d'_1}
			\ar@<-0ex>@/_2ex/[ddr]^>>>{d_0}
			\ar@<-1ex>@/_2ex/[ddr]_>>>{d_1}
			\ar@<2ex>@/^3.5ex/[rr]^{id}
			\ar@<3.5ex>@{}[rr]|{\alpha'_0  \Uparrow \, \alpha'_1  \, \Uparrow}
			& W'  \ar[dd]^{k} \ar[r]_\sim^{s'} \ar@{}[ddr]|{\mu \, \Uparrow} & X  \ar[dd]^{s * d_1} \\
			\ar@{}[r]^{\gamma_0 \Uparrow \, \gamma_1 \, \Uparrow}
			& \\
			& W \ar[r]_\sim^{s} & Z}} 
	\quad = \quad 
	\vcenter{\xymatrix@C=4pc@R=1pc{
			X \ar@{}@<3ex>[rr]|  {\gi \, \Uparrow}
			\ar@{}[ddrr]|{{\alpha} \Uparrow \, \gi \Uparrow}
			\ar@/^3ex/[rrdd]^{s * d_1} 
			\ar@<-0ex>@/_2ex/[ddr]^{d_0}
			\ar@<-1ex>@/_2ex/[ddr]_{d_1} 
			\ar@<2ex>@/^3.5ex/[rr]^{{id_X}}
			& & X  \ar[dd]^{s * d_1} \\
			& \\
			& W \ar[r]_\sim^{s} & Z}} $$ 
	
{We have then a cylinder $C'$ and a morphism 
$M = (k,\,s*d_1,\,\gamma_0, \gamma_1, \mu, \gi):C' \mr{} C$,} which allows to {transfer $H$ backwards}: we set $h' = h * k$, 
$\rho = \gi$, and 
we use Remark \ref{rem:etaepsencorresp}, item 2 in order to have $H'$ such that {$H' \germ > H$.}
	\smallskip
	
	\noindent 2. We factorize 
	$\vcenter{\xymatrix{X \amalg X \ar[rr]^{\binom{d_0}{d_1}} \ar@<-1ex>@{^{(}->}[rd]_{\binom{d'_0}{d'_1}} && W \\
			& W' \ar@{->>}[ru]_k
			\ar@<-1ex>@{}[ru]^[@]\sim \ar@{}[u]|>>>>{\gamma \ \Uparrow \ \cong}}}$ using axiom {\bf M2}, and we set $s' = s*k$.

In this way we have a cylinder $C'$ except for the 2-cells 
$\alpha'_0$, $\alpha'_1$, and all the data of a morphism 
$M = (k,\,id_X,\,\gamma_0, \gamma_1, \gi, \gi):C' \mr{} C$ as follows:			 
 
$$\vcenter{\xymatrix@C=4pc@R=1pc{
			X 
			\ar@<1.5ex>[r]^{d'_0}
			\ar@<.5ex>[r]_{d'_1}
			\ar@<-0ex>@/_2ex/[ddr]^>>{d_0}
			\ar@<-1ex>@/_2ex/[ddr]_>>{d_1}
			\ar@<2ex>@/^3.5ex/[rr]^{id_X}
			& W'  \ar[dd]^{k} \ar[r]_\sim^{s'} \ar@{}[ddr]|{\gi \, \Uparrow} & X  \ar[dd]^{id_X} \\
			\ar@{}[r]^{\gamma_0 \Uparrow \, \gamma_1 \, \Uparrow}
			& \\
			& W \ar[r]_\sim^{s} & X}} 
	\quad \, \quad 
	\vcenter{\xymatrix@C=4pc@R=1pc{
			X 
			\ar@{}[rrdd]|{\alpha_0  \Uparrow \, \alpha_1  \, \Uparrow}
			\ar@{}@<3ex>[rr]|{\gi \, \Uparrow}
			\ar@/^3ex/[rrdd]^{id_X} 
			\ar@<-0ex>@/_2ex/[ddr]^{d_0}
			\ar@<-1ex>@/_2ex/[ddr]_{d_1} 
			\ar@<2ex>@/^3.5ex/[rr]^{id_X}
			& & X  \ar[dd]^{id_X} \\
			& \\
			& W \ar[r]_\sim^{s} & X}} $$ 
	
	\noindent 
{We can then {\em transfer backwards} $C$ and $H$ to get the desired $C'$ and $H'$:
first, by item 2 in \mbox{Remark \ref{rem:alphasencorresp}}, we complete $C'$ with 2-cells $\alpha'_0$, $\alpha'_1$.
Then, just like in step 1,  we set $h' = h * k$, 
$\rho = \gi$, and 
we use Remark \ref{rem:etaepsencorresp}, item 2 in order to have $H'$ such that {$H' \germ > H$.}}
\end{proof}
\begin{proposition} \label{2-cellsWofc}
	Let $\xymatrix{X \ar@<1ex>[r]^f \ar@<-1ex>[r]^g & Y}\in\C$, with $Y$ a fibrant object. 
	Let $\xymatrix{f \ar@2{~>}@<0.25ex>[r]^{H} & g}$ be a $w$-homotopy. Then 
	there is a fibrant $q$-homotopy $H'$ such that $[H'] = [H]$. 
\end{proposition}
\begin{proof}
		Using in turn Lemmas \ref{lema0}, \ref{lema:3pasos}, and \ref{lem:3pasosbis} we get a fibrant $q$-homotopy $H'$ such that $[H'] \germ [H]$, and the proof finishes by Lemma \ref{lema:para3pasos}.
\end{proof}

\subsection{Some further properties of the homotopy bicategory} \label{moreprop}

{In the remainder of this section we will prove two results that hold for $q$-homotopies, namely Lemmas \ref{lemma_A} and \ref{homotopia a izquierda da a derecha}, which are bicategorical versions of two basic results in \cite{Quillen}. These results will be used in the following sections, and both depend on the following {\em double lifting property}:}

\begin{proposition} \label{doubleLs_prop} ${}$

a) Let  $A \mr{i} X$ be a trivial cofibration, and $Y \mr{p_0} B$, $Y \mr{p_1} B$ be morphisms in $\cc{C}$, such that $(p_0,\, p_1): Y \to B \times B$ is a fibration. Then, 
for each pair of invertible $2$-cells $\gamma_0$,
$\gamma_1$
as in the two diagrams on the left, there exist a morphism $f$ and invertible $2$-cells $\lambda_0,\; \lambda_1,\; \rho_0,\; \rho_1$ satisfying the equations written below: 
\begin{equation} \label{doubleLs}
\xymatrix@R=7ex@C=9ex
      {
       A  \ar[r]^{a} \ar[d]^{i}  \ar@{}[rd]|{\gamma_0 \Downarrow}
      & Y  \ar[d]^{p_0}                  
      \\
          X  \ar[r]^{b_0} 
      & B
       }
\xymatrix@=3ex{\\ ,} \;\;
\xymatrix@R=7ex@C=9ex
       {
      A  \ar[r]^{a} \ar[d]^{i}  \ar@{}[rd]|{\gamma_1 \Uparrow}
    & Y  \ar[d]^{p_1}
    \\
      X  \ar[r]^{b_1} 
    & B
       } 
\hspace{2ex}\xymatrix@=3ex{\\ \leadsto} \hspace{2ex}
\xymatrix@R=7ex@C=9ex
      {
          A  \ar[r]^{a} \ar[d]^{i} 
             \ar@{}[rd]|(.2){\lambda_0 \Uparrow}
             \ar@{}[rd]|(.7){\rho_0 \; \Downarrow}  
        & Y  \ar[d]^{p_0}                 
        \\
          X  \ar[r]^{b_0} \ar[ru]^{f}
        & B
       } 
\xymatrix@=3ex{\\ ,} \;\;
\xymatrix@R=7ex@C=9ex
      {
          A  \ar[r]^{a} \ar[d]^{i} 
             \ar@{}[rd]|(.2){\lambda_1 \Downarrow} 
             \ar@{}[rd]|(.7){\rho_1 \; \Uparrow} 
        & Y  \ar[d]^{p_1}                  
        \\
          X  \ar[r]^{b_1} \ar[ru]^{f}
        & B
      }
\end{equation}

$
\gamma_0 \circ p_0 \ast \lambda_0 \;=\; \rho_0 \ast i \,,
\;\;\;\;\;\; 
p_1 \ast \lambda_1 \,\circ\, \gamma_1 \;=\;  \rho_1 \ast i \,, 
\;\;\;\;\;\;
\lambda_0 \circ \lambda_1 \,=\, id,
\;\;\; 
\lambda_1 \circ \lambda_0 \,=\, id.
$

b) Let $Y \mr{p} B$ be a trivial fibration and $A \mr{i_0} X$, 
$A \mr{i_1} X$ such that
$\binom{i_0}{i_1}: A \amalg A \to X $ is a cofibration. Then, 
for each pair of invertible $2$-cells $\gamma_0$,
$\gamma_1$ 
as in the two diagrams on the left, there exist a morphism $f$ and invertible $2$-cells 
$\lambda_0,\; \lambda_1,\; \rho_0,\; \rho_1$ satisfying the equations written below: 
\begin{equation} \label{doubleLsop}
\xymatrix@R=7ex@C=9ex
      {
       A  \ar[r]^{a_0} \ar[d]^{i_0} 
          \ar@{}[rd]|{\gamma_0 \Uparrow}
      & Y  \ar[d]^{p}                  
      \\
          X  \ar[r]^{b} 
      & B
       }
\xymatrix{\\ ,} \;\;
\xymatrix@R=7ex@C=9ex{
      A  \ar[r]^{a_1} \ar[d]^{i_1}  
         \ar@{}[rd]|{\gamma_1 \Downarrow}
    & Y  \ar[d]^{p}
    \\
      X  \ar[r]^{b} 
    & B
       }
\hspace{2ex}\xymatrix{ \\ \leadsto} \hspace{2ex} 
\xymatrix@R=7ex@C=9ex
      {
          A  \ar[r]^{a_0} \ar[d]^{i_0} 
             \ar@{}[rd]|(.2){\lambda_0 \Uparrow}
             \ar@{}[rd]|(.7){\rho_0 \; \Downarrow}  
        & Y  \ar[d]^{p}                 
        \\
          X  \ar[r]^{b} \ar[ru]^{f}
        & B
       } 
\xymatrix{\\ ,} \;\;
\xymatrix@R=7ex@C=9ex
      {
          A  \ar[r]^{a_1} \ar[d]^{i_1} 
             \ar@{}[rd]|(.2){\lambda_1 \Downarrow} 
             \ar@{}[rd]|(.7){\rho_1 \; \Uparrow} 
        & Y  \ar[d]^{p}                  
        \\
          X  \ar[r]^{b} \ar[ru]^{f}
        & B
      }
\end{equation} 
$
\gamma_0 \,\circ\, \rho_0 \ast i_0  \;=\; p \ast \lambda_0\,,
\;\;\;\;\;\; 
\rho_1 \ast i \,\circ\, \gamma_1 \;=\;  p \ast \lambda_1\,, 
\;\;\;\;\;\;
\rho_0 \circ \rho_1 \,=\, id,
\;\;\; 
\rho_1 \circ \rho_0 \,=\, id.
$
\end{proposition}
\begin{proof}
$
\xymatrix@R=2.5ex
 {\\ \text{a) Apply axiom {\bf M1} to the $2$-cell} \;\;\;\;}
\xymatrix@R=7ex@C=9ex
      {
       A  \ar[r]^{a} \ar[d]^{i}  
          \ar@{}[rd]|{(\gamma_0, \gamma_1^{-1}) \Downarrow}
      & Y  \ar[d]^{(p_0, p_1)}                  
      \\
          X  \ar[r]^{(b_0, b_1)} 
      & B \times B
       }
$

\hspace{4ex} b) is the dual statement.
\end{proof}

\begin{remark} \label{doublelifting_p}
We think of the diagrams in \eqref{doubleLs} (resp. \eqref{doubleLsop}), as 
{\em double lifting properties}. That is, we say that a triple 
$(i,\, p_0,\, p_1)$ (resp. $(i_0,\, i_1,\, p)$) satisfies {\bf LLs} (resp. {\bf LLs$^{\text{op}}$}) if any pair of squares as in the left admits fillers as in Lemma \ref{lem:lifting_prop} {\bf Ls}, for which we can take the same $f$, and the $2$-cells $\lambda_0,\; \lambda_1$ (resp. $\rho_0,\; \rho_1$) mutually inverse. 
\end{remark}

\begin{lemma} \label{invertiblep*full}
Let {$Z \mrr{f}{g} X \mr{p} Y$ be arrows in 
$\cc{C}$, with $\,p\,$ a trivial fibration,}
and $p*f \mrhpy{H} p*g$ be a $q$-homotopy with invertible cells
$H = (C,\,h,\,\eta,\, \varepsilon)$ (we refer to the notation in Definition \ref{lhpy}). Then there exists a  
$q$-homotopy $f \mrhpy{H'} g$ such that 
$[H] = [p\ast H'] = p \ast [H']$, which \mbox{furthermore has the same cylinder as $H$.} Note then that if $H$ is fibrant, so is $H'$.
Note also that by Proposition \ref{prop:siemprefaithfulenHoC} such an $[H']$ is unique. 
\end{lemma}
\begin{proof}
We apply proposition \ref{doubleLs_prop} b) as follows:
$$
\xymatrix@=7ex
      {
          Z  \ar[r]^{g} \ar[d]^{d_1}  \ar@{}[rd]|{\varepsilon \Uparrow}
        & X  \ar[d]^{p}
        \\
          W  \ar[r]^{h} 
        & Y
       } 
\;\; \xymatrix@=3ex{\\ and} \;\;
\xymatrix@=7ex
      {
          Z  \ar[r]^{f} \ar[d]^{d_0}  \ar@{}[rd]|{\eta \Downarrow}
        & X  \ar[d]^{p}                  
        \\
          W  \ar[r]^{h} 
        & Y
       }
\;\;\;\; \xymatrix@=3ex{\\ \leadsto} \;\;\;\;
\xymatrix@=7ex
      {
          Z  \ar[r]^{g} \ar[d]^{d_1} 
             \ar@{}[rd]|(.3){\varepsilon' \Uparrow} 
             \ar@{}[rd]|(.7){\rho_0 \; \Downarrow} 
        & X  \ar[d]^{p}                  
        \\
          W  \ar[r]^{h} \ar[ru]^{h'\!\!\!\!}
        & Y
      }
\;\; \xymatrix@=3ex{\\ and} \;\;
\xymatrix@=7ex
      {
          Z  \ar[r]^{f} \ar[d]^{d_0} 
             \ar@{}[rd]|(.3){\eta' \Downarrow}
             \ar@{}[rd]|(.7){\rho_1 \; \Uparrow}  
        & X  \ar[d]^{p}                 
        \\
          W  \ar[r]^{h} \ar[ru]^{h'\!\!\!\!}
        & Y
       } 
$$
This yields a fibrant $q$-homotopy $f \mrhpy{H'} g$,  
$H' = (C,\,h',\,\eta',\, \varepsilon')$ such that 
$$
p\ast \varepsilon' \,=\, \varepsilon \circ \rho_0 \ast d_1\,, 
\;\;\;  
p \ast \eta' \,=\, \rho_1 \ast d_0 \circ \eta\,,
\;\;\;  
\rho_0 \circ \rho_1 = id\,.
$$
It follows 
$$  
p\ast \varepsilon'\circ \rho_1 \ast d_1 \,=\, \varepsilon\,,
\;\;\;\;
\rho_0 \ast d_0 \circ p \ast \eta' \,=\, \eta\,.
$$ 
We show that $[p * H'] = [H]$ using Corollary \ref{coro:lemapara3pasosconC1igualC2} with $H^1 = H$ and 
$H^2 = p*H'$. It remains to check the hypotheses of the Corollary, which in this case are exactly the two preceding equations.
\end{proof}

Our next task is to prove a statement analogous to Lemma \ref{invertiblep*full}, but for the \mbox{$q$-homotopies} $I^\mu$ given by the $2$-cells $\mu$ of $\cc{C}$, not necessarily invertible.
It is here that we will need the full strength of the lifting property {\bf L}, and not just {\bf Ls}.

The following proposition expresses the content of axiom {\bf M1} when the object $A$ is the biInitial object of the bicategory.  

\begin{proposition} \label{p_* ess surjective-full}
Given a cofibrant object $Z$ and a trivial fibration $X \mr{p} Y$, the functor $\cc{C}(Z,\,X) \mr{p_\ast} \cc{C}(Z,\,Y)$ is essentially surjective and essentially full.
\end{proposition}
\begin{proof}
Consider the lifting property {\bf L} for the arrows $0 \mr{} Z$ and \mbox{$X \mr{p} Y$.} The categories $\cc{C}(0,\,X),\; \cc{C}(0,\,Y)$ are equivalent to the singleton category $\nn{1}$, thus since equivalences are stable by biPullingback, the functor 
\mbox{$\nn{P} \mr{} \cc{C}(Z,\,Y)$} in the diagram below is an equivalence of categories:
$$
\vcenter{
\xymatrix@=7ex
    {
     \cc{C}(Z,\,X) \ar@/^5ex/[rr]^{}
                   \ar[r]^(.65){h} 
                   \ar@/_4ex/[rd]^{p_\ast}  
   & \nn{P} \ar[d]_{\simeq} 
            \ar[r]^(.4){}
            \ar@{}[rd]|{\cong}
   & \nn{1} \ar[d]^{\simeq} 
   \\
   {} 
   & \cc{C}(Z,\,Y) \ar[r]^{} 
   & \nn{1}
    }}
$$
Since $Z$ is cofibrant, by axiom {\bf M1} the functor $h$ is essentially surjective and essentially full, thus so is $p_\ast$.
\end{proof}
Recalling  Remark \ref{MM1} it follows: 
\begin{corollary} \label{factorizationofmu} 
Let $Z \mrr{f}{g} X \mr{p} Y$, with $p$ a trivial fibration and $Z$ cofibrant, and let ${p * f} \Mr{\mu} {p * g}$ be a 2-cell. Then:

There exist morphisms $Z \mrr{f'}{g'} X$, invertible $2$-cells  
$p*f \Mr{\rho_1} p*f'$, $p*g \Mr{\rho_2} p*g'$, and  a 2-cell 
$f' \Mr{\delta} g'$ such that 
$p*\delta \circ \rho_1 = \rho_2 \circ \mu$. \cqd  
\end{corollary}

\begin{lemma} \label{for2cellp*full}
Let $Z \mrr{f}{g} X \mr{p} Y$ be arrows in $\cc{C}$, with $\,p\,$ a trivial fibration, and $Z$ cofibrant. 
Given any $2$-cell $p*f \Mr{\mu} p*g$ in $\cc{C}$, there exists a 2-cell $f \Mr{\alpha} g$ in $\cc{H}o(\C,\cc{W})$ such that \mbox{$[I^\mu] = p \ast \alpha$,}
 Note that by 
Proposition \ref{prop:siemprefaithfulenHoC} such an $\alpha$ is unique. 

Furthermore, $\alpha$ can be taken to be the class of a sequence 
{$[H_2,\, H_3,\, H_1]$} as in \ref{sin:varios2}, with these three homotopies fibrant.
\end{lemma}
\begin{proof}
By the corollary above we have a factorization of $\mu$,
$\mu = \rho_2^{-1} \circ p\ast \delta \circ \rho_1$. 
Since $[I^{\rho_1}]$ and $[I^{\rho_2^{-1}}]$ can be taken fibrant and with invertible cells, we can apply Proposition \ref{invertiblep*full} to get $H_1$, $H_2$ such that $p \ast [H_1] = [I^{\rho_1}]$, $p \ast [H_2] = [I^{\rho_2^{-1}}]$. We define $\alpha$ as the class of the sequence {$[H_2,\, I^\delta,\, H_1]$}, then we have
$$
p \ast \alpha = 
p \ast [H_2] \circ p \ast [I^\delta] \circ p \ast [H_1] =
[I^{\rho_2^{-1}}] \circ p \ast [I^\delta] 
\circ [I^{\rho_1}] =
[I^\mu]. 
$$
\end{proof}

By combining Lemmas \ref{invertiblep*full} and \ref{for2cellp*full}, we will prove:

\begin{lemma} \label{p*full}
Let $Z \mrr{f}{g} X \mr{p} Y$ be arrows in $\cc{C}$, with $\,p\,$ a trivial fibration, and $Z$ cofibrant. 
Given any $2$-cell $p*f \Mr{\alpha} p*g$ in $\cc{H}o^f(\C,\cc{W})(Z,Y)$, there exists a 2-cell $f \Mr{\beta} g$ in $\cc{H}o^f(\C,\cc{W})(Z,X)$ such that \mbox{$\alpha = p \ast \beta$} (where the superscript ``$f$" indicates the sub-bicategory {defined by the} fibrant homotopies (indistinctly $w$- or $q$-), see Notation \ref{not:HomlHomr} below).
\end{lemma}

\begin{proof}
Clearly it suffices to show this for the case in which $\alpha = [H]$, the class of a fibrant q-homotopy. 
We use proposition \ref{prop:Hcompde3} and write 
$
[H] = [\varepsilon \circ (h * H^C) \circ \eta]
$, note that this is the composition  
$p*f \Mr{\eta} h * d_0 \Xr{h * H^C} h * d_1 \Mr{\varepsilon} p*g$ so we cannot apply Lemmas \ref{invertiblep*full} and \ref{for2cellp*full} directly to these 2-cells. Since $Z$ is cofibrant, we can however consider the lifting property ${\bf Ls.}$, as in \eqref{L1}, separately for $h * d_0$ and $h * d_1$

\begin{equation*}
\vcenter{\xymatrix@R=9ex@C=10ex
   {
    0 \ar@{}[dr] | {\cong \, \Downarrow \;\; }
      \ar[r] \ar[d] & X \ar[d]^p
    \\
    Z \ar[r]^{h * d_0}
    \ar@<-1ex>[r]_{h * d_1}
    & Y
   }}
\qquad
\vcenter{\xymatrix{\leadsto}}
\hspace{5ex}
\vcenter{\xymatrix@R=9ex@C=11ex
      {
          0  \ar[r]^{} \ar[d]^{} 
             \ar@{}[rd]|(.22){\; \cong \, \Uparrow \, }
             \ar@{}[rd]|(.7){\cong \, \Downarrow \, \rho_0,\rho_1}  
        & X  \ar[d]^{p}                 
        \\
          Z  \ar@{}@<-1.5ex>[ru]^[@]{s_0,s_1}
          \ar[ru]
         \ar[r]^{h * d_0}
    \ar@<-1ex>[r]_{h * d_1}
        & Y
       }}
\hspace{1ex}
\end{equation*}

We can then define:

\begin{enumerate}
    \item $\alpha_1$ as the composition of the 2-cells $p*f \Mr{\eta} h * d_0 \Xr{\rho_0^{-1}} p*s_0$.
    
    \item $\alpha_2$ as the composition 
    $p*s_0 \Mr{\rho_0}  h * d_0 \Xr{h * H^C} h * d_1 \Xr{\rho_1^{-1}} p * s_1$, as defined in Remark \ref{2compoH}, which is a single fibrant homotopy with invertible 2-cells,
    
    \item $\alpha_3$ as the composition of the 2-cells $p*s_1 \Mr{\rho_1} h * d_1 \Mr{\varepsilon} p*g$,
\end{enumerate}

and note that $[H] = [I^{\alpha_3}] \circ [\alpha_2] \circ [I^{\alpha_1}]$. Then, if we apply Lemma \ref{invertiblep*full} to $\alpha_2$ and Lemma \ref{for2cellp*full} to $\alpha_1$ and $\alpha_3$, we get three 2-cells $\beta_i$ in $\cc{H}o^f(\C,\cc{W})(Z,X)$ such that 
$p * \beta_1 = [I^{\alpha_1}]$,
$p * \beta_2 = [\alpha_2]$, 
$p * \beta_3 = [I^{\alpha_3}]$.
Since then 
$$
\alpha = [H] =
 [I^{\alpha_3}] \circ [\alpha_2] \circ [I^{\alpha_1}]
= p\ast \beta_3 \circ p\ast \beta_2 \circ p\ast \beta_1 = p\ast (\beta_3 \circ \beta_2 \circ \beta_1),
$$
we can take $\beta_3 \circ \beta_2 \circ \beta_1$ for the desired $\beta$.
\end{proof}

The following Lemma corresponds to Quillen's \cite[\S 1, Lemma 7]{Quillen}. 

\begin{lemma} \label{lemma_A}
Given a cofibrant object $Z$ and a trivial fibration 
$X \mr{p} Y$, the functor $\cc{H}o^f(\C,\cc{W})(Z,X) \mr{p_*} \cc{H}o^f(\C,\cc{W}) (Z,Y)$ of postcomposition with $p$ (see \ref{sin:horiz_comp}) is an equivalence of categories.
\end{lemma}
\begin{proof}
By Proposition  \ref{prop:siemprefaithfulenHoC} we know that $p_*$ is faithful, and by Proposition \ref{p_* ess surjective-full} it follows that it is essentially surjective. The fact that it is full is proved in Lemma \ref{p*full}. 
\end{proof}

\subsection{Switching between right and left homotopies} \label{sub:interplay}

The following lemma is a bicategorical version of  
\cite[I, \S1, Lemma 5 (i)]{Quillen}. It allows to {\em switch} from left to right homotopies {(recall Definitions \ref{lhpy} and \ref{rhpy})} without changing the equivalence class. 
In addition, the right homotopy can be taken with respect to an arbitrary fixed $q$-path-object, which by the dual of Remark
\ref{rem:existencilindros} can be assumed to be cofibrant.

\begin{lemma}\label{homotopia a izquierda da a derecha}
	Let 
	$\xymatrix{X \ar@<1ex>[r]^f \ar@<-1ex>[r]^g &
		Y} \in\C$, with $X$ a cofibrant object, $\homotopy{f}{H}{g}$ a left $q$-homotopy, and $P$ a $q$-path-object for $Y$. Then there exists a right $q$-homotopy $\homotopy{f}{K}{g}$ with path object $P$ such that $[K]=[H]$. 
\end{lemma}
\begin{proof}
Let $H = (C,h,\eta,\eps)$, $C = (W,d_0,d_1,s,\alpha_0,\alpha_1)$ and $P = (V,c_0,c_1,t,\beta_0,\beta_1)$ (see Definitions \ref{lhpy} and \ref{rhpy}).
We define 2-cells $\gamma_0$ and $\gamma_1$ as the composites:
$$
\gamma_0 \;\meq{(1)} \; 
c_1*t*f \Xr{\beta_1 * f} f \Xr{\eta} h*d_0, 
\hspace{4ex} 
\gamma_1 \;\meq{(2)}\; 
f*s*d_0 \Xr{f*\alpha_0} f \Xr{\beta_0^{-1}*f } c_0*t*f
$$

We assume first that $H$ has invertible cells, thus $\eta$ is invertible, and so is $\gamma_0$. Since $\gamma_1$ was already invertible, we apply Proposition \ref{doubleLs_prop} a) and we have (note that $d_0$ is a trivial cofibration by Lemma \ref{A cofibrant da d y c cofibraciones triviales}, {it is here where we need $X$ to be cofibrant}):
$$
\xymatrix@R=7ex@C=9ex
      {
          X  \ar[r]^{t\ast f} \ar[d]^{d_0}  \ar@{}[rd]|{\gamma_0 \Downarrow}
        & V  \ar[d]^{c_1}                  
        \\
          W  \ar[r]^{h} 
        & Y
       }
 \xymatrix@=3ex{\\ ,} \;
\xymatrix@R=7ex@C=9ex
      {
          X  \ar[r]^{t\ast f} \ar[d]^{d_0}  
             \ar@{}[rd]|{\gamma_1 \Uparrow}
        & V  \ar[d]^{c_0}
        \\
          W  \ar[r]^{f\ast s} 
        & Y
       } 
\;\;\;\; \xymatrix@=3ex{\\ \leadsto} \;\;\;\;
\xymatrix@R=7ex@C=9ex
      {
          X  \ar[r]^{t\ast f} \ar[d]^{d_0} 
             \ar@{}[rd]|(.2){\lambda_0 \Uparrow}
             \ar@{}[rd]|(.7){\rho_0 \; \Downarrow}  
        & V  \ar[d]^{c_1}                 
        \\
          W  \ar[r]^{h} \ar[ru]^{\ell}
        & Y
       } 
 \xymatrix@=3ex{\\ ,} \;
\xymatrix@R=7ex@C=9ex
      {
          X  \ar[r]^{t\ast f} \ar[d]^{d_0} 
             \ar@{}[rd]|(.2){\lambda_1 \Downarrow} 
             \ar@{}[rd]|(.7){\rho_1 \; \Uparrow} 
        & V  \ar[d]^{c_0}                  
        \\
          W  \ar[r]^{f\ast s} \ar[ru]^{\ell}
        & Y
      }
$$
$$
\rho_0 \ast d_0 \;\meq{(3)}\; \gamma_0 \,\circ \, c_1 \ast \lambda_0\,,
\;\;\; 
\rho_1 \ast d_0  \;\meq{(4)}\; c_0 \ast \lambda_1 \,\circ\, \gamma_1 \,,
\;\;
\lambda_0 \circ \lambda_1 \,\meq{(5)}\, id,
\;\;
\lambda_1 \circ \lambda_0 \,=\, id.
$$

\noindent We define $k = \ell \ast d_1$ $: X \mr{d_1} W \mr{\ell} V$, and
the $2$-cells $\delta$ and $\epsilon$ as the composites:
$$
\delta \;\meq{(a)}\; f \Xr{f \ast \alpha_1^{-1}} f \ast s \ast d_1 
\Xr{\rho_1 \ast d_1} c_0 \ast \ell \ast d_1,
\hspace{4ex}
\epsilon \;\meq{(6)}\; c_1 \ast \ell \ast d_1 \Xr{\rho_0 \ast d_1} 
h \ast d_1 \Xr{\varepsilon} g
$$
We define the right homotopy $K$ as $K = (P,k,\delta,\epsilon)$.

\vspace{1ex} 

We prove now that $[K] = [H]$, that is,  
$\widehat{FK} = \widehat{FH}$, see Definition \ref{def:notacionFyrelationHo}. 

\vspace{1ex}

{\em
Since $F$ is a $2$-functor, to simplify the notation we can omit the letter $F$, and we do so in what follows.}

\vspace{1ex}

\noindent Recall the $2$-cells from \ref{sin:varios1}
$$
d_0 \Xr{\widehat{C}} d_1 \;\;\text{and}\;\; c_0 \Xr{\widehat{P}} c_1 
\;, \;\;\; s \ast \widehat{C} \;\meq{(b)}\; \alpha_1^{-1} \circ \alpha_0 
\;\;\text{and}\;\; 
\widehat{P} \ast t \;\meq{(7)}\; \beta_1^{-1} \circ \beta_0
$$
from (a) and (b) it follows that 
$$
\delta \;\meq{(8)}\; f \Xr{f \ast \alpha_0^{-1}} f \ast s \ast d_0 
\Xr{f \ast s \ast \widehat{C}} f \ast s \ast d_1 
\Xr{\rho_1 \ast d_1} c_0 \ast \ell \ast d_1  
$$
Recall now that $\widehat{H}$ and $\widehat{K}$ are defined as the composites
$$
\widehat{H} \;\meq{(9)}\; f \Xr{\eta} h \ast d_0 
\Xr{h \ast \widehat{C}} 
h \ast d_1 \Xr{\varepsilon} g 
\;\; \text{and}\;\; 
\widehat{K} \;\meq{(10)}\; f \Xr{\delta} c_0 \ast t 
\Xr{\widehat{P} \ast t} 
c_1 \ast t \Xr{\epsilon} g 
$$

\smallskip
\noindent
We prove now that $\widehat{K} = \widehat{H}$, using the elevator calculus described in \S \ref{sub:prelimenappendix}, in Appendix \ref{sec:transfer}:
$$
\vc{\xymatrix@C=0ex
    {
     f \clcputo{\widehat{K}}
   \\
     g
    }}
\;\; \meq{(10)} \;\;
\vc{\xymatrix@C=-0.8ex@R=3.5ex 
    {
     {} & f \cla{\delta} & {} 
   \\
    c_0 \clcputo{\widehat{P}} & {} & t \eqq      
   \\
    c_1 & {} \clcc{\epsilon} & t       
   \\
     {} & g & {}                              
   }}
\;\; \meq{(6)(8)} \;\;
\vc{\xymatrix@C=1ex@R=3.5ex
    {
     f \eqq & {} \dr \dcr{\alpha_0^{-1}} & {} \dl  
   \\
     f \eqq & s \eqq & d_0 \clcputo{\widehat{C}}
   \\
     f \dr \dcr{\rho_1} & \dl s & d_1 \eqq 
   \\
     c_0 \clcputo{\widehat{P}} & \ell \eqq & d_1  \eqq
   \\
     c_1  \dl \dcr{\rho_0} & \ell \dr & d_1  \hs{-2} \eqq
   \\
    {} & \hs{-35} h \hs{-20} \dl & \hs{-4} d_1 \hs{-10} \dr   
                             \dcl{\varepsilon}
   \\
     {} & {} & \hs{-20}{g}
    }}
\;\; \meq{(el)} \;\;
\vc{\xymatrix@C=1ex@R=3.5ex
    {
     f \eqq & {} \dr \dcr{\alpha_0^{-1}} & {} \dl  
   \\
     f \dl & s \dcl{\rho_1} \dr & d_0 \eqq
   \\
     c_0  \clcputo{\widehat{P}} & \ell \eqq & d_0 \eqq 
   \\
     c_1 \dl \dcr{\rho_0}   & \ell \dr & d_0 \eqq    
   \\
     {} &  \hs{-16} h \ar@<-1.6ex>@{=}[d] & d_0   
                                           \clcputo{\widehat{C}}
   \\
     {} & \hs{-35} h \hs{-20} \dl & \hs{-4} d_1 \hs{-10} \dr   
                             \dcl{\varepsilon}
   \\
     {} & {} & \hs{-20}{g}  
    }}
\;\; \meq{(3)(4)} \;\;
\vc{\xymatrix@C=1ex@R=3.5ex
    {
     f \eqq & {} \dr \dcr{\alpha_0^{-1}} & {} \dl  
   \\
     f \dl & s \dc{\gamma_1} & d_0 \dr
   \\
     c_0 \eqq & t \dl \dcr{\lambda_1} & f \dr
   \\
     c_0 \clcputo{\widehat{P}} & \ell \eqq & d_0 \eqq
   \\
     c_1 \eqq & \ell \dl \dcr{\lambda_0} & d_0 \dr
   \\
     c_1 \dla & t \dc{\gamma_0} & f \dr
   \\ 
     {} &  \hs{-16} h \ar@<-1.6ex>@{=}[d] & d_0   
                                           \clcputo{\widehat{C}}
   \\
     {} & \hs{-35} h \hs{-20} \dl & \hs{-4} d_1 \hs{-10} \dr   
                             \dcl{\varepsilon}
   \\
     {} & {} & \hs{-20}{g}  
  }}
\;\; \meq{(el)(5)} \;\;
$$ 
$$
\;\; \meq{(el)(5)} \;\;
\vc{\xymatrix@C=1ex@R=3.5ex
    {
     f \eqq & {} \dr \dcr{\alpha_0^{-1}} & {} \dl  
   \\
     f \dl & s \dc{\gamma_1} & d_0 \dr
   \\
     c_0 \clcputo{\widehat{P}} & t \eqq & f \eqq
   \\
     c_1 \dla & t \dc{\gamma_0} & f \dr
   \\ 
     {} &  \hs{-16} h \ar@<-1.6ex>@{=}[d] & d_0   
                                           \clcputo{\widehat{C}}
   \\
     {} & \hs{-35} h \hs{-20} \dl & \hs{-4} d_1 \hs{-10} \dr   
                             \dcl{\varepsilon}
   \\
     {} & {} & \hs{-20}{g}  
    }}
\;\; \meq{(7)} \;\;
\vc{\xymatrix@C=1ex@R=3.5ex
    {
     f \eqq & {} \dr \dcr{\alpha_0^{-1}} & {} \dl  
   \\
     f \dl & s \dc{\gamma_1} & d_0 \dr
   \\
     c_0  \dl \dcr{\beta_0} & t \dr & f \eqq     
   \\
     {} \dr \dcr{\beta_1^{-1}} & {} \dl & f \eqq
   \\
     c_1 \dla & t \dc{\gamma_0} & f \dr         
   \\ 
     {} &  \hs{-16} h \ar@<-1.6ex>@{=}[d] & d_0   
                                           \clcputo{\widehat{C}}
   \\
     {} & \hs{-35} h \hs{-20} \dl & \hs{-4} d_1 \hs{-10} \dr   
                             \dcl{\varepsilon}
   \\
     {} & {} & \hs{-20}{g}  
    }}
\;\; \meq{(1)(2)} \;\;
\vc{\xymatrix@C=-0.8ex@R=3.5ex 
    {
     {} & f \cla{\eta} & {} 
   \\
    h \eqq & {} & d_0 \clcputo{\widehat{C}}      
   \\
    h & {} \clcc{\varepsilon} & d_1       
   \\
     {} & g & {}                              
   }}
\;\; \meq{(9)} \;\;
\vc{\xymatrix@C=0ex
    {
     f \clcputo{\widehat{H}}
   \\
     g
    }}
$$
 
{
This finishes the proof for $H$ with invertible cells. Let 
now $[H]$ be with arbitrary cells. By Proposition \ref{prop:Hcompde3} we have 
$$
[H] = [\eps \circ (h * H^C) \circ \eta] = 
[I^\eps] \circ [h * H^C] \circ [I^\eta]
$$

Recalling that $h\ast H^C$ has invertible cells, we can apply the above and get $K'$ such that $[K'] = [h\ast H^C]$. We can take the desired $q$-homotopy $K$ to be $\eps \circ K' \circ \eta$, {defined by the dual of Remark \ref{2compoH}}, which still has $P$ as path object 
and satisfies

$$[K] = [\eps \circ K' \circ \eta] = [I^\eps] \circ [K'] \circ [I^\eta] = [I^\eps] \circ [h * H^C] \circ [I^\eta] = [H]
$$
(recall that $I^\eps$ as well as $I^\eta$ can be considered to be right or left homotopies, see Remark \ref{Imuleft=Ymurigth}) 
}
\end{proof}

\section{Replacement for model bicategories} \label{sec:replacement-model}

{ 
In this section we will define a fibrant-cofibrant replacement for model bicategories.
Using just the axioms of a model bicategory, and mimicking the 1-dimensional construction in \cite[I, proof of Th. 1]{Quillen}, we can define an assignation on objects and arrows
$\xymatrix@C=8ex{\cc{C} \ar@{~>}[r]^{RQ} & 
\cc{C}}$, 
\mbox{$X\mr{f} Y \,\leadsto \,RQX \mr{RQf} RQY$,}
such that $RQX$ is a fibrant-cofibrant object for all $X \in \cc{C}$. We call such an assignation a  {\em preassignment} of bicategories (see Definition \ref{def:assignment}). As we mentioned in the introduction of this paper, we are considering Quillen's axioms in \cite{Quillen} as a base for our work.
Since the appearance of \cite{Quillen}, in the vast literature on model categories, these axioms have been strengthened and for example in \cite{Hirsch}, \cite{Hovey},  there is the assumption of the existence of functorial factorizations which allow to define, for a model category $\bf C$, a fibrant-cofibrant replacement functor $\bf C \mr{} C$.
However, we are not assuming a (pseudo)functorial factorization in the model bicategory axioms, and this is the reason why we do not have a fibrant-cofibrant replacement pseudofunctor $\cc{C} \xr{RQ} \cc{C}$.

Our strategy for having nevertheless a fibrant-cofibrant replacement pseudofunctor, which can be used to show a localization theorem as in \cite[I, proof of Th. 1]{Quillen} is the following.
 First we note that by definition $\cc{C}$ and $\cc{H}o^f(\C,\cc{W})$ have the same objects and arrows, so that $RQ$ above is equivalently a preassignment $\xymatrix@C=8ex{\cc{H}o^f(\C,\cc{W}) \ar@{~>}[r]^{RQ} & 
\cc{H}o^f(\C,\cc{W})}$ (see Definition \ref{def:assignment}).
Then, by  results in Appendix \ref{sec:transfer} on \emph{transfer of structures}, we will extend this preassignment to a  pseudofunctor $R_\ell Q$, {which is the composition of two pseudofunctors $R_\ell$ and $Q$, satisfying that
$R_\ell Q = RQ$ on objects and arrows.} $R_\ell Q$ is a fibrant-cofibrant replacement \mbox{$\cc{H}o^f(\C,\cc{W}) \xr{R_\ell Q} \cc{H}o^f(\C,\cc{W})$} defined on the fibrant homotopies.
In particular, we can consider the pseudofunctor $q$ defined as the composition $$q : \cc{C} \mr{i} \cc{H}o^f(\C,\cc{W}) \xr{R_\ell Q} \cc{H}o^f(\C,\cc{W}).$$
Note that, even though there may not be structural 2-cells of $\cc{C}$ making $\xymatrix{\cc{C} \ar@{~>}[r]^{RQ} & \cc{C}}$ above into a pseudofunctor, the pseudofunctor structure of $q$ is showing that there are structural homotopies serving that purpose, just like in the original case \cite{Quillen}. 
In Section \ref{sec:locthm}, we will show that $q$, or 
more precisely its co-restriction to $\cc{H}o^f_{fc}(\C,\cc{W})$,
 is the pseudofunctor giving the localization of $\cc{C}$ at the weak equivalences.
}

\begin{notation} \label{not:HomlHomr}
{Until now, we have usually omitted to specify that we were working with left homotopies. Since in this section we will work with both left and right homotopies, for clarity we will indicate with a superscript ``$\ell$" or ``$r$" which are the \mbox{$2$-cells} of the homotopy bicategory being considered. }

The fibrant left homotopies (and cofibrant right homotopies) will play a major role in the remainder of the paper. We denote by $\cc{H}o^f(\C,\cc{W})^\ell$ the sub-bicategory of 
$\cc{H}o(\C,\cc{W})^\ell$ {with the same objects and arrows, whose 2-cells are}
given by the fibrant homotopies (which can be either $w$- or $q$-homotopies, by Lemma 
\ref{lem:3pasosbis}). 
{ By this we mean precisely that the 2-cells of $\cc{H}o^f(\C,\cc{W})^\ell$ are the classes of finite sequences 
{$[H^n,\,\ldots \,,H^2,\, H^1]$} of homotopies, as in \ref{sin:varios2}, which are all fibrant. 
{We define $\cc{H}o^f_{inv}(\C,\cc{W})^\ell \subseteq \cc{H}o^f(\C,\cc{W})^\ell$ as the sub-bicategory for which, in addition, we ask all these homotopies to have invertible cells.

The reader will note that it could be possible to vertically compose the homotopies $H^1,H^2...,H^n$ above as in Lemma \ref{lemacompo}, so that $[H^n,\,\ldots \,,H^2,\, H^1]$ is also the class of a single fibrant homotopy, as is the case when $X$ is cofibrant and $Y$ is fibrant by Lemma \ref{2-cellsWofc}. However, this fact is not relevant for the definitions of $\cc{H}o^f_{inv}(\C,\cc{W})^\ell$ and $\cc{H}o^f(\C,\cc{W})^\ell$, which also include the case of arbitrary objects $X$ and $Y$.}
}

Note that $\cc{H}o^f(\C,\cc{W})^\ell$ {and $\cc{H}o^f_{inv}(\C,\cc{W})^\ell$} are indeed sub-bicategories of 
$\cc{H}o(\C,\cc{W})^\ell$, this follows because the composition $r*H$ is defined to have the same cylinder as $H$, and $H*\ell$ with a cylinder that has the same ``$s$".   
This is not the case if we consider arbitrary $q$-homotopies instead of the fibrant ones: arbitrary $q$-homotopies are not closed under horizontal composition and do not form a bicategory, see \ref{sin:horiz_comp}.

We will also work with the dual cofibrant right homotopies, that we denote with a $c$ instead of an $f$, $\cc{H}o^c(\C,\cc{W})^r$.

{\em As a mnemonic, remember that superscripts refer to  homotopies, and subscripts refer to objects.}
\end{notation}

{
\begin{proposition}[{\bf Cofibrant replacement}] \label{cofibrantprop}
There exist a pseudofunctor 
\mbox{$\cc{H}o^f(\C,\cc{W})^\ell \mr{Q} \cc{H}o^f(\C,\cc{W})^\ell$}
and a pseudonatural transformation $Q \Mr{\rho} id$ such that $\rho_X$ is a trivial fibration for each $X$. 
All the $QX$ are cofibrant objects and, when restricted to cofibrant objects, $Q$ and $\rho$ coincide with the identity.
\end{proposition}
}


\begin{proof}
{\bf Cofibrant replacement on objects and arrows.} 

\vspace{1ex}

For each object $X \in \cc{C}$, we choose a cofibrant object $QX$ and a trivial fibration \mbox{$QX \mr{\rho_X} X$.} If $X$ is already cofibrant, we require $QX=X$ and $\rho_X=id_X$. Note that this can be obtained by factorizing  $0 \mr{} X$ using axiom {\bf M2}.

	     For each arrow $X \mr{f} Y$, we choose an arrow 
$QX \mr{Qf} QY$ and an invertible 2-cell
$$
f \ast \rho_X \Mr{\rho_f} \rho_Y \ast Qf \,, \hspace{5ex}
\vcenter{\xymatrix{ QX \ar@{}[dr]|{\Downarrow \rho_f} \ar[d]_{Qf} \ar[r]^{\rho_X} & X \ar[d]^f \\
 QY \ar[r]^{\rho_Y} & Y }} 
$$.
  
If $X$ and $Y$ are cofibrant, we require $Qf=f$ and $\rho_f=id_f$. Note that this can be obtained using axiom {\bf M1s}:
$$
 \vcenter{
\xymatrix@C=1.5pc@R=3pc{0 \ar[rr]^{} \ar[d]_{} &\ar@{}[d]|{\cong}& QY \ar[d]^{\rho_Y} \\
			QX \ar[r]_{\rho_X} & X \ar[r]_f & Y}}
\quad \leadsto \quad
\vcenter{\xymatrix@C=1.5pc@R=1.2pc{0 \ar@{}[dr]|{\cong}\ar[rr]^{} \ar[dd]_{} && QY \ar[dd]^{\rho_Y} \\
	& \ar@{}[dr]|{ \cong \; \Uparrow \; \rho_f \;\;} \\
			QX \ar@{-->}[rruu]|{\comw{MM^M} Qf \comw{MM^M} } \ar[r]_{\rho_X} & X \ar[r]_f& Y}}
$$	
{(we take $\rho_f$ as the inverse of the 2-cell given by the axiom)} 			

\begin{remark} \label{rem:Qfwesifloes}
If $X$ is fibrant, then so is $QX$ by axiom {\bf M3}. 
\end{remark}
\begin{remark} \label{Qwe=we}
If $f$ is a weak equivalence, then so is $Qf$ by axiom {\bf M5}.
\end{remark}

\smallskip

Recall from Remark \ref{choose_I^mu} that, for each
2-cell $\rho_f$,
$I^{\rho_f}$ can be taken to be a fibrant $w$-homotopy {with invertible cells}. 
{By Proposition \ref{invertibleH} the class of such a 
\mbox{$w$-homotopy} is an invertible 2-cell of $\cc{H}o^f(\C,\cc{W})$
that we will {abuse and denote also} by $\rho_f$.} Recall also that, since $\rho_X$ is a trivial fibration and $QX$ is cofibrant, from Lemma \ref{lemma_A} we know that 
$\cc{H}o^f(\cc{C},\cc{W})^\ell(W, QX)
 \xr{(\rho_{X})_\ast} \cc{H}o^f(\cc{C},\cc{W})^\ell(W, X)$ is full and faithful for $W$ in the image of $Q$.
     
\vspace{1ex}
  
We refer now to Definitions \ref{def:assignment}, 
\ref{psntd} and \ref{transferdatum} in Appendix \ref{sec:transfer}, with 
\mbox{$\cc{B} = \cc{D} = \cc{H}o^{f}(\cc{C},\cc{W})^\ell$}. The considerations we just made show that we have a transfer datum from the preassignment determined by $Q$ to the identity considered as a {pseudofunctor. In this way, by Theorem \ref{thetheorem}, $Q$ adquires a pseudofunctor structure in such a way that $\rho$ becomes a pseudonatural transformation}
\end{proof}

\smallskip
A fibrant replacement is a cofibrant replacement in the dual bicategory. To set up the notation we write the dual of Proposition \ref{cofibrantprop}, whose
 proof uses the {\em forward transfer} Theorem  \ref{theoptheorem} instead of Theorem \ref{thetheorem}.

\begin{proposition}[{\bf Fibrant replacement}] \label{fibrantprop}
There exist a pseudofunctor 
\mbox{$\cc{H}o^c(\C,\cc{W})^r \mr{R} \cc{H}o^c(\C,\cc{W})^r$}  and a pseudonatural transformation $id \Mr{\lambda} R$ such that 
$\lambda_X$ is a trivial cofibration for each $X$. All the $RX$ are fibrant objects and, when restricted to fibrant objects, 
$R$ and $\lambda$ coincide with the identity.
\cqd
\end{proposition}

\vspace{1ex}

\begin{sinnadastandard} \label{composition}
\noindent
{\bf Composing a fibrant and a cofibrant replacement.}
 
It is clear how to compose a fibrant and a cofibrant replacement on objects and arrows, but not so clear how to compose the corresponding pseudofunctors. As expected, Lemma \ref{homotopia a izquierda da a derecha} and its dual provide the key. They allow to define the ``switch" 2-functors 
$s$ and $s^{op}$, $s\,s^{op} = id$, $s^{op}\,s  = id$, which are the identity on objects and arrows, and in the $2$-cells switch the homotopies from left to right and viceversa, but the $2$-cell remains the same. 
 Note that Lemma \ref{homotopia a izquierda da a derecha} holds when all the involved objects are cofibrant, thus its dual will hold when they are fibrant.
Note also that, by the dual of Remark \ref{rem:Qfwesifloes}, $R$ and $\lambda$ can be restricted to 
$\cc{H}o^{c}_c(\cc{C},\cc{W})^r
\mr{R} 
\cc{H}o^{c}_c(\cc{C},\cc{W})^r
$, $id \Mr{\lambda} R$.

{
We abuse and denote with the same letters $\cc{H}o^f(\C,\cc{W})^\ell \mr{Q} \cc{H}o^f_c(\C,\cc{W})^\ell$,
\mbox{$\cc{H}o^c_c(\C,\cc{W})^r \mr{R} \cc{H}o^c_{fc}(\C,\cc{W})^r$,} and
$\cc{H}o^{c}_{fc}(\cc{C},\cc{W})^r \mr{s^{op}} 
 \cc{H}o^{f}_c(\cc{C},\cc{W})^\ell$ the pseudofunctors which are either co-restrictions to their images or post-compositions with inclusions, of the respective 
 endo-pseudofunctors.
}
This allows to define a \emph{fibrant replacement on left homotopies}, that is the arrow $R_l$ defined as the composition $R_\ell = s^{op}\,R\,s$, which can be composed with $Q$ as follows:
$$
\xymatrix@C=5ex
   {
    \cc{H}o^{f}(\cc{C},\cc{W})^\ell       
                 \ar[r]^Q
  & \cc{H}o^{f}_c(\cc{C},\cc{W})^\ell     
                 \ar[r]^s
                 \ar@{->}@<-4pt>`u[rrr]`[rrr]^{R_\ell}[rrr]
  & \cc{H}o^{c}_c(\cc{C},\cc{W})^r        
                 \ar[r]^R
  & \cc{H}o^{c}_{fc}(\cc{C},\cc{W})^r     
                 \ar[r]^{s^{op}} 
  & \cc{H}o^{f}_{fc}(\cc{C},\cc{W})^\ell  
   }
$$
Note that $R_\ell X = RX$, $R_\ell f = Rf$. There is a pseudonatural transformation 
$id \Mr{\lambda_\ell} R_\ell$ defined as $\lambda_\ell = s^{op}\,\lambda\,s$.
Note that the \mbox{$X$-components} of $\lambda_\ell$ are the same that those of $\lambda$, $(\lambda_\ell)_X = \lambda_X$, we also have $(\lambda_\ell)_f = \lambda_f$. We define a fibrant-cofibrant replacement pseudofunctor as the composite \mbox{$R_\ell\,Q$.}
\end{sinnadastandard}

From Propositions \ref{cofibrantprop} and \ref {fibrantprop} we have:

\begin{proposition} [{\bf Fibrant-cofibrant replacement}]
\label{fib-cofib}
There exist a pseudofunctor 
\mbox{$\cc{H}o^f(\C,\cc{W})^\ell \xr{R_\ell Q} 
                                     \cc{H}o^f(\C,\cc{W})^\ell$}
and pseudonatural transformations 
\mbox{
      $
      id \Ml{\rho} Q
\xymatrix@C=5ex
    {
     {} \ar@2{->}[r]^{\lambda_\ell\, Q}
   & {}
    }
      R_\ell\,Q
     $ 
   } 
such that $\rho_X$ is a trivial fibration and  
$(\lambda\,Q)_X$ is a trivial cofibration  for each $X$.
All the $(R_\ell Q) X$ are fibrant-cofibrant objects and, when restricted to fibrant-cofibrant objects, $R_\ell Q$, $\rho$ and 
$\lambda_\ell\, Q$ coincide with the identity.
\end{proposition}
From Remark \ref{Qwe=we} and its dual it follows
\begin{remark} \label{RQwe=we}
Note that if $f$ is a weak equivalence, so is $(R_\ell Q) f$. \cqd
\end{remark}

\section{The localization theorem.}  \label{sec:locthm}

Much as in the 1-dimensional case, Lemma \ref{homotopia a izquierda da a derecha} can be used together with its dual to give \emph{niceness} {properties of homotopies between arrows in $\cc{C}(W,\,Z)$ when $W$ is cofibrant and $Z$ is fibrant.}  Note that, using also Proposition \ref{2-cellsWofc}, it follows that in this case all notions of homotopy {($w$ or $q$, right or left, right cofibrant or left fibrant)} coincide {in the strong sense that the equality (given by the switch functors) is an isomorphism of categories}
$\cc{X}(W,Z) =  \cc{Y}(W,Z)$, where $\cc{X}$, $\cc{Y}$ stand for the homotopy bicategories corresponding to any two choices of these concepts. We have
$$
\HoC = \cc{H}o_{fc}(\C,\cc{W}) 
 = \cc{H}o^c_{fc}(\C,\cc{W})^r = \cc{H}o^f_{fc}(\C,\cc{W})^\ell
$$

As a consequence, although for the fibrant-cofibrant replacement both left and right homotopies are necessary, the latter ones become \emph{implicit}. Thus in {this section} we will drop the superscript ``$\,\ell\,$" in the 
{ $\cc{H}o$ bicategories}, which \emph{will always have left homotopies as $2$-cells.} 

\vspace{1ex}

\begin{definition} \label{sin:defdeHoCcomoHofc}
We {set} $\HoC =\Wofc$, that is the full sub-bicategory of $\Wo$ given by the fibrant-cofibrant objects.
\end{definition}
From the Lemma \ref{lemacompo} we have
\begin{remark}
The 2-cells of $\HoC$ are classes of a single $q$-homotopy. For any pair of $q$-homotopies $[H_1,\,H_2]$ in $\HoC$,
there is a single $q$-homotopy $[H]$ in $\HoC$ such that 
$[H] =  [H_1,\,H_2]$.
\end{remark}

We consider the $2$-functor 
$\cc{C} \mr{i} \cc{H}o^f(\C,\cc{W})$ in \ref{sin:varios2}, Remark \ref{choose_I^mu}, and the fibrant-cofibrant replacement pseudofunctor 
$R_\ell Q$ in Proposition \ref{fib-cofib}. 
Note that $R_\ell Q$ takes its values in $\HoC$, {here it is convenient to denote its co-restriction as}
$\cc{H}o^f(\C,\cc{W}) \mr{r} \HoC$.
We have thus the following commutative diagram {, where $q$ is defined as the composite $q = r\,i$ and $j$ denotes the inclusion}: 
\begin{equation} \label{diagram_q}
\xymatrix@C=8ex
    {
     \cc{C} \ar[d]_i  \ar[r]^{q} 
   & \HoC   \ar[d]^j  
   \\
     \cc{H}o^f(\C,\cc{W})  \ar[r]^{R_\ell Q} \ar[ur]^{r}
   & \cc{H}o^f(\C,\cc{W})   
    }
\end{equation}
We will show that $q$ is the localization of $\C$ with respect to $\cc{W}$, that is, we will prove the following theorem:

\vspace{1ex}

\begin{theorem} \label{teo:localizationformodel}
The pseudofunctor $q = r \,i$ is the localization of $\C$ with respect to $\cc{W}$, in the sense that it maps the arrows of $\cc{W}$ to equivalences, and furthermore for each bicategory $\cc{D}$, precomposition with $q$,
\mbox{$Hom(\HoC,\cc{D}) \mr{q^*} Hom_{\cc{W},\Theta}(\C,\cc{D})$} is a biequivalence of bicategories, where $Hom_{\cc{W},\Theta}(\C,\cc{D})$ stands for the full sub-bicategory of $Hom(\C,\cc{D})$ given by those pseudofunctors that send weak equivalences into equivalences.   
\end{theorem}

The proof of this theorem depends on the following three results.

\begin{proposition} \label{qok}
The pseudofunctor $q$ maps the arrows of $\cc{W}$ to equivalences.
\end{proposition}
\begin{proof}
If $f$ is a weak equivalence, by Remark \ref{RQwe=we} $q(f)$ is also a weak equivalence, but now between fibrant-cofibrant objects. Thus by Proposition \ref{prop:compderetractos} it is a composite of split weak equivalences (see Definition \ref{def:retrsect}), which by  
\cite[Prop. 3.45]{DDS.loc_via_homot} is an equivalence. 
\end{proof}

Let $i\cc{W}$ be the class of arrows in 
$\cc{H}o^{f}(\C,\cc{W})$ of the form $i\,f$ with $f \in \cc{W}$.  
\begin{proposition} \label{medioteorema}
{Consider the diagram \eqref{diagram_q}.}
For each bicategory $\cc{D}$, precomposing with $r$,
$Hom(\HoC,\cc{D}) \mr{r^*} Hom(\cc{H}o^{f}(\C,\cc{W}),\cc{D})
$
factors through the full sub-bicategory 
$Hom_{i\cc{W},\Theta}(\cc{H}o^{f}(\C,\cc{W}),\cc{D})$ and together with $j^*$ establishes a biequivalence of bicategories:
$$
Hom(\HoC,\cc{D}) \mrl{r^*}{j^*}  
Hom_{i\cc{W},\Theta}(\cc{H}o^{f}(\C,\cc{W}),\cc{D})
$$
\end{proposition}
\begin{proof}
Let $F \in Hom(\HoC,\cc{D})$ and 
$f \in \cc{W}$, then $(r^* F)(if) = F r\,if = Fqf$ which is an equivalence by Proposition \ref{qok}. This shows the factorization of $r^*$. 

We already have $j^*\, r^* = (r\,j)^* = id^* = id$, so it only remains to establish an equivalence 
$id  \simeq r^*\,j^* {= (jr)^* = (R_\ell Q)^*}$. 
Let $id \Ml{\rho^*} Q^*
\xymatrix@C=6ex{{} \ar@2{->}[r]^{{(\lambda_\ell Q)}^*}&{}}  {(R_\ell Q)}^*$ be the pseudonatural transformations induced by those in Proposition \ref{fib-cofib}, to conclude the proof we will show that they are equivalences.
By item 3. in \ref{equivalencias}, it suffices to check that $((\rho^*)_F)_X = F(\rho_X)$ and $(({(\lambda_\ell Q)}^*)_F)_X = F(\lambda_{QX})$ are equivalences, for each 
$\cc{H}o^f(\C,\cc{W}) \mr{F} \cc{D}$ that maps weak equivalences to equivalences, and for each $X$.
{Since $\rho_X$ and $\lambda_{QX}$ are indeed weak equivalences (see Propositions \ref{cofibrantprop} and \ref{fibrantprop}), this is precisely the case.}
\end{proof}

\vspace{1ex}

We have also the following result proven in
\cite[Theorem 3.42, Remark 3.47]{DDS.loc_via_homot}:
\begin{theorem} \label{teo:mediapuposta}
The $2$-functor $\C \mr{i} \cc{H}o^{f}(\C,\cc{W})$ is such that precomposing with $i$ establishes a biequivalence  of bicategories, which in fact is an isomorphism:
$$
Hom_{i\cc{W},\Theta}(\cc{H}o^{f}(\C,\cc{W}),\cc{D}) \mr{i^*} Hom_{\cc{W}, \Theta}(\C,\cc{D})
$$

\vspace{-4.5ex}

\cqd 
\end{theorem}

\vspace{1ex}

Finally, since $q^* = (r\,i)^* = i^*\,r^*$, putting together Proposition \ref{medioteorema} and Theorem \ref{teo:mediapuposta} finishes the proof of the localization theorem, Theorem \ref{teo:localizationformodel}.  
\qed

\vspace{2ex}

We finish the paper by showing that the set-theoretical difficulties in the construction of localizations are resolved for the homotopy bicategory. From Proposition \ref{homotopia a izquierda da a derecha} (and its dual version) it follows immediately:

\begin{proposition}[Single cylinder] \label{singlecylinder} 
Let $X \mrr{f}{g} Y$ be any pair of arrows in 
\mbox{$\HoC = \cc{H}o_{fc}(\C,\cc{W})\,$,} and fix a cylinder $C$. Then all the $q$-homotopies $\homotopy{f}{}{g}$ can be \mbox{chosen} with $C$ as its cylinder. \qed
\end{proposition}
 
Recall that a bicategory is said to be locally small if, for each pair of objects $X,Y$, the category $\C(X,Y)$ is small, that is, has a set of objects and a set of arrows. From Proposition \ref{singlecylinder} it follows:
\begin{corollary} \label{coro:size}
If $\C_{fc}$ is locally small (in particular when $\C$ is so), then so is $\HoC$. \qed
\end{corollary}

\begin{appendix}

\section{{Transfer of structure for bicategories}}
\label{sec:transfer}

\subsection{Assignments and preAssignments.} \label{sub:prelimenappendix}

We {introduce} the notions of assignment and preassignment between bicategories, as well as lax transformations between them, and show how they are as lax functors and lax natural transformations but with some of the structure omitted.
{ The computations in this Appendix are written in the following {\em elevator notation}, so in order to make them easier to follow we write these definitions using that notation too.}

\vspace{2ex}
\noindent
{\bf Elevator calculus.\footnote{Developed in 1969 by the second author for draft use.}} 
A $2$-cell \; $f \Mr{\alpha} g$ \;is written \; 
$\vcenter{\xymatrix@C=0ex{f \dcell{\alpha} \\ g }} \;$. 
Each elevator diagram represents a composition of $2$-cells in a fixed bicategory. Objects are omitted, arrows are composed from right to left, and 2-cells from top to bottom. Using the basic move ``$(el)$", {which corresponds to the middle-four interchange as shown below,} we form configurations of cells that fit valid equations in order to prove new equations. {This is justified by the Coherence result in \ref{basic notation}.}
$$ \label{ascensor}
\!\!\!\!\!\!\!\!\!
\begin{tabular}{c}
\mbox{middle-four interchange:} \\
     \xymatrix{
A \ar@<2ex>[r]^{f_1} \ar@<-2ex>[r]_{f_2} \ar@{}[r]|{\Downarrow  \ \alpha} & B \ar@<2ex>[r]^{g_1} \ar@<-2ex>[r]_{g_2} \ar@{}[r]|{\Downarrow  \ \beta} & C
}  \\
$(g_2 * \alpha) \circ (\beta * f_1) = 
(\beta * f_2) \circ (g_1 * \alpha) = 
\beta * \alpha$
\end{tabular}
 \leadsto  \qquad
\vcenter{\xymatrix@C=0ex
         {
             g_1 \dcell{\beta} & f_1 \did
          \\
             g_2 \did & f_1 \dcell{\alpha}
          \\
             g_2  &  f_2 
         }}
{\;\;\;\meq{(el)}\;\;\;}
\vcenter{\xymatrix@C=0ex
         {
             g_1 \did & f_1 \dcell{\alpha}
          \\
             g_1 \dcell{\beta} & f_2 \did
          \\
             g_2 & f_2 
         }}
{\;\;\;\meq{(el)}\;\;\;}
\vcenter{\xymatrix@C=0ex
         {g_1 \dcell{\beta} 
             & f_1 \dcell{\alpha} \\
             g_2 & f_2 
         }}
$$
{We emphasize that the elevator calculus is nothing more and nothing else than a different notation for 2-cells and their compositions when written as word sequences (not pasting diagrams), which renders the middle-four interchange transparent and allows a visual guide to find and prove equations.}
\begin{definition} \label{def:assignment} \label{PF} ${}$

1) A \emph{preassignment} between bicategories $\cc{B} \Mrl{F} \cc{D}$ consists of the following data: 

For each object $X$ and for each arrow 
$X \mr{f} Y$ of $\cc{B}$, an object $FX$ and an arrow 
\mbox{$FX \mr{Ff} FY$ of $\cc{D}$.}

\vspace{1ex}
 
2) An \emph{assignment} between bicategories is a preassignment, plus:

For each $2$-cell $f \Mrl{\alpha} g$ in $\cc{B}$, a $2$-cell 
$Ff \Mrl{F\alpha} Fg$ in $\cc{D}$ such that for all $X,Y$, 
\mbox{$\cc{B}(X,Y) \mr{F} \cc{D}(FX,FY)$} satisfies the functor axioms.

\vspace{1ex}

3) A \emph{lax functor}  between bicategories is an assignment, plus:

For each object $X$ of $\cc{B}$, a 2-cell
$id_{FX} \Xr{\xi_X} F(id_X)$, referred to as the \emph{unit} of the lax functor,
and for each pair of arrows \mbox{$X \mr{f} Y \mr{g}  Z$} of 
$\cc{B}$, 
a 2-cell $Fg \ast Ff \Xr{\phi_{f,g}} F(g \ast f)$, referred to as the \emph{multiplication} of the lax functor, such that
(we will omit the subindices of $\xi$ and $\phi$):

i) $\xi$ and $\phi$ satisfy the unit and associativity monoidal axioms:

For each $X \mr{f} Y\in \cc{B}$, \hspace{.2cm}
{\bf LF1.} 
$\vcenter{\xymatrix@C=-2ex{Ff \did && \dcell{\xi} \\
Ff && Fid_X  \\
& Ff \clold{\phi} }}
\vcenter{\xymatrix@C=0ex{ \; = \; }}
\vcenter{\xymatrix@C=0ex{Ff \did \\ Ff }} \quad$
\hspace{.2cm}{\bf LF2. } 
$\vcenter{\xymatrix@C=-2ex{\dcell{\xi} && Ff \did \\
Fid_X  && Ff \\
& Ff \clold{\phi} }}
\vcenter{\xymatrix@C=0ex{ \; = \; }}
\vcenter{\xymatrix@C=0ex{Ff \did \\ Ff }}$

\vspace{1ex}
 
For each $X \mr{f} Y \mr{g} Z \mr{h} W\in \cc{B}$, 
\hspace{.2cm}{\bf LF3.} 
$\vcenter{\xymatrix@C=-2.5ex{Fh && Fg {\color{white}XX}  & Ff \did \\
& F(hg) \ar@{}[u]|{\phi} 
\ar@{-}[lu]
\ar@<.5ex>@{-}[ru]
&& Ff \\
&& F(hgf) \clold{\phi}   }}
\vcenter{\xymatrix@C=0ex{ \; = \; }}
\vcenter{\xymatrix@C=-2.5ex{Fh \did & {\color{white}XX} Fg && Ff \\
Fh && F(gf) \ar@{}[u]|{\phi} 
\ar@<-.5ex>@{-}[lu]
\ar@{-}[ru] \\
& F(hgf) \clold{\phi} }}$

\vspace{1ex}

ii)  \mbox{$\ast \circ (F\times F) \Mr{\phi} F\circ \ast:\cc{B}(X,Y)\times \cc{B}(Y,Z) \mr{} \cc{B}(X,Z)$} satisfies the axioms of a natural transformation with components $\phi_{f,g}$. We will often use this naturality, thus we make it explicit: 

For each $\xymatrix{X \ar@<1.6ex>[r]^{f} 
             \ar@{}@<-1.3ex>[r]^{\!\! {\alpha} \, \!\Downarrow}
             \ar@<-1.1ex>[r]_{g} & Y      \ar@<1.6ex>[r]^{s} 
             \ar@{}@<-1.3ex>[r]^{\!\! {\beta} \, \!\Downarrow}
             \ar@<-1.1ex>[r]_{t} & Z}\in \cc{B}$, \hspace{2ex}              
      {\bf N$\phi$.} 
$\vcenter{\xymatrix@C=-2ex{Fs && Ff \\
& F(s f) \clold{\phi} \dcellbb{F(\beta\alpha)} \\
& F(t g) }}
\vcenter{\xymatrix@C=0ex{ \; = \; }}
\vcenter{\xymatrix@C=-2ex{Fs \dcellb{F\beta} && Ff \dcellb{F\alpha} \\ 
Ft && Fs \\
& F(t g) \clold{\phi} }}$

\vspace{1ex}

A \emph{pseudofunctor} is a lax functor such that all the 2-cells $\phi$ and $\xi$ are invertible.
\end{definition}

\begin{definition} ${}$ \label{psntd} \label{PN}
Let $\cc{B} \Mrl{G} \cc{D}$ and $\cc{B} \Mrl{F} \cc{D}$ be  preassignments of bicategories:

\vspace{1ex}

1) 
A \emph{lax transformation} $\theta:{F}\,\hpy \,{G}$ between preassignments consists of a family of arrows 
$FX \xr{\theta_{X}} GX$, one for each $X\in \cc{B}$, 
and a family of 2-cells 
$$
\,\theta_f\!:\, Gf \ast \theta_X \Mr{} \theta_Y \ast Ff, \hspace{5ex}
\vcenter{\xymatrix{ FX \ar@{}[dr]|{\Downarrow \theta_f} \ar[d]_{Ff} \ar[r]^{\theta_X} & GX \ar[d]^{Gf} \\
 FY \ar[r]^{\theta_Y} & GY }} 
$$

2) If $F$ and $G$ are assignments, a \emph{lax transformation} is required to satisfy the condition

\noindent {\bf LN0.} For each ${X}\cellrd{{f}}{\alpha}{{g}}{Y}\in \cc{B}$, $\quad\vcenter{\xymatrix@C=-0pc{
                      {G}{f} \dcellb{{G}\alpha}  
		      &&
		      \theta_{X}  \did 
		      \\
		       {G}{g} \dl
		       &&
		      \theta_{X} \dr \ar@{}[dll]|{\theta_{g}}
		      \\
		       \theta_{{Y}} 
		       && 
		      {F}{g} 
		      }}
     \vcenter{\xymatrix@C=-.4pc{\quad = \quad }}
      \vcenter{\xymatrix@C=-0pc{
		      {G}{f} \dl
		      &&
		      \theta_{X} \dr \ar@{}[dll]|{\theta_{f}} 
		      \\
		      \theta_{{Y}} \did
		      &&
		      {F}{f} \dcellb{{F}\alpha} 
		      \\ 
		      \theta_{Y}
		      &&
		      {F}{g} 
		      }}$

\vspace{1ex}

3) If $F$ and $G$ are lax functors, a \emph{lax transformation} is required, in addition, to satisfy:

\noindent {\bf LN1.} For each ${X}\in \cc{B}$, 
$\quad \vcenter{\xymatrix@C=0pc
	   {
		   \theta_{{X}} \did 
		& \dcell{\xi}  
		\\
		\theta_{{X}} 
	        &
	        {F}id_{X}
		}}
\vcenter{\xymatrix@C=0pc{\; = \; }}
\vcenter{\xymatrix@C=0pc
       {
         \dcell{\xi} 
	& 
	\theta_{X} \did
	\\
	{G}id_{X} \dl
	& \theta_{X} \ar@{}[dl]|{\theta_{id_{X}}}  \dr
	\\
        \theta_{X}
	&
	{F}id_{X}
	}}$   
	
\vspace{1.5ex}

\noindent {\bf LN2.} For each ${X}\stackrel{{f}}\rightarrow {Y}\stackrel{{g}}\rightarrow {Z}\in \cc{B}$, 
$\quad \vcenter{\xymatrix@C=-0pc{
 	   {G}{g} \did	 
	   & & 
 	   {G}{f} \dl 
	   &&
 	   \theta_{X} \ar@{}[dll]|{\theta_{f}} \dr
 	   \\
 	   {G}{g} \dl
 	   & & 
 	   \theta_{Y} \ar@{}[dll]|{\theta_{g}} \dr
 	   &&
 	    {F}{f} \did
 	   \\ 
 	   \theta_{{Z}} \did
            & &  {F}{g} &&
            {F}{f} \\
            \theta_{{Z}} 
            &&&
            \!\!\!\!\! {F}({g}{f}) \!\!\!\!\! \clold{\phi} &
            }}
 \vcenter{\xymatrix@C=-.4pc{\quad = \quad }}
\vcenter
   {
   \xymatrix@C=-0.5pc
        {{G}{g} 
 		 &&
 		 {G}{f} 
 		 & \quad & \theta_{X} \did 
 	   \\
 	   &
 		 {G}({g}{f}) \clold{\phi}   \dl
 		&& &
 		\theta_{X} \ar@{}[dll]_{\theta_{{g}{f}}}  \dr
		\\
 		&
 		 \theta_{Z} 
 		& &&
 		 {F}( {g}{f}) 
 		}
   }
 $

\smallskip

\vspace{1ex}

Lax transformations between lax functors as defined above are no other than the {\em usual lax natural transformations} (as in for example \cite{Lackcompanion}), and we refer to them like that as well.

An op-lax transformation (in all the three cases above) is as a lax transformation but with the 2-cells $\theta_f$ reversed, satisfying, as the case may be, dual axioms.

A \emph{pseudo transformation} is a lax transformation with the 2-cells $\theta_f$ invertible.
 Note than in this case the lax and op-lax axioms are equivalent. The collection $\theta_f^{-1}$ determines an op-pseudo transformation.

\emph{Note that we can consider pseudonatural transformations between lax functors.}
\end{definition}

\begin{definition}
A \mbox{\emph{modification}} $\rho: \theta \rightarrow \eta$ between lax transformations (in the three cases in Definition \ref{psntd}) is a family of 2-cells $\theta_{X}\Xr{\rho_{X}}\eta_{X}$ of $\cc{D}$, one for each $X\in \cc{B}$, such that:

\smallskip

\noindent {\bf M.} For each $ X\mr{f} Y\in \cc{B},
\; $\hspace{.2cm}
$\vcenter{\xymatrix@C=-0pc{
		      Gf \dl 
		      && 
		      \theta_{X} \dr \ar@{}[dll]|{\theta_f} 
		      \\
		      \theta_Y \dcellb{\rho_Y}
		      && 
		      Ff \did  
		      \\
		      \eta_Y
		      &&
		      Ff
		      }}
     \vcenter{\xymatrix@C=-.4pc{\quad = \quad }}
      \vcenter{\xymatrix@C=-0pc{
		      Gf \did 
		      && 
		      \theta_X \dcellb{\rho_X} 
		      \\
		      Gf \dl 
		      && 
		      \eta_X \dr \ar@{}[dll]|{\eta_f} 
		      \\
		      \eta_Y
		      && 
		      Ff
		      }}$
\end{definition}

\vspace{1ex}

\begin{notation} \label{notation} 
For each pair of bicategories $\cc{B},\cc{D}$, the definitions above determine the 2-categories in the following diagram, 
with objects as indicated in the notation, arrows the pseudo transformations, and 2-cells the modifications.
We have forgetful 2-functors $|-|$, faithful at the level of arrows and fully-faithful at the level of 2-cells:
\begin{equation} \label{eq:42-cat}
\xymatrix@R=0pt
    {
     lax\cc{F}unc(\cc{B}, \cc{D}) \ar[rd]^{|-|} & {} & {}
  \\ {} & \cc{A}ssg(\cc{B}, \cc{D}) \ar[r]^{|-|} 
        & pre\cc{A}ssg(\cc{B}, \cc{D})
  \\ psd\cc{F}unc(\cc{B}, \cc{D}) \ar[ru]^{|-|} & {} & {}
    }
\end{equation}

We write $\cc{X}(\cc{B}, \cc{D})$ to indicate that we refer to one of these four 2-categories.
We omit to describe the compositions of arrows and 2-cells in these 2-categories, as they are ubiquitous in the literature.  
We note that we have other {\em larger} 2-categories 
$\cc{X}_\ell(\cc{B}, \cc{D})$,
whose arrows are lax transformations, 
connected by similar forgetful 2-functors,
which will not be relevant for our work here.
\end{notation} 

\subsection{Quasiequivalences}
Recall from \ref{sin:varios1} that an arrow $X \mr{f} X'$ in a bicategory $\cc{D}$ is a 
{\em quasiequivalence} if for every object $Y$ in $\cc{D}$ the functors $\cc{D}(Y,X) \mr{f_*} \cc{D}(Y,X')$ and 
$\cc{D}(X',Y) \mr{f^*} \cc{D}(X,Y)$ are full and faithful.

\begin{proposition} [{\bf pointwise quasiequivalences}]\label{quasipointwise}
{If an arrow in any of the four 2-categories 
$\cc{X}(\cc{B}, \cc{D})$
in \eqref{eq:42-cat}, that is a 
pseudo transformation \mbox{$F \Mr{\theta} G$},} is such that for every $X$ in 
$\cc{B}$, $FX \Mr{\theta_X} GX$ is a \mbox{quasiequivalence} in $\cc{D}$, then it is a quasiequivalence {in 
$\cc{X}(\cc{B}, \cc{D})$}. In this case we say that $\theta$ is a \emph{pointwise quasiequivalence}.  
\end{proposition}

\begin{proof}
We will show the case where $\theta_*$ is full and faithful, the case 
$\theta^*$ is dual.  
Let $\cc{B} \mr{H} \cc{D}$ be any lax or pseudo functor, and 
$H \Mr{\alpha} F$, $H \Mr{\beta} F$ be pseudonatural transformations. Consider for each $X \in \cc{B}$ the following diagram:
$$
\xymatrix@C= 10ex
   {
    \hspace{5ex} 
    \cc{X}(\cc{B},\, \cc{D})(H, F)(\alpha, \beta) 
    \ar[r]^{\theta_\ast}
    \ar[d]
  & \cc{X}(\cc{B},\, \cc{D})(H, G)
            (\theta \!\circ \!\alpha, \theta \!\circ\! \beta)
    \ar[d]
 \\   
    \hspace{5ex} 
    \cc{D}(HX, FX)(\alpha_X, \beta_X) 
    \ar[r]^{{(\theta_X)}_\ast}
  & \cc{D}(HX, GX)
     (\theta \!\circ \!\alpha)_X, (\theta \!\circ\! \beta)_X)
   }
$$
where the arrow in the bottom is a bijection by hypothesis. We will show that then the upper horizontal arrow is also a bijection.

Let $\theta \!\circ \!\alpha \mr{\mu} \theta \!\circ\! \beta$ be a modification, so for each $X$ we have a $2$-cell 
\mbox{$(\theta \!\circ \!\alpha)_X \Mr{\mu_X} (\theta \!\circ\! \beta)_X$, }
thus there is a unique $2$-cell $\alpha_X \Mr{\rho_X} \beta_X$ such that $\theta_X \ast \rho_X = \mu_X$. It remains to prove the modification axiom {\bf M.} for $\rho$.
By the modification axiom for $\mu$ we have
$$
\vc{\xymatrix@C=-4pt               
    { 
     Gf \sd{-2}{4} \na{20}{(\theta \circ \alpha)_f}
   & (\theta \!\circ\!\alpha)_X \sd{2}{-4} 
  \\  
     (\theta \!\circ\!\alpha)_Y \cl{12}{\mu_Y}
   & Hf \eq{0}
  \\  
     (\theta \!\circ\!\beta)_Y
   & Hf
   }} 
\hs{10} \meq{M} \hs{10}
\vc{\xymatrix@C=-4pt               
    { 
     Gf \eq{0} 
   & (\theta \!\circ\!\alpha)_X \cl{12}{\mu_X}
  \\  
    Gf \sd{0}{4} \na{22}{(\theta \circ \beta)_f}
  & (\theta \!\circ \!\beta)_X \sd{2}{-4}
  \\  
    (\theta \!\circ\!\beta)_Y
  & Hf
   }} 
$$
Recalling the definition of the vertical composition of pseudonatural transformations and that 
$\theta_X \ast \rho_X = \mu_X$, the previous equation becomes equation (1) below
$$
\vc{\xymatrix@C=0pt
    { 
     Gf \sd{-2}{4} \na{12}{\theta_f}
   & \theta_X \sd{2}{-4} 
   & \alpha_X \eq{0}
  \\  
     \theta_Y \eq{0}
   & Ff \sd{-2}{4} \na{12}{\alpha_f}
   & \alpha_X \sd{2}{-4}
  \\  
     \theta_Y \eq{0} 
   & \alpha_Y \cl{10}{\rho_Y}
   & Hf \eq{0}
 \\  
     \theta_Y
   & \beta_Y
   & Hf 
   }}
\hs{10} \meq{(1)} \hs{10}
\vc{\xymatrix@C=0pt
    { 
     Gf \eq{0}
   & \theta_X \eq{0}
   & \alpha_X \cl{10}{\rho_X}
  \\  
     Gf \sd{-2}{4} \na{12}{\theta_f}
   & \theta_X \sd{2}{-4} 
   & \beta_X \eq{0}
  \\  
     \theta_Y \eq{0}
   & Ff \sd{-2}{4} \na{12}{\beta_f}
   & \beta_X \sd{2}{-4}
 \\  
     \theta_Y
   & \beta_Y
   & Hf 
   }}
\hs{10} \meq{(el)} \hs{10}
\vc{\xymatrix@C=0pt
    { 
     Gf \sd{-2}{4} \na{12}{\theta_f}
   & \theta_X \sd{2}{-4} 
   & \alpha_X \eq{0}
  \\  
     \theta_Y \eq{0}
   & Ff \eq{0}
   & \alpha_X \cl{10}{\rho_X}
  \\  
     \theta_Y \eq{0}
   & Ff \sd{-2}{4} \na{12}{\beta_f}
   & \beta_X \sd{2}{-4}
 \\  
     \theta_Y
   & \beta_Y
   & Hf 
   }}
$$
Since $\theta_f \ast \alpha_X$ is invertible and $(\theta_Y)_\ast$ is faithful, it follows that equation {\bf M.} holds for $\rho$.
\end{proof}

\vspace{1ex}

As far as we know, unlike the case of equivalences, quasiequivalences may fail to be pointwise in general.

\subsection{Transfer of laxfunctor structure.}

\begin{definition} \label{transferdatum}
A transfer datum $F \Mrl{\theta} G$, in any of the four categories in diagram (\ref{eq:42-cat}), is a  pseudo transformation such that  
\mbox{$\cc{D}(W,FX) \Xr{{(\theta_X \!)}_{\!*}} \cc{D}(W,GX)$} 
is full and \mbox{faithful} for each $W$ in the image of $F$, and all $X$ in 
$\cc{B}$. Note that if the $\theta_X$ are pointwise quasiequivalences or equivalences, $\theta$ is a  transfer datum.
\end{definition}

\begin{theorem}[\bf{Backward transfer}] \label{thetheorem}
${}$

1. Let $F$, $G$ be preassignments and $F \Mrl{\theta} G$ be a transfer datum in $pre\cc{A}ssg$, 
if $G$ has an assignment structure, then $F$ can be uniquely furnished with an assignment structure in such a way that 
$\theta$ becomes a transfer datum in $\cc{A}ssg$.

2. Let $F$, $G$ be assignments and $F \Mrl{\theta} G$ be a transfer datum in $\cc{A}ssg$, if G has a lax or a pseudo functor structure, then $F$ can be uniquely furnished with (respectively) a lax or a pseudo functor structure, in such a way that $\theta$ becomes a pseudo transformation.

3. Remark that from 1. and 2. it follows that given preassignments $F$, $G$ \mbox{and a transfer} datum $F \Mr{\theta} G$, a lax or pseudo functor structure on $G$ can be transferred backwards to $F$. 
\end{theorem}
Before proving Theorem \ref{thetheorem}, we state an equivalent dual Theorem \ref{theoptheorem}, and we interpret these results in the context of equi-fibrations in the sense of Lack \cite{Lack2Mod}.
We let the reader check the following which holds by the very definitions involved.
\begin{proposition} \label{yoga}
Let $\cc{X}(\cc{B}, \cc{D})$ be any of the 2-categories appearing in \eqref{eq:42-cat}. Then:
 $$
 \cc{X}(\cc{B}, \cc{D})^{op} = 
 \cc{X}_{op}(\cc{B}^{op},\cc{D}^{op})
 \hspace{2ex} and \hspace{2ex}
 \cc{X}(\cc{B}, \cc{D})(F, G) = 
 \cc{X}_{op}(\cc{B}^{op},\cc{D}^{op})
                             (\overline{G}, \overline{F}),
 \hspace{2ex} where:                             
 $$ 
\hspace{3ex} 1. $\cc{X}_{op}$ denotes the respective 2-categories with optransformations instead of transformations (that is, the 2-cells $\theta_f$ go in the opposite direction, see Definition \ref{PN}).

2. Given $\cc{B} \mr{F} \cc{D}$ in 
{$\cc{X}(\cc{B}, \cc{D})$}, we denote 
$\cc{B}^{op} \mr{\overline{F}} \cc{D}^{op}$  
the corresponding object in the dual category 
$\cc{X}(\cc{B}, \cc{D})^{op}$. Given a lax transformation (optransformation)  $F \Mrl{\theta} G$, we have a lax op-transformation (transformation) $\overline{G} \Mrl{\theta} \overline{F}$
\end{proposition}

\begin{definition} \label{optransferdatum}
An optransfer datum $F \Mrl{\theta} G$ in any of the four categories in Diagram (\ref{eq:42-cat}), is a  pseudo optransformation such that  
$\cc{D}(GX,W) \Xr{{(\theta_X \!)}^{\!*}} \cc{D}(FX,W)$ 
is full and faithful for each $W$ in the image of $F$, and all $X$ in 
$\cc{B}$. Note that if the $\theta_X$ are pointwise quasiequivalences or equivalences, then $\theta$ is a  optransfer datum.
\end{definition}

\begin{proposition} \label{duality}
With the notations in Proposition \ref{yoga}, $F \Mrl{\theta} G$ is a \mbox{optransfer} datum if and only if 
$\overline{G} \Mrl{\theta} \overline{F}$ is a transfer datum.
\end{proposition}

It follows that we have a theorem equivalent to Theorem \ref{thetheorem}:
\begin{theorem} [{\bf Forward transfer}] \label{theoptheorem}
Let $F$, $G$ be preassignments and $F \Mrl{\theta} G$ be an optransfer datum, then the structures can be transferred forward from $F$ to $G$. We let the reader specify the three items corresponding to the three items in Theorem \ref{thetheorem}.
\end{theorem}  

\vspace{1ex}

Given a transfer datum $F \Mrl{\theta} |G'|$ for any of the forgetful 2-functors in the diagram \eqref{eq:42-cat}, Theorem \ref{thetheorem} means that there is a unique 
$F'$ and $F' \Mrl{\theta'} G'$ such that $F = |F'|$ and 
$|\theta'| = \theta$. 
{This is reminiscent of the concept of fibration}, and we have a dual statement for an optransfer datum $|F'| \Mrl{\theta} G$ {and opfibrations}.

\vspace{1ex}

In fact, since equivalences in the four 2-categories in the diagram \eqref{eq:42-cat} are both transfer and optransfer data, and recalling item 3. in \ref{equivalencias}, from Theorems \ref{thetheorem} and \ref{theoptheorem} it follows:

\begin{theorem} \label{fibrations}
The forgetful 2-functors in the diagram \eqref{eq:42-cat} are equi-fibrations and equi-opfibrations 
$\cc{X} \mr{} \cc{Z}$ in the sense of Lack
\cite[\S 2, p.188]{Lack2Mod} (see also
\cite{nL}). Note that the second condition in the definition of equi-fibration holds automatically in this case, because the forgetful 2-functors are fully faithful on 2-cells.
\end{theorem} 

\begin{remark}
Given two bicategories $\cc{X}, \;\cc{Z}$ with  distinguished classes of arrows $\Psi \subset \cc{X}$, $\Omega \subset \cc{Z}$, there is a natural definition of 
$(\Omega$,$\Psi)$-fibration (and 
$(\Omega$,$\Psi)$-opfibration), 
{in which arrows of $\Omega$ are lifted to arrows of $\Psi$,
equivalent to} the notion of $equi$-$fibration$ when $\Psi = \Omega = Equivalences$.  
\end{remark} 

{With this definition, the considerations above Theorem \ref{fibrations} show that} the forgetful \mbox{2-functors} in \eqref{eq:42-cat} are $(\Omega$,$\Psi)$-fibrations (resp. opfibrations) for the classes 
$\Omega$ and $\Psi$ of transfer data (resp. optransfer data).
Finally, pointwise quasi-equivalences are also transfer and optransfer data, so Theorem \ref{fibrations} also holds for 
$(\Omega$,$\Psi)$-fibration and opfibrations with $\Psi = \Omega = Pointwise \; quasiequivalences$.
\vspace{2ex}

\noindent
{\bf Proof of Theorem \ref{thetheorem}}. 

1.  we define a functor
$
\cc{B}(X, Y) \mr{\widetilde{G}} \cc{D}(FX,\, GY)
$
by the formulas, for $X \mrr{f}{g} Y$

\vspace{1ex}

\noindent and  $f \Mr{\alpha} g$ in $\cc{B}$:
$
\; \widetilde{G}{f} = \theta_Y \circ Ff, \;\;
\widetilde{G}{g} = \theta_Y \circ Fg \;\; 
\; and \;\;\;  
\widetilde{G}\alpha = 
\theta_g \circ (G\alpha \ast \theta_X) \circ \theta_f^{-1}.
$ 
The functoriality of 
$\widetilde{G}$ follows from that of 
${G}$. 
Consider the diagram
\begin{equation} \label{Gtilde}
\xymatrix
   {
    & \cc{D}(FX, FY) \ar[rd]^{(\theta_Y)_\ast} 
   \\
    \cc{B}(X, Y) \ar[rr]^{\widetilde{G}}
                 \ar@{-->}[ru]^F
    && \cc{D}(FX, GY)
   }
\end{equation}
Since $(\theta_Y)_\ast$ is full and faithful it follows there is a unique $\beta$ such that $(\theta_Y)_\ast(\beta) = \widetilde{G}\alpha$. This determines a factorization $F$ of 
$\widetilde{G}$ as indicated in the diagram, 
$F\alpha = \beta$. The functoriality of $F$ follows from that of $\widetilde{G}$. Finally, the equation 
 $(\theta_Y)_\ast(F\alpha) = \widetilde{G}\alpha$ is equivalent to axiom {\bf LN0} in Definition \ref{psntd}, which shows that $\theta$ becomes an assignment transformation.
 
\vspace{1ex}
 
\hspace{2ex} 2. 
\emph{Definition of the unit for $F$}. Consider
$\cc{D}(FX, FX) \xr{(\theta_X)_\ast} \cc{D}(FX, GX)$. From the right hand elevator below, we see that there is a unique 
$id_{FX} \Xr{\xi_X} F(id_X)$ such that the equation holds:
$$
\vc{\xymatrix@C=0pt                 
    {  
     \theta_X \eq{0} & {} \cl{-4}{\xi}
   \\  
     \theta_X & Fid_X
   }}
\hs{10} \meq{(\xi)} \hs{10}
\vc{\xymatrix@C=0pt                
    {  
     {} \cl{-4}{\xi} & \theta_X \eq{0}
  \\  
     Gid_X \sd{-6}{6} \na{16}{\theta_{id_X}} 
    & \theta_X \sd{0}{-6}
  \\  
     \theta_X & Fid_X
   }}
$$
This equation is clearly equivalent to axiom {\bf LN1} in Definition \ref{PN}, and if the structural 2-cell $\xi$ of $G$ is invertible then so will be the one of $F$.  

\vspace{1ex}

\emph{Definition of the multiplication for $F$}. Consider 
$\cc{D}(FX, FZ) \xr{(\theta_Z)_\ast} \cc{D}(FX, GZ)$ and the right side elevator below. It follows there is a unique
$Fg \ast Ff \Xr{\phi_{f,g}} F(g \ast f)$ such that the equation holds: 
$$
\vc{\xymatrix@C=0pt              
    { 
   \theta_Z \eq{0} 
   & Fg \sd{-2}{6} \na{12}{\phi}
   & Ff \sd{2}{-6}
  \\  
   \theta_Z 
   & {}
   & \hs{-20} F(gf) 
  }}
\hs{10} \meq{(\phi)} \hs{10}
\vc{\xymatrix@C=0pt
    { 
     \theta_Z \sd{-2}{4} \na{12}{\theta_g^{-1}}
   & Fg \sd{2}{-4} 
   & Ff \eq{0}
  \\  
     Gg \eq{0}
   & \theta_Y \sd{-2}{4} \na{12}{\theta_f^{-1}}
   & Ff \sd{2}{-4}
  \\  
     Gg \sd{0}{6} \na{12}{\phi}
   & Gf \sd{2}{-4}
   & \theta_X \eq{0}
 \\  
     {}
   & \hs{-15} G(gf) \sd{-13}{6}
   & \theta_X \sd{0}{-6} \na{-17}{\theta_{gf}}
 \\  
     {}
   & \hs{-15} \theta_Z 
   & \hs{-5} F(gf)
   }}
$$
This equation is clearly equivalent to axiom {\bf LN2} in Definition \ref{PN}, and if the structural 2-cell $\phi$ of $G$ is invertible, so is that of $F$.

\vspace{1ex}

We have shown that there is a unique possible way of furnishing $F$ with a unit and a multiplication in such a way that the arrows 
$\theta_X$ and the 2-cells $\theta_f$ form a pseudo transformation.
Now we will prove that this is indeed a lax functor structure for $F$, that is, we will show the axioms in Definition \ref{PF}. We will prove that the equations of 
2-cells hold after composing with an arrow of the form 
$\theta_X$, and use the faithfulness of $(\theta_X)_\ast$.

\vspace{1ex}

\emph{Proof of Axiom LF1}. 
Consider
$\cc{D}(FX, FY) \xr{(\theta_Y)_\ast} \cc{D}(FX, GY)$,
$$
\vc{\xymatrix@C=0pt                      
    { 
     \theta_Y \eq{0}
   & Ff \eq{0}
   & {} \cl{-4}{\xi}
  \\  
     \theta_Y \eq{0} 
   & Ff \sd{-2}{4} \na{12}{\phi} 
   & Fid_X \sd{2}{-4}
  \\  
     \theta_Y
   & {}
   & \hs{-30}{Ff} 
   }}
\hs{10} \meq{(\phi)} \hs{10}             
\vc{\xymatrix@C=0pt
    { 
     \theta_Y \eq{0}
   & Ff \eq{0}
   & {} \cl{-4}{\xi}
  \\  
     \theta_Y \sd{-2}{4} \na{15}{\theta_f^{-1}}
   & Ff \sd{2}{-4} 
   & Fid_X \eq{0}
  \\  
     Gf \eq{0}
   & \theta_X \sd{-2}{4} \na{16}{\theta_{id_X}^{-1}}
   & Fid_X \sd{0}{-4}
  \\  
     Gf \sd{0}{6} \na{12}{\phi}
   & Gid_X \sd{0}{-6}
   & \theta_X \eq{0}
  \\  
     {}
   & \hs{-25} Gf \sd{-12}{8} \na{10}{\theta_f}
   & \hs{0} \theta_X \sd{-2}{-8}
  \\  
     {}
   & \hs{0} \theta_Y
   & \hs{-15} Ff
   }}
\hs{10} \meq{(el)} \hs{10}
\vc{\xymatrix@C=0pt                      
    { 
      \theta_Y \sd{-2}{4} \na{15}{\theta_f^{-1}} 
    & Ff \sd{2}{-4}
    & {} \eq{0}
   \\  
      Gf \eq{0}
    & \theta_X \eq{0}
    & {} \cl{-4}{\xi}
   \\  
         Gf \eq{0}
   & \theta_X \sd{-2}{4} \na{16}{\theta_{id_X}^{-1}}
   & Fid_X \sd{0}{-4}
  \\  
     Gf \sd{0}{6} \na{12}{\xi}
   & Gid_X \sd{0}{-6}
   & \theta_X \eq{0}
  \\  
     {}
   & \hs{-25} Gf \sd{-12}{8} \na{10}{\theta_f}
   & \hs{0} \theta_X \sd{-2}{-8}
  \\  
     {}
   & \hs{0} \theta_Y
   & \hs{-15} Ff
   }}  
\hs{10} \meq{(\xi)} 
$$
$$
\meq{(\xi)} \hs{10}
\vc{\xymatrix@C=0pt                      
    { 
     {}  
   & \cl{18}{\theta_f^{-1}} \hs{-35}\theta_Y 
   & \hs{-10} Ff
  \\  
     Gf \eq{0} \sd{20}{-5} \na{27}{\xi} 
   & {} 
   & \theta_X \eq{0} \sd{-20}{5}
  \\ 
     Gf \eq{0} 
   & Gid_X \sd{-4}{6} \na{17}{\theta_{id_X}}
   & \theta_X \sd{0}{-6}
  \\ 
     Gf \eq{0}
   & \theta_X \sd{-4}{6} \na{17}{\theta_{id_X^{-1}}}
   & Fid_X \sd{0}{-6}
  \\ 
     Gf \sd{0}{6} \na{15}{\xi}
   & Gid_X \sd{2}{-6}
   & \theta_X \eq{0}
  \\ 
     {}
   & \hs{-20} Gf \sd{-10}{6} \na{13}{\theta_f}
   & \theta_X \sd{0}{-6}
  \\ 
     {}
   & \hs{-20} \theta_Y
   & Ff 
   }}
\hs{10} \meq{} \hs{10}
\vc{\xymatrix@C=0pt                      
    { 
     {}  
   & \cl{18}{\theta_f^{-1}} \hs{-35}\theta_Y 
   & \hs{-10} Ff
  \\  
     Gf \eq{0} \sd{20}{-5} \na{27}{\xi} 
   & {} 
   & \theta_X \eq{0} \sd{-20}{5}
  \\ 
     Gf \sd{0}{6} \na{15}{\xi}
   & Gid_X \sd{2}{-6}
   & \theta_X \eq{0}
  \\ 
     {}
   & \hs{-20} Gf \sd{-10}{6} \na{10}{\theta_f}
   & \theta_X \sd{0}{-6}
  \\ 
     {}
   & \hs{-20} \theta_Y
   & Ff 
   }}
\hs{10} \meq{LF1} \hs{10}
\vc{\xymatrix@C=0pt                      
    { 
     \theta_Y \sd{-2}{4}\na{12}{\theta_f^{-1}}  
   & Ff \sd{2}{-4}
  \\ 
     Gf \eq{0}
   & \theta_X \eq{0}
  \\ 
     Gf \sd{-2}{4} \na{12}{\theta_f}
   & \theta_X \sd{2}{-4}
  \\ 
    \theta_Y
  & Ff  
   }}
\hs{10} \meq{} \hs{10}
\vc{\xymatrix@C=0pt                      
    { 
     \theta_Y \eq{0} & Ff \eq{0}
    \\
     \theta_Y & Ff 
   }}                     
$$ 

\emph{Proof of Axiom LF2}. The proof is analogous to the one of LF1.

\vspace{1ex}

\emph{Proof of Axiom LF3}.
Consider 
$\cc{D}(FX, FW) \xr{(\theta_W)_\ast} \cc{D}(FX, GW)$,
$$
\vc{\xymatrix@C=-2pt                      
    {  
     \theta_W \eq{0} 
   & Fh \eq{0}  
   & Fg \sd{-2}{8} \na{13}{\phi}
   & Ff \sd{2}{-8}
  \\  
    \theta_W \eq{0}
   & Fh \sd{2}{10}   
   & {} \na{-2}{\phi} 
   & \hs{-27} F(gf) \sd{-8}{-10} 
  \\  
    \theta_W  
   & {}
   & \hs{-2} F(hgf)
   & {}      
   }}
\hs{10} \meq{(\phi)} \hs{10}           
\vc{\xymatrix@C=-2pt
    {  
     \theta_W \eq{0} 
   & Fh \eq{0}  
   & Fg \sd{-2}{8} \na{14}{\phi}
   & Ff \sd{2}{-6}
 \\  
    \theta_W \sd{-2}{4} \na{14}{\theta_h^{-1}} 
   & Fh \sd{2}{-4}   
   & {} 
   & \hs{-20} F(gf) \eq{-10} 
 \\  
     Gh \eq{0}
   & \theta_Z \sd{-2}{8} \na{23}{\theta_{gf}^{-1}}
   & {}
   & \hs{-20} F(gf) \sd{-8}{-8}
 \\  
     Gh \sd{0}{12}%
   & {} \na{-4}{\phi}
   & \hs{-27} G(gf) \sd{-14}{-12}   & \hs{-27} \theta_X  \eq{-18} 
 \\  
     {} 
   & G(hgf) \sd{-4}{6}
   & {} \na{-10}{\theta_{hgf}}
   & \hs{-27} \theta_X \sd{-16}{-6}
 \\  
    {}                  
  & \hs{8} \theta_W     
  & \hs{-4} F(hgf)        
  & {}                  
   }}
\hs{10} \meq{(el)} \hs{10}               
\vc{\xymatrix@C=-2pt                      
    {  
     \theta_W \sd{-2}{4} \na{14}{\theta_h^{-1}}
   & Fh \sd{2}{-4}
   & Fg \eq{0}
   & Ff \eq{0}
 \\  
     Gh \eq{0}
   & \theta_Z \eq{0}       
   & Fg \sd{-2}{8} \na{14}{\phi}
   & Ff \sd{2}{-6}
 \\  
     Gh \eq{0}
   & \theta_Z \sd{-2}{8} \na{23}{\theta_{gf}^{-1}}
   & {}
   & \hs{-20} F(gf) \sd{-8}{-8}
 \\       					
     Gh \sd{0}{12}
   & {} \na{-4}{\phi}
   & \hs{-27} G(gf) \sd{-14}{-12}   & \hs{-27} \theta_X  \eq{-18} 
 \\  
     {} 
   & G(hgf) \sd{-4}{6}
   & {} \na{-10}{\theta_{hgf}}
   & \hs{-27} \theta_X \sd{-16}{-6}
 \\  
    {}
  & \hs{8} \theta_W
  & \hs{-4} F(hgf)
  & {}
   }}
\hs{10} \meq{(\phi)}                 
$$
$$
\meq{(\phi)} \hs{10}                 
\vc{\xymatrix@R=1.5pc@C=-2pt                   
    { 
     \theta_W \sd{-2}{4} \na{14}{\theta_h^{-1}}
   & Fh \sd{2}{-4}
   & Fg \eq{0}
   & Ff \eq{0}
 \\  
     Gh \eq{0}
   & \theta_Z \sd{-2}{4} \na{14}{\theta_g^{-1}}       
   & Fg \sd{-2}{-4} 
   & Ff \eq{0} 
 \\  
     Gh \eq{0}
   & Gg \eq{0}
   & \theta_Y \sd{-2}{4} \na{14}{\theta_f^{-1}}
   & Ff \sd{2}{-4}
 \\  
     Gh \eq{0}
   & Gg \sd{2}{8} \na{16}{\phi}
   & Gf \sd{2}{-6}
   & \theta_X \eq{0}
 \\  
     Gh \eq{0}
   & {}
   & \hs{-20} G(gf) \sd{-13}{6}
   & \theta_X \sd{0}{-6} \na{-18}{\theta_{gf}}
 \\  
     Gh \eq{0}
   & {}
   & \hs{-15} \theta_Z \sd{-8}{4} \na{8}{\theta_{gf}^{-1}}
   & \hs{-5} F(gf)  \sd{0}{-4}
 \\  
     Gh \sd{0}{12}
   & {} \na{-4}{\phi}
   & \hs{-20} G(gf) \sd{-14}{-12}   
   & \hs{-20} \theta_X  \eq{-10} 
 \\  
     {} 
   & G(hgf) \sd{-4}{6}
   & {} \na{-10}{\theta_{hgf}}
   & \hs{-20} \theta_X \sd{-12}{-6}
 \\  
    {}
  & \hs{8} \theta_W
  & \hs{-4} F(hgf)  
  & {}
   }}
\hs{10} \meq{} \hs{10}                 
\vc{\xymatrix@C=-2pt
    { 
     \theta_W \sd{-2}{4} \na{14}{\theta_h^{-1}}
   & Fh \sd{2}{-4}
   & Fg \eq{0}
   & Ff \eq{0}
 \\  
     Gh \eq{0}
   & \theta_Z \sd{-2}{4} \na{14}{\theta_g^{-1}}       
   & Fg \sd{-2}{-4} 
   & Ff \eq{0} 
 \\  
     Gh \eq{0}
   & Gg \eq{0}
   & \theta_Y \sd{-2}{4} \na{14}{\theta_f^{-1}}
   & Ff \sd{2}{-4}
 \\  
     Gh \eq{0}
   & Gg \sd{0}{6} \na{16}{\phi}
   & Gf \sd{2}{-4}
   & \hs{-4}\theta_X \eq{-4}
 \\  
     Gh \sd{0}{12}
   & {} \na{-4}{\phi}
   & \hs{-25} G(gf) \sd{-14}{-12}   
   & \hs{-4} \theta_X  \eq{-4} 
 \\  
     {} 
   & \hs{-5} G(hgf) \sd{-4}{6}
   & {} \na{-10}{\theta_{hgf}}
   & \hs{-4} \theta_X \sd{-6}{-8}
 \\  
    {}
  & \hs{8} \theta_W
  & \hs{-4} F(hgf)  
  & {}
   }}
\hs{10} \meq{LF3} \hs{10}          
\vc{\xymatrix@C=0pt
    { 
     \theta_W \sd{-2}{4} \na{14}{\theta_h^{-1}}
   & Fh \sd{2}{-4}
   & Fg \eq{0}
   & Ff \eq{0}
 \\  
     Gh \eq{0}
   & \theta_Z \sd{-2}{4} \na{14}{\theta_g^{-1}}       
   & Fg \sd{-2}{-4} 
   & Ff \eq{0} 
 \\  
     Gh \eq{0}
   & Gg \eq{0}
   & \theta_Y \sd{-2}{4} \na{14}{\theta_f^{-1}}
   & Ff \sd{2}{-4}
 \\  
     Gh \sd{0}{6} \na{12}{\phi}
   & Gg \sd{0}{-6} 
   & \hs{-10}Gf \eq{-4}
   & \hs{-4}\theta_X \eq{-2}
 \\  
     {} 
   & \hs{-25} G(hg) \sd{-16}{6} \na{6}{\phi}
   & \hs{-10} Gf \sd{0}{-6}   
   & \hs{-4} \theta_X \eq{-2}  
 \\  
     {} 
   & {}
   & \hs{-40} G(hgf) \sd{-28}{8}\na{-4}{\theta_{hgf}}
   & \hs{-4} \theta_X \sd{-4}{-8}
 \\  
    {}
  & \hs{8} \theta_W
  & {} 
  & \hs{-30} F(hgf)
   }}
\hs{10} \meq{(el)}          
$$
$$
\meq{(el)} \hs{10}          
\vc{\xymatrix@C=0pt       
    { 
     \theta_W \sd{-2}{4} \na{14}{\theta_h^{-1}}
   & Fh \sd{2}{-4}
   & Fg \eq{0}
   & Ff \eq{0}
 \\  
     Gh \eq{0}
   & \theta_Z \sd{-2}{4} \na{12}{\theta_g^{-1}}       
   & Fg \sd{-2}{-4} 
   & Ff \eq{0} 
 \\  
     Gh \sd{0}{6} \na{14}{\phi}
   & Gg \sd{2}{-6}
   & \theta_Y  \eq{0} 
   & Ff \eq{0} 
 \\  
     {} 
   & \hs{-25} G(hg) \eq{-10}  
   & \hs{0}\theta_Y  \sd{-2}{4} \na{14}{\theta_f^{-1}}
   & \hs{-4} Ff \sd{2}{-4}
 \\  
     {} 
   & \hs{-25} G(hg) \sd{-16}{6} \na{6}{\phi}
   & \hs{-10} Gf \sd{0}{-6}   
   & \hs{-4} \theta_X \eq{-2}  
 \\  
     {} 
   & {}
   & \hs{-40} G(hgf) \sd{-28}{8}\na{-4}{\theta_{hgf}}
   & \hs{-4} \theta_X \sd{-4}{-8}
 \\  
    {}
  & \hs{8} \theta_W
  & {} 
  & \hs{-30} F(hgf)
   }}
\hs{10} \meq{} \hs{10}          
\vc{\xymatrix@R=1.5pc@C=0pt                 
    { 
     \theta_W \sd{-2}{4} \na{14}{\theta_h^{-1}}
   & Fh \sd{2}{-4}
   & Fg \eq{0}
   & Ff \eq{0}
 \\  
     Gh \eq{0}
   & \theta_Z \sd{-2}{4} \na{14}{\theta_g^{-1}}       
   & Fg \sd{0}{-4} 
   & Ff \eq{0} 
 \\  
     Gh \sd{0}{6} \na{14}{\phi}
   & Gg \sd{2}{-6}
   & \theta_Y \eq{0} 
   & Ff \eq{0}
 \\  
     {}  \na{28}{\theta_{hg}}
   & \hs{-20} G(hg) \sd{-15}{6} 
   & \hs{0} \theta_Y \sd{0}{-6}
   & \hs{-4} Ff \eq{0}
 \\  
     {}  \na{30}{\theta_{hg}^{-1}}
   & \hs{-20} \theta_W  \sd{-6}{-6} 
   & \hs{-10} F(hg) \sd{-8}{6}   
   & \hs{-4} Ff \eq{0} 
 \\  
     {} 
   & \hs{-20} G(hg) \eq{-10}
   & \hs{0} \theta_Y \sd{-2}{4} \na{12}{\theta_f^{-1}}
   & \hs{-4} Ff \sd{0}{-4} 
 \\  
     {} \na{32}{\phi}
   & \hs{-20} G(hg) \sd{-10}{6} 
   & \hs{0} Gf \sd{-2}{-6}
   & \hs{-4} \theta_X \eq{-4}     
 \\  
     {} 
   & {}
   & \hs{-40} G(hgf) \sd{-28}{8}\na{-4}{\theta_{hgf}}
   & \hs{-4} \theta_X \sd{-4}{-8}
 \\  
    {}
  & \hs{8} \theta_W
  & {} 
  & \hs{-30} F(hgf)
  }}
\hs{10} \meq{(\phi)} \hs{10}      
\vc{\xymatrix@C=0pt                    
    { 
    \theta_W \eq{0}
  & Fh \sd{-2}{6} \na{12}{\phi}
  & Fg \sd{2}{-6} 
  & \hs{0}Ff \eq{-4}
 \\  
    \theta_W \eq{0}
  &  {}  
  & \hs{-25} F(hg) \sd{-14}{6} \na{4}{\phi}
  & \hs{-10} Ff \sd{-2}{-6}   
 \\  
    \theta_W 
  & {} 
  & {}
  & \hs{-40} F(hgf) 
  }} 
$$

\emph{Proof of the naturality $N\phi$.}  
Consider 
$\cc{D}(FX, FZ) \xr{(\theta_Z)_\ast} \cc{D}(FX, GZ)$,
$$
\vc{\xymatrix@C=0pt                 
    { 
     \theta_Z \eq{0}
   & Fg \sd{-2}{4} \na{10}{\phi}
   & Ff \sd{2}{-4}
  \\  
     \theta_Z \eq{0}
   & {}
   & \hs{-20} F(gf) \sd{-26}{4} \sd{6}{-4} 
                  \na{-10}{F(\beta\alpha)}
  \\  
     \theta_Z
   & {}
   & \hs{-20} F(ts)
   }}
\hs{10} \meq{(\phi)} \hs{10}      
\vc{\xymatrix@C=0pt
    { 
     \theta_Z \sd{-2}{4} \na{12}{\theta_g^{-1}}
   & Fg \sd{2}{-4} 
   & Ff \eq{0}
  \\  
     Gg \eq{0}
   & \theta_Y \sd{-2}{4} \na{12}{\theta_f^{-1}}
   & Ff \sd{2}{-4}
  \\  
     Gg \sd{0}{6} \na{12}{\phi}
   & Gf \sd{2}{-4}
   & \theta_X \eq{0}
 \\  
     {}
   & \hs{-15} G(gf) \sd{-13}{6}
   & \theta_X \sd{0}{-6} \na{-15}{\theta_{gf}}
 \\  
     {} 
   & \hs{-22} \theta_Z \eq{-10}
   & \hs{-20} F(gf) \sd{-26}{4} \sd{6}{-4} 
                  \na{-10}{F(\beta\alpha)}
  \\  
     {} 
   & {} \hs{-22} \theta_Z
   & \hs{-20} F(ts)
   }}
\hs{10} \meq{LN0} \hs{10}      
\vc{\xymatrix@C=0pt
    { 
     \theta_Z \sd{-2}{4} \na{12}{\theta_g^{-1}}
   & Fg \sd{2}{-4} 
   & Ff \eq{0}
  \\  
     Gg \eq{0}
   & \theta_Y \sd{-2}{4} \na{12}{\theta_f^{-1}}
   & Ff \sd{2}{-4}
  \\  
     Gg \sd{0}{6} \na{12}{\phi}
   & Gf \sd{2}{-4}
   & \theta_X \eq{0}
  \\  
     {}
   & \hs{-15} G(gf) \sd{-24}{4} \sd{9}{-4} 
                  \na{-7}{G(\beta\alpha)}
   & \theta_X \eq{0}
  \\  
     {}
   & \hs{-15} G(ts) \sd{-15}{4} \na{5}{\theta_{ts}}
   & \theta_X  \sd{2}{-4}
  \\  
     {} 
   & {} \hs{-20} \theta_Z
   & \hs{-15} F(ts)
  }}
\hs{10} \meq{N\phi} \hs{10}      
\vc{\xymatrix@C=0pt
    { 
     \theta_Z \sd{-2}{4} \na{12}{\theta_g^{-1}}
   & Fg \sd{2}{-4} 
   & Ff \eq{0}
  \\  
     Gg \eq{0}
   & \theta_Y \sd{-2}{4} \na{12}{\theta_f^{-1}}
   & Ff \sd{2}{-4}
  \\  
     Gg \cl{10}{G\beta}
   & Gf \cl{10}{G\alpha}
   & \theta_X \eq{0}
  \\  
    Gt \sd{0}{6} \na{12}{\phi}
  & Gs \sd{2}{-4}
  & \theta_X \eq{0}
 \\  
  {}
  & \hs{-15} G(ts) \sd{-15}{4} \na{5}{\theta{ts}}
  & \theta_X \sd{2}{-4}
 \\  
    {} 
  & {} \hs{-20} \theta_Z
  & \hs{-15} F(ts)
   }}
\hs{10} \meq{(el)} 
$$
$$
\meq{(el)} \hs{10}             
\vc{\xymatrix@C=0pt
    { 
     \theta_Z \sd{-2}{4} \na{12}{\theta_g^{-1}}
   & Fg \sd{2}{-4} 
   & Ff \eq{0}
  \\  
     Gg \cl{10}{G\beta}
   & \theta_Y \eq{0}
   & Ff \eq{0}
  \\  
     Gt \eq{0}
   & \theta_Y \sd{-2}{4} \na{12}{\theta_f^{-1}}
   & Ff \sd{2}{-4}
  \\  
     Gg \eq{0}
   & Gf \cl{10}{G\alpha}
   & \theta_X \eq{0}
  \\  
    Gt \sd{0}{6} \na{12}{\phi}
  & Gs \sd{2}{-4}
  & \theta_X \eq{0}
 \\  
  {}
  & \hs{-15} G(ts) \sd{-15}{4} \na{5}{\theta{ts}}
  & \theta_X \sd{2}{-4}
 \\  
    {} 
  & {} \hs{-20} \theta_Z
  & \hs{-15} F(ts)
   }}
\hs{10} \meq{LN0^{op}} \hs{10}             
\vc{\xymatrix@C=0pt
    { 
     \theta_Z \eq{0} 
   & Fg \cl{10}{F\beta}
   & Ff \eq{0}
  \\  
     \theta_Z \sd{-2}{4} \na{12}{\theta_t^{-1}}
   & Ft \sd{2}{-4} 
   & Ff \eq{0}
  \\  
     Gt \eq{0} 
   & \theta_Y \eq{0}
   & Ff \cl{10}{F\alpha}
  \\  
     Gt \eq{0}
   & \theta_Y \sd{-2}{4} \na{12}{\theta_s^{-1}}
   & Fs \sd{2}{-4}
  \\  
    Gt \sd{0}{6} \na{12}{\phi}
  & Gs \sd{2}{-4}
  & \theta_X \eq{0}
 \\  
  {}
  & \hs{-15} G(ts) \sd{-15}{4} \na{5}{\theta{ts}}
  & \theta_X \sd{2}{-4}
 \\  
    {} 
  & {} \hs{-20} \theta_Z
  & \hs{-15} F(ts)
   }}
\hs{10} \meq{(el)} \hs{10}             
\vc{\xymatrix@C=0pt
    { 
     \theta_Z \eq{0}
   & Fg \cl{10}{F\beta}
   & Ff \cl{10}{F\alpha}
  \\  
     \theta_Z \sd{-2}{4} \na{12}{\theta_t^{-1}}
   & Ft \sd{2}{-4} 
   & Fs \eq{0}
  \\  
     Gt \eq{0}
   & \theta_Y \sd{-2}{4} \na{12}{\theta_s^{-1}}
   & Fs \sd{2}{-4}
  \\  
    Gt \sd{0}{6} \na{12}{\phi}
  & Gs \sd{2}{-4}
  & \theta_X \eq{0}
 \\  
  {}
  & \hs{-15} G(ts) \sd{-15}{4} \na{5}{\theta{ts}}
  & \theta_X \sd{2}{-4}
 \\  
    {} 
  & {} \hs{-20} \theta_Z
  & \hs{-15} F(ts)     
   }}
\hs{10} \meq{(\phi)} \hs{10}             
\vc{\xymatrix@C=0pt
    { 
     \theta_Z \eq{0}
   & Fg \cl{10}{F\beta}
   & Ff \cl{10}{F\alpha}
  \\  
     \theta_Z \eq{0}
   & Ft \sd{-2}{4} \na{10}{\phi}
   & Fs \sd{2}{-4}
  \\  
      \theta_Z
   & {}
   & \hs{-20} F(ts)  
    }}
$$
\cqd

\end{appendix} 

\bibliographystyle{unsrt}

\end{document}